\documentclass[11pt]{amsart}
\usepackage[a4paper,margin=1in]{geometry}
\usepackage[T1]{fontenc}
\usepackage{lmodern}
\usepackage{microtype}
\usepackage{amsmath,amssymb,amsthm,mathtools}
\usepackage{mathrsfs}
\usepackage{enumitem}
\usepackage{aliascnt}
\usepackage{booktabs}
\usepackage{xcolor}
\usepackage[colorlinks=true,linkcolor=blue,citecolor=blue,urlcolor=blue]{hyperref}
\usepackage[nameinlink,capitalise,noabbrev]{cleveref}

\allowdisplaybreaks

\numberwithin{equation}{section}
\newtheorem{theorem}{Theorem}[section]

\newaliascnt{proposition}{theorem}
\newtheorem{proposition}[proposition]{Proposition}
\aliascntresetthe{proposition}

\newaliascnt{lemma}{theorem}
\newtheorem{lemma}[lemma]{Lemma}
\aliascntresetthe{lemma}

\newaliascnt{corollary}{theorem}
\newtheorem{corollary}[corollary]{Corollary}
\aliascntresetthe{corollary}

\newaliascnt{definition}{theorem}
\newtheorem{definition}[definition]{Definition}
\aliascntresetthe{definition}

\newaliascnt{remark}{theorem}
\newtheorem{remark}[remark]{Remark}
\aliascntresetthe{remark}

\crefname{theorem}{Theorem}{Theorems}
\Crefname{theorem}{Theorem}{Theorems}

\crefname{proposition}{Proposition}{Propositions}
\Crefname{proposition}{Proposition}{Propositions}

\crefname{lemma}{Lemma}{Lemmas}
\Crefname{lemma}{Lemma}{Lemmas}

\crefname{corollary}{Corollary}{Corollaries}
\Crefname{corollary}{Corollary}{Corollaries}

\crefname{definition}{Definition}{Definitions}
\Crefname{definition}{Definition}{Definitions}

\crefname{remark}{Remark}{Remarks}
\Crefname{remark}{Remark}{Remarks}

\usepackage{tikz}
\usetikzlibrary{arrows.meta,decorations.pathreplacing}

\definecolor{qred}{RGB}{190,45,55}
\definecolor{stageblue}{RGB}{76,120,168}
\definecolor{stageteal}{RGB}{47,151,142}
\definecolor{stageorange}{RGB}{242,142,43}
\definecolor{stagepurple}{RGB}{141,101,166}
\definecolor{rvblue}{RGB}{72,117,166}
\definecolor{rvdark}{RGB}{27,55,87}
\definecolor{rvorange}{RGB}{232,145,65}
\definecolor{rvshade}{RGB}{220,232,242}
\definecolor{rvprob}{RGB}{224,224,224}
\usetikzlibrary{arrows.meta,calc,decorations.pathmorphing,patterns.meta}
\hypersetup{
  pdftitle={No Gelation and Global Existence for a Boltzmann Equation
  with Regularly Varying Mass-Exchange Rates},
  pdfauthor={Siwei Luo and Jian-Guo Liu}
}

\newcommand{\R}{\mathbb R}
\newcommand{\Sph}{\mathbb S^{d-1}}

\newcommand{\one}{\mathbf 1}
\newcommand{\norm}[1]{\left\|#1\right\|}
\newcommand{\abs}[1]{\left|#1\right|}
\newcommand{\pair}[2]{\left\langle#1,#2\right\rangle}

\title[No Gelation and Global Existence for BME with RV Mass-Exchange Rates]{No Gelation and Global Existence for a Boltzmann Equation with Regularly Varying Mass-Exchange Rates}
\author{Siwei Luo}
\address{School of the Gifted Young, University of Science and Technology of China,
Hefei 230026, China}
\email{luosw@mail.ustc.edu.cn}

\author{Jian-Guo Liu}
\address{Departments of Mathematics and Physics, Duke University,
Durham, North Carolina 27708, USA}
\email{jian-guo.liu@duke.edu}
\date{July 2026}
\begin{document}
\begin{abstract}
We develop a no-gelation mechanism that yields global solutions to the spatially homogeneous Boltzmann equation with mass exchange for Grad-cutoff hard potentials $B=E^\gamma b(\xi), 0<\gamma<1$ and regularly varying mass-exchange rates. 
Without assuming any higher mass moment, we construct a convex superlinear weight assembled from dyadic hinges. Production of the weighted moment by collisions between large particles of comparable mass is absorbed by dissipation through collisions with a uniformly populated reservoir of bounded-mass particles. Regular variation makes the signed increments in these two collision configurations comparable, while mass conservation and $\gamma<1$ yield the vanishing factor $L^{\gamma-1}$. This yields a uniform moment bound of the mass-cutoff approximations on every finite time interval and rules out finite-time mass escape to infinity. Consequently, for every nonnegative initial density with finite physical moments, we obtain a global integral weak solution.
\end{abstract}
\maketitle\tableofcontents
\section{Introduction}
Size redistribution driven by mass transfer via aggregation or
collision is observed in systems such as branched polymerization,
aerosol and colloid dynamics, and droplet collisions
\cite{Flory1941,Friedlander2000,AshgrizPoo1990}.  A central question in the resulting
size-structured kinetic equations is whether repeated interactions can
transport a positive part of the total mass across an unbounded sequence
of size scales within a finite time.  This phenomenon is called
{\bf \emph{gelation}} in coagulation theory.  Mathematically, it appears as a
loss of the first mass moment carried by the finite-size distribution
\cite{Aldous1999,Leyvraz2003,BanasiakLambLaurencot2019}.
The gelation process does not contradict the exact conservation of mass
in each elementary interaction, because local conservation laws alone
cannot control the total mass flux moving toward \(m=\infty\).
Therefore, ruling out gelation requires a quantitative mechanism capable
of both controlling this transfer toward large masses and ensuring the
conservation of the first moment at every finite time.  Finding such a
mechanism for the Boltzmann equation with mass exchange, under physically
natural exchange rates that may grow up to the linear scale, is the main
purpose of this paper.

The standard example of this issue is the Smoluchowski coagulation
equation.  In this process, each binary event merges two clusters into a
single cluster with a mass equal to the sum of the two, thereby generating
a one-way flux toward larger sizes.  The growth of the coagulation kernel
determines whether this flux is compatible with the conservation of the
first moment: for a broad class of homogeneous kernels of degree greater
than \(1\), finite-time gelation occurs, whereas appropriate subcritical growth
conditions can yield mass-conserving solutions
\cite{EscobedoMischlerPerthame2002,
EscobedoLaurencotMischlerPerthame2003,
BanasiakLambLaurencot2019}.  In coagulation--fragmentation systems,
fragmentation introduces a competing flux toward smaller sizes, and its
strength relative to coagulation can lead either to global mass
conservation or to persistent gelation
\cite{EscobedoMischlerPerthame2002,
EscobedoLaurencotMischlerPerthame2003,
LaurencotVanRoessel2015}.  Exchange-driven growth is also related to the
present problem, in which each interaction transfers mass between two
clusters and can move mass in either direction.  Nevertheless,
sufficiently rapid exchange can still drive mass to infinity.  Recent
results identify regimes of ordinary growth, finite-time gelation, and
instantaneous gelation
\cite{BenNaimKrapivsky2003,SiGiri2025}.

The Boltzmann equation with mass exchange (BME), introduced by Degond
and Liu in a discrete-mass setting \cite{DegondLiu2025} and later studied
directly for continuous masses in \cite{LuoLiu2026}, places this
redistribution mechanism in a coupled mass--velocity phase space.  An
incoming pair \((m,v)\), \((m_1,v_1)\), with total mass \(S=m+m_1\), is
replaced by two outgoing particles with masses \(\alpha S\) and
\((1-\alpha)S\), while the post-collisional velocities are chosen so that
total mass, momentum, and kinetic energy are conserved in every
collision.  The event has the \(2\to2\) structure and therefore also
preserves particle number.  At the opposite boundary \(m=0\),
fragmentation processes may exhibit {\bf\emph{shattering}}, in which a
positive amount of mass is transferred to particles of vanishing size
in finite time.  For the continuous-mass BME, we ruled out this
phenomenon in
\cite[Proposition~4.2 and Lemma~8.5]{LuoLiu2026}.
The present work concerns the large-mass boundary: conservation in each
collision still allows repeated redistributions to transfer mass through
successively larger scales.  Indeed, each collision compares two
partitions of the same total mass: the outgoing partition may be more
balanced or more unequal than the incoming one, so the collision
increment of a convex function that detects large masses has no fixed
sign.  Moreover, the collision frequency contains the hard-potential
factor \(E^\gamma\), where
\(E=\dfrac{mm_1}{m+m_1}|v-v_1|^2\), and hence couples the mass partition
to the relative velocity.  The no-gelation problem for the BME therefore
involves a signed flux across large mass scales with a velocity-dependent
weight, a structure that does not arise in the classical coagulation
equation or in exchange-driven models without velocity.

The available Cauchy theory reflects this difficulty.  For bounded
continuous symmetric mass-exchange rates,
\cite[Theorem~2.2]{LuoLiu2026} constructs global nonnegative
\(L^1\)-integral weak solutions from initial data with finite particle
number, total mass, and kinetic energy.  When boundedness of the
mass-exchange rate is replaced by the general linear-growth condition
\(a(m,m_1,\alpha)\lesssim 1+m+m_1\), the same work establishes a conservative local 
theory under an additional weighted-moment assumption.
More precisely, writing
\[
 H_q(f)(t)=\int_X(1+m+m|v|^2)^qf(t,x)\,dx,
\]
local existence holds for \(H_q(f_0)<\infty\) with \(q\geq1+\gamma\),
and the solution can be continued across a finite time \(T\) provided
that \(H_{1+\gamma}(f)\in L^1(0,T)\)
\cite[Theorems~8.1 and~8.8]{LuoLiu2026}.  Since the three physical
moments alone do not control this critical continuation quantity,
these estimates do not give a global theory for linearly growing rates.
This leaves open whether additional, physically natural structure of
the exchange kernel can prevent the escape of mass to infinity and yield
global mass-conserving solutions from the physical moments alone.

We resolve this question for exchange laws that separate the
activation rate of an exchange event from its outcome.
More precisely, we assume
\[
 a(m,m_1,\alpha)=\lambda(m,m_1)g(\alpha),
 \qquad S=m+m_1,
\]
where \(g\) is a continuous symmetric probability density and
\(\lambda(m,m_1)\asymp\kappa(S)\) uniformly for \(m,m_1>0\).
The function \(\lambda\) measures the effective capacity of the exchange channels
activated by an encounter, whereas \(g(\alpha)\,d\alpha\) is the
conditional law of the outgoing share
\(\alpha=m'/S\).  A concrete realization of a fixed \(g\) is a
mixing-dominated temporary-coalescence regime.  Two droplets first
form a connected liquid complex, exchange material through 
collision-induced internal motion, and then separate into two
remnants
\cite{AnidjarTambourGreenberg1995}.  If internal mixing occurs before separation, the
temporary complex may lose much of the information about the incoming
partition \(m/S\).  Its separation can then be only modeled in terms of the
dimensionless share \(\alpha\). The comparison
\(\lambda(m,m_1)\asymp\kappa(m+m_1)\) has another physical meaning.
Writing \(\theta=m/S\), it requires the mass ratio to modify the
total-size intensity only by a bounded factor.  This property holds, for
example, for the following standard collision geometries.  For compact
constant-density particles in \(d\) dimensions, \(m=c_\rho r^d\), and
the combined-radius cross section satisfies
\[
 (r+r_1)^{d-1}
 =c_\rho^{-(d-1)/d}S^{(d-1)/d}
  \bigl[\theta^{1/d}+(1-\theta)^{1/d}\bigr]^{d-1}
 \asymp S^{(d-1)/d}
\]
uniformly for \(0\leq\theta\leq1\)
\cite{DegondLiu2025}.  Likewise, if the two particles
contribute additive \(p\)-homogeneous active capacities, 
then
\[
 S^p\leq m^p+m_1^p\leq 2^{1-p}S^p,
 \qquad 0\leq p\leq1,
\]
while an activity attached directly to the temporary combined complex
is a function of its conserved total mass \(S\) itself.  We assume that
\(\kappa\) satisfies the global linear bound
\(\kappa(S)\lesssim1+S\) and is regularly varying at infinity with index
\(p\in[0,1]\).  Regular variation describes an asymptotically
self-similar activity law, where the power \(S^p\) gives the leading
growth while a slowly varying factor allows smaller corrections caused
by shielding, porosity, or slowly changing accessibility.  Experimental
mass--radius laws for colloidal aggregates and calculations of
accessible surface area, diffusion-controlled reaction rates, and
transport to fractal aerosol aggregates show such leading power-law
behavior
\cite{FilippovZuritaRosner2000}.  These physical laws motivate
the regularly varying regimes described later in
\cref{rem:model-physical-content}.

For every Grad-cutoff hard-potential kernel
\(B(E,\xi)=E^\gamma b(\xi)\), with \(0<\gamma<1\), and every
nonnegative initial datum satisfying
\[
 \int_X(1+m+m|v|^2)f_0\,dx<\infty,
\]
our main theorem constructs a global nonnegative integral weak solution
that conserves particle number and total mass and has nonincreasing
kinetic energy.  More precisely, the initial mass tail is used to
construct a continuous nonnegative convex function \(\Phi\), with
\(\Phi(m)/m\to\infty\), and the corresponding \(\Phi\)-moment remains
uniformly bounded on every finite time interval along the bounded-kernel
approximations.  This tail-adapted superlinear estimate prevents mass
from escaping through \(m=\infty\) and allows total mass to pass to the
limit.  In particular, global no-gelation holds even at the additive
linear scale without assuming any prescribed superlinear mass moment.

The proof begins with a convex superlinear mass weight built from dyadic
hinge functions:
\[
 \Phi(m)=m+\sum_{k\geq0}q_k(m-L_k)_+,
 \qquad L_k=2^kL_0.
\]
The coefficients \(q_k\) are chosen from the actual initial mass tail so
that the initial \(\Phi\)-moment is finite, while \(\Phi\) is
superlinear and the neighboring coefficients remain comparable under
fixed changes of scale.  Thus the construction imposes no predetermined
polynomial rate of growth.

The main estimate uses the sign of the collision increment of \(\Phi\).
An exact formula for each hinge function shows that collisions between
sufficiently unequal incoming masses have a strictly negative
\(g\)-averaged increment.  Positive production at large total mass is
therefore confined to comparable incoming pairs.  On each finite
time interval, the small-mass estimate and conservation of particle
number and total mass provide a fixed bounded mass interval containing
a uniformly positive number of particles.  Collisions between a large
particle and particles in this interval give the negative contribution
used to control the positive contribution from two large comparable
particles.  At a dyadic scale \(L_j\), regular variation and the
comparison \(\lambda(m,m_1)\asymp\kappa(m+m_1)\) bound the positive
large--large increment from above by
\(C\kappa(L_j)L_jq_j\), while they bound the magnitude of the negative
large--bounded increment from below by
\(c\kappa(L_j)L_jq_j\).  This common scale cancels in the integral
comparison.  The hard-potential frequency comparison contributes
\(L_j^\gamma\), whereas the mass bound on the population of the
corresponding shell contributes \(L_j^{-1}\).  Consequently, the
positive large--large contribution is bounded by a constant times
\(L_j^{\gamma-1}\) times the available negative large--bounded
contribution.  Since \(\gamma<1\), this factor tends to zero, and the
large-scale production can be controlled after summation over the
dyadic scales.  This yields a bound for the \(\Phi\)-moment that is
uniform in \(N\) and \(t\in[0,T]\), for every \(T>0\).

The limiting procedure is carried out in the space of finite Radon
measures.  Viewing the mass-cutoff approximate solutions as
\(\mu_N(t)=f_N(t,x)\,dx\), the small-mass estimate and the physical
moment bounds give tightness of the one-particle measures, while the
\(\Phi\)-estimate supplies the first-moment control needed to preserve
total mass.  Prokhorov compactness, together with the uniform
bounded-Lipschitz time estimate and the metric-valued
Arzelà--Ascoli theorem, yields a global narrowly continuous
measure-valued limit.  Localized collision-rate estimates, together
with the small- and large-mass controls, provide tightness of the
collision measures.  Localization to compact sets then identifies the
limiting collision form and removes the mass cutoff.

Since narrow convergence alone does not exclude singular limits,
absolute continuity is recovered in a separate step.  Total-variation
continuity places the singular parts of \(\mu_t\), on each finite time
interval, on a common Lebesgue-null set and the mixed gain terms are
absolutely continuous.  Testing the measure equation on this set then
shows that the singular particle number \(\sigma_t(X)\) is
nonincreasing.  It vanishes initially and hence vanishes for every
\(t\geq0\).  Thus \(\mu_t=f(t,x)\,dx\), and the collision measure admits
a jointly measurable density in
\(L^\infty(0,\infty;L^1(X))\).  The measure equation consequently
becomes the global \(L^1\)-valued Bochner integral equation and yields
the stated \(W^{1,\infty}\)-regularity in time.

The remainder of the paper is organized as follows.
\Cref{sec:model} introduces the continuous-mass collision geometry,
formulates the regularly varying mass-exchange assumptions and their
physical interpretation, and states the main theorem.
\Cref{sec:mass-cutoff-approximation} constructs the mass-cutoff
approximate solutions, establishes the uniform collision-rate and
time-regularity estimates, and produces a fixed bounded mass interval
containing a positive number of particles.  \Cref{sec:no-gelation}
constructs the tail-adapted mass weight \(\Phi\), derives the signed
redistribution estimate, and controls the large-comparable contribution
by the large--bounded contribution.  Finally,
\Cref{sec:passage-limit} obtains the global measure-valued limit,
identifies the limiting collision measure, proves absolute continuity,
and recovers the \(L^1\)-valued integral formulation.

\section{Model, Assumptions, and Main Results}
\label{sec:model}
\subsection{Collision geometry with mass exchange}
\label{subsec:collision-geometry}
We first recall the collision geometry with mass exchange introduced
in \cite{DegondLiu2025,LuoLiu2026}.  Fix an integer \(d\geq1\) and set
\begin{equation}
 X=(0,\infty)_m\times\R^d_v,
 \qquad x=(m,v),\qquad x_1=(m_1,v_1),
 \qquad dx=dm\,dv.
 \label{eq:model-state-space}
\end{equation}
For an incoming pair, write
\begin{equation}
 S=m+m_1,
 \qquad
 \theta=\frac{m}{S},
 \qquad
 V=\frac{mv+m_1v_1}{S},
 \qquad
 u=v-v_1.
 \label{eq:model-incoming-coordinates}
\end{equation}
Given a redistribution share \(\alpha\in(0,1)\) and
\(\omega\in\Sph\), let
\begin{equation}
 R_\omega u=u-2\pair{u}{\omega}\omega,
 \qquad
 s=\left(\frac{\theta(1-\theta)}
 {\alpha(1-\alpha)}\right)^{1/2}.
 \label{eq:model-reflection-factor}
\end{equation}
The outgoing variables are
\begin{align}
 m'&=\alpha S,
 &v'&=V+(1-\alpha)sR_\omega u,
 \label{eq:model-first-output}\\
 m_1'&=(1-\alpha)S,
 &v_1'&=V-\alpha sR_\omega u.
 \label{eq:model-second-output}
\end{align}
They satisfy
\begin{align}
 m'+m_1'&=m+m_1,
 \label{eq:model-mass-conservation}\\
 m'v'+m_1'v_1'&=mv+m_1v_1,
 \label{eq:model-momentum-conservation}\\
 m'\abs{v'}^2+m_1'\abs{v_1'}^2
 &=m|v|^2+m_1\abs{v_1}^2.
 \label{eq:model-energy-conservation}
\end{align}
The reduced relative kinetic energy is
\begin{equation}
 E(x,x_1)=\frac{mm_1}{m+m_1}\abs{v-v_1}^2
 =S\theta(1-\theta)\abs u^2,
 \qquad
 0\leq E(x,x_1)\leq m|v|^2+m_1\abs{v_1}^2.
 \label{eq:model-reduced-energy}
\end{equation}

For a Borel function \(\psi:X\to\R\), put
\begin{equation}
 \Delta\psi
 =\psi(m',v')+\psi(m_1',v_1')
 -\psi(m,v)-\psi(m_1,v_1).
 \label{eq:model-collision-increment}
\end{equation}
Fix \(e_0\in\Sph\) once and for all, and adopt the Borel convention
\begin{equation}
 {u}/{\abs u}:=e_0
 \qquad\text{when }u=0.
 \label{eq:zero-relative-velocity-convention}
\end{equation}
Whenever the integral is absolutely convergent, define
\begin{align}
 \mathcal Q(h,h)[\psi]
 :=\frac12\int_{X^2}\int_0^1\int_{\Sph}
 &a(m,m_1,\alpha)E(x,x_1)^\gamma
 b\left(\frac{u}{\abs u}\cdot\omega\right)
 \Delta\psi
 h(x)h(x_1)\,d\omega\,d\alpha\,dx_1\,dx.
 \label{eq:model-static-collision-form}
\end{align}
The value chosen in \eqref{eq:zero-relative-velocity-convention} does
not affect any collision integral because \(E(x,x_1)^\gamma=0\) when
\(u=0\).

\subsection{Assumptions}We now state the assumptions on the collision kernel $B$, mass-exchange kernel $a$ and initial data $f_0$.

\textbf{Hard-potential collision kernel with Grad cutoff.}
\label{subsec:model-assumptions}
The angular kernel satisfies
\begin{equation}
 B(E,\xi)=E^\gamma b(\xi),
 \qquad 0<\gamma<1,
 \qquad b:[-1,1]\to[0,\infty)\ \hbox{measurable},
 \label{eq:model-hard-potential}
\end{equation}
and, for one and hence every \(e\in\Sph\),
\begin{equation}
 \int_{\Sph}b(e\cdot\omega)\,d\omega
 =:\norm{b}_{L^1(\Sph)}<\infty.
 \label{eq:model-angular-cutoff}
\end{equation}
\textbf{(RV) Regularly varying mass-exchange rate with a fixed
redistribution law.}
The mass-exchange rate is assumed to satisfy precisely
\eqref{eq:model-fixed-share}--\eqref{eq:model-regular-variation}:
\begin{align}
 &g\in C([0,1];[0,\infty)),
 \qquad g(\alpha)=g(1-\alpha),
 \qquad
 \int_0^1g(\alpha)\,d\alpha=1,
 \label{eq:model-fixed-share}\\
 &\lambda\in C((0,\infty)^2;(0,\infty)),
 \qquad \lambda(m,m_1)=\lambda(m_1,m),
 \qquad
 a(m,m_1,\alpha)=\lambda(m,m_1)g(\alpha),
 \label{eq:model-rate-disintegration}\\
 &\kappa\in C((0,\infty);(0,\infty)),
 \qquad 0<\lambda_-\leq\lambda_+<\infty,
 \qquad
 \lambda_-\kappa(S)
 \leq\lambda(m,m_1)
 \leq\lambda_+\kappa(S),
 \label{eq:model-intensity-comparison}\\
 &p\in[0,1],
 \qquad 0<\kappa_+<\infty,
 \qquad \kappa(S)\leq\kappa_+(1+S),
 \qquad
 \lim_{S\to\infty}\frac{\kappa(\zeta S)}{\kappa(S)}
 =\zeta^p
 \quad\text{for every }\zeta>0.
 \label{eq:model-regular-variation}
\end{align}
In particular,
\begin{equation}
 0\leq a(m,m_1,\alpha)
 \leq a_+(1+m+m_1),
 \qquad
 a_+=\lambda_+\kappa_+\norm g_{L^\infty(0,1)}.
 \label{eq:model-linear-upper}
\end{equation}

\textbf{Initial conditions.} For a nonnegative density \(h\) on \(X\), set
\begin{equation}
 \begin{aligned}
 M_0(h)&=\int_X h(x)\,dx,
 &M_1(h)&=\int_X m h(x)\,dx,\\
 M_2(h)&=\int_X m|v|^2h(x)\,dx,
 &P(h)&=\int_X mv h(x)\,dx.
 \end{aligned}
 \label{eq:model-physical-moments}
\end{equation}
For a finite nonnegative Radon measure \(\mu\) on \(X\), define $M_0(\mu),M_1(\mu),M_2(\mu)$, and $P(\mu)$ by the same formulas with \(h\,dx\) replaced by \(d\mu\).
The initial datum is assumed to satisfy
\begin{equation}
 f_0\in L^1\bigl(X;(1+m+m|v|^2)\,dx\bigr),
 \qquad f_0\geq0.
 \label{eq:model-initial-data}
\end{equation}

\subsection{Main results}
\label{subsec:model-main-results}

\begin{definition}[Global integral weak solution]
\label{def:model-integral-solution}
A nonnegative function \(f\) is a global integral weak solution
 with initial datum \(f_0\) if
\begin{equation}
 f\in C([0,\infty);L^1(X))
 \cap L^\infty_{\mathrm{loc}}
 (0,\infty;L^1(X;(1+m+m|v|^2)\,dx)),
 \label{eq:model-solution-space}
\end{equation}
the collision form is represented for almost every \(t>0\) by a
density \(\mathbf Q(f(t),f(t))\in L^1(X)\), satisfying
\[
 \int_X\psi(x)\mathbf Q(f(t),f(t))(x)\,dx
 =\mathcal Q(f(t),f(t))[\psi],
 \qquad \psi\in L^\infty(X),
\]
Moreover, the map
\(t\mapsto\mathbf Q(f(t),f(t))\) belongs to
\(L^1_{\mathrm{loc}}(0,\infty;L^1(X))\), and
\begin{equation}
 f(t)=f_0+\int_0^t\mathbf Q(f(s),f(s))\,ds
 \qquad\text{in }L^1(X),\quad t\geq0,
 \label{eq:model-Bochner-identity}
\end{equation}
where the integral is a Bochner integral.  Equivalently, for every
bounded Borel function \(\psi:X\to\R\) and every \(t\geq0\),
\begin{equation}
 \int_X f(t,x)\psi(x)\,dx
 =\int_X f_0(x)\psi(x)\,dx
 +\int_0^t\mathcal Q(f(s),f(s))[\psi]\,ds.
 \label{eq:model-bounded-test-identity}
\end{equation}
\end{definition}

\begin{theorem}[Global existence and no gelation]
\label{thm:model-global-no-gelation}
Assume \eqref{eq:model-hard-potential}--\eqref{eq:model-regular-variation},
and let \(f_0\) satisfy \eqref{eq:model-initial-data}.  Then there exists a global integral
weak solution \(f\geq0\) such that
\begin{equation}
 f\in W^{1,\infty}(0,\infty;L^1(X))
 \cap L^\infty
 (0,\infty;L^1(X;(1+m+m|v|^2)\,dx)),
 \label{eq:model-global-solution-class}
\end{equation}
and
\begin{equation}
 \mathbf Q(f,f)\in L^\infty(0,\infty;L^1(X)).
 \label{eq:model-global-collision-density}
\end{equation}
For every \(t\geq0\),
\begin{equation}
 M_0(f(t))=M_0(f_0),
 \qquad
 M_1(f(t))=M_1(f_0),
 \qquad
 M_2(f(t))\leq M_2(f_0).
 \label{eq:model-moment-laws}
\end{equation}

More precisely, there is a continuous, nonnegative, convex function
\(\Phi:[0,\infty)\to[0,\infty)\), depending only on \(f_0\), such
that
\begin{equation}
 \int_X\Phi(m)f_0(x)\,dx<\infty,
 \qquad
 \frac{\Phi(m)}m\longrightarrow\infty
 \quad \text{ as }m\to\infty,
 \label{eq:model-superlinear-weight}
\end{equation}
and the mass-cutoff approximate solutions \(f_N\) constructed
below satisfy, for every \(T>0\),
\begin{equation}
 \sup_{N\geq1}\sup_{0\leq t\leq T}
 \int_X\Phi(m)f_N(t,x)\,dx<\infty.
 \label{eq:model-uniform-Phi-bound}
\end{equation}
Consequently,
\begin{align}
 \lim_{R\to\infty}\sup_{N\geq1}\sup_{0\leq t\leq T}
 \int_{\{m>R\}}m f_N(t,x)\,dx&=0,
 \label{eq:model-approximate-no-gelation}\\
 \lim_{R\to\infty}\sup_{0\leq t\leq T}
 \int_{\{m>R\}}m f(t,x)\,dx&=0.
 \label{eq:model-limit-no-gelation}
\end{align}
Thus no positive mass is lost at \(m=\infty\) at any finite time.
\end{theorem}

\begin{remark}[Physical meaning and representative models covered by
\textnormal{(RV)}]
\label{rem:model-physical-content}
\normalfont
The decomposition in \eqref{eq:model-rate-disintegration} separates the
dependence of the exchange rate on the incoming masses from the
distribution of the outgoing mass share.  The factor \(E^\gamma b\) in
the collision kernel gives the dependence of the collision frequency on
relative kinetic energy and scattering direction.  Since
\(\displaystyle\int_0^1g(\alpha)\,d\alpha=1\), the factor
\(\lambda(m,m_1)\) is the effective capacity of the mass-transfer
channels activated by that encounter, whereas \(g(\alpha)\,d\alpha\) gives the
conditional distribution of the fraction of \(S=m+m_1\) carried by one
outgoing particle. 

A fixed \(g\) is a self-similar redistribution law.  Once
masses are measured as fractions of \(S\), the outgoing-share
distribution is assumed to be independent of the absolute particle size
and of the incoming share \(m/S\).  Temporary joining, mixing, and
separation of two droplets provide one possible physical interpretation,
as discussed in the Introduction.

Regular variation of \(\kappa\) has also a simple geometric interpretation.
One may write
\begin{equation}
 \kappa(S)=S^pL(S),
 \qquad
 \frac{L(\zeta S)}{L(S)}\longrightarrow1
 \quad\text{for every }\zeta>0,
 \label{eq:model-kappa-physical-scaling}
\end{equation}
where the power \(S^p\) gives the leading growth of the active region
for mass-exchange and the slowly varying factor \(L\) allows smaller
corrections caused, for example, by finite-size geometry, porosity,
partial screening, or slowly changing accessibility.  More concretely,
suppose that a particle of radius \(r\) has mass \(m=c_fr^{d_f}\), while
its accessible part has size \(r^{d_a}\ell(r)\), where
\(0\leq d_a\leq d_f\) and \(\ell\) is slowly varying.  Then its exchange
intensity is regularly varying in the mass with index \(p=d_a/d_f\).
Thus \(p\) measures how the accessible part grows relative to the
mass-bearing part. For the pure-power laws below, the sum of the intensities of two
particles has the same scale as a function of their total mass because,
for \(0\leq p\leq1\),
\begin{equation}
 (m+m_1)^p
 \leq m^p+m_1^p
 \leq2^{1-p}(m+m_1)^p.
 \label{eq:model-pair-activity-comparison}
\end{equation}
When a slowly varying correction is present, it is included in the
total-size law \(\kappa(S)=S^pL(S)\), and the two-sided comparison between
\(\lambda(m,m_1)\) and \(\kappa(S)\) is imposed separately in
\eqref{eq:model-intensity-comparison}.  This comparison allows the
incoming mass ratio to change the intensity by a bounded factor without
changing its scale in the total mass.

The following four classes illustrate the physical range of
\textnormal{(RV)}; see \cref{fig:physical-activity-scalings}.
\begin{enumerate}[label=\textnormal{(\roman*)},leftmargin=2.4em]
 \item \emph{Boundary-controlled exchange.}
 For compact particles in \(d\) dimensions, constant density gives
 \(r(m)\asymp m^{1/d}\).  The geometric hard-particle cross-sectional
 factor used in the original BME derivation is
 \begin{equation}
  \lambda_{\mathrm{hs}}(m,m_1)
  =C_d\bigl(m^{1/d}+m_1^{1/d}\bigr)^{d-1}
  \asymp(m+m_1)^{(d-1)/d}.
  \label{eq:model-hard-particle-example}
 \end{equation}
 See \cite{DegondLiu2025}.  Hence it satisfies \textnormal{(RV)} with
 \(p=(d-1)/d\).  In three dimensions this is \(p=2/3\).  The same
 exponent is obtained when exchange is controlled by a boundary layer
 of fixed microscopic thickness, since the number of sites available for exchange
 is proportional to surface area.  This class describes compact
 droplets, grains, or clusters whose interiors do not take part directly
 in the exchange.

 \item \emph{Bulk-controlled exchange.}
 If collision-induced mixing reaches the whole particle, each
 elementary mass unit can contribute an exchange channel.  The number of active units is
 then proportional to \(m+m_1\), giving the bulk or additive law
 \begin{equation}
  \lambda_{\mathrm{bulk}}(m,m_1)\asymp m+m_1,
  \qquad p=1.
  \label{eq:model-bulk-example}
 \end{equation}
 Here the accessible part grows at the same rate as the total mass and this
 is the linear endpoint allowed by \textnormal{(RV)}.

 \item \emph{Exchange on a porous or fractal accessible set.}
 For a porous or fractal aggregate with mass dimension \(d_f\) and
 accessible-set dimension \(d_a\),
 \begin{equation}
  \lambda_{\mathrm{frac}}(m,m_1)
  \asymp m^{d_a/d_f}+m_1^{d_a/d_f}
  \asymp(m+m_1)^{d_a/d_f},
  \qquad p=\frac{d_a}{d_f}.
  \label{eq:model-fractal-example}
 \end{equation}
 Mass--radius laws and transport areas that depend on size are standard
 for fractal aerosol and colloidal aggregates
 \cite{FilippovZuritaRosner2000}.  If most of the
 aggregate is accessible, then \(d_a\) is close to \(d_f\) and \(p\) is
 close to one.  Strong shielding or limited pore access gives a smaller
 \(d_a\).  Gradual changes in accessibility may be included in the
 slowly varying factor \(L\), provided that the comparison in
 \eqref{eq:model-intensity-comparison} remains valid.

 \item \emph{Bounded or slowly growing exchange.}
 If only a bounded number of contact regions or slow transport paths can
 be used during one collision, the exchange intensity remains bounded
 at large sizes, which intersects the bounded-kernel setting of \cite{LuoLiu2026}.  A bounded law approaching a limit, such as
 \begin{equation}
  \kappa(S)=\kappa_\infty(1-e^{-S/S_{\rm sat}}),
  \label{eq:model-saturating-example}
 \end{equation}
 is regularly varying with \(p=0\).  The same index also includes rates
 that grow more slowly than every positive power, such as
 \(\kappa(S)=[\log(2+S)]^q\), \(q>0\).  Thus \(p=0\) includes both a
 bounded number of exchange regions and a number that grows very slowly
 with particle size.
\end{enumerate}

The four classes correspond to exchange controlled by a boundary, by
the whole mass, by a scale-dependent accessible part, and by a bounded
or slowly growing number of exchange regions.  Thus
\textnormal{(RV)} covers a broad range of physical scenarios within one assumption.
\end{remark}
\begin{figure}[t]
  \centering
  \resizebox{\textwidth}{!}{
    \input{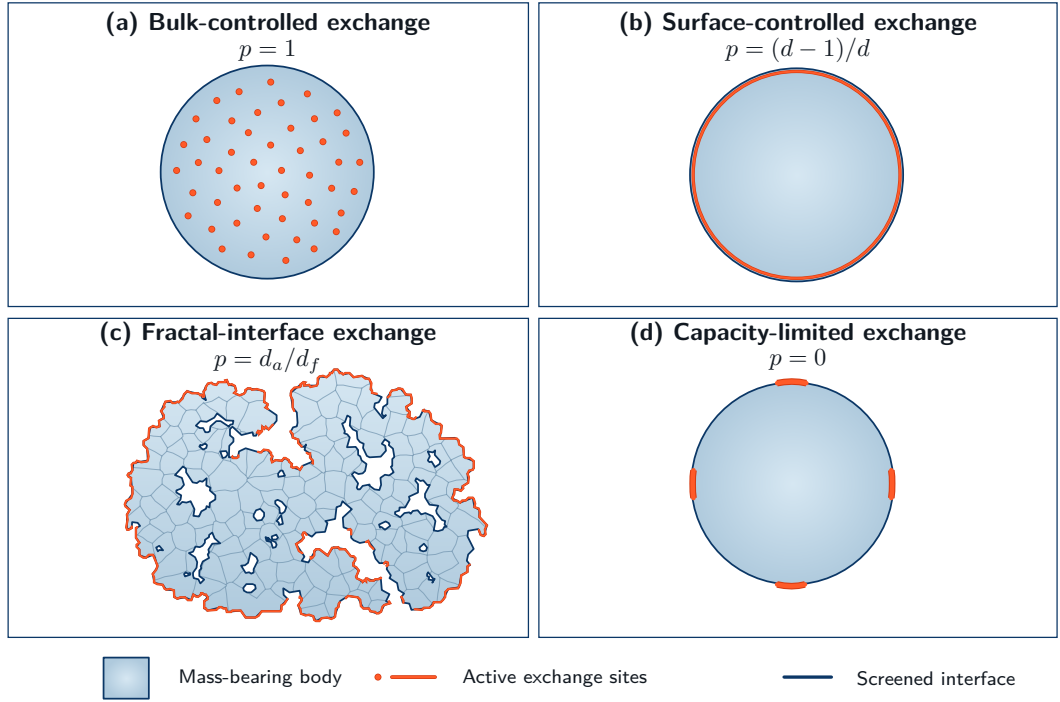}
  }
  \caption{Physical realizations of the four activity scalings under the hypothesis \textnormal{(RV)}.
  The pale-blue region is the mass-bearing body, coral marks the
  active mass-exchange sites or accessible interface, and navy marks
  screened interface when mass-exchange collisions happen.}
  \label{fig:physical-activity-scalings}
\end{figure}

\section{Mass-Cutoff Approximation Solutions}\label{sec:mass-cutoff-approximation}
\subsection{Existence} In this section, we truncate the mass-exchange rate \(a(m,m_1,\alpha)\) and retain the full
hard-potential collision kernel \(B(E,\xi)\), which places each approximate truncated equation
within the bounded mass-exchange rate theory of \cite[Theorem~2.2]{LuoLiu2026}.

\begin{lemma}[Joint mass--energy rate bound]
\label{lem:joint-rate}
Let \(h\geq0\) have finite \(M_0(h),M_1(h),M_2(h)\).  Then
\begin{align}
 \int_{X^2}S E^\gamma h(x)h(x_1)\,dx_1\,dx
 &\leq
 \bigl(2M_0(h)M_1(h)\bigr)^{1-\gamma}
 \bigl(2(M_1(h)M_2(h)-\abs{P(h)}^2)\bigr)^\gamma,
 \label{eq:weighted-rate}\\
 \int_{X^2}E^\gamma h(x)h(x_1)\,dx_1\,dx
 &\leq2^\gamma M_0(h)^{2-\gamma}M_2(h)^\gamma.
 \label{eq:unweighted-rate}
\end{align}
Consequently,
\begin{align}
 \int_{X^2}(1+S)E^\gamma h(x)h(x_1)\,dx_1\,dx
 &\leq
 \bigl(2M_0(h)M_1(h)\bigr)^{1-\gamma}
 \bigl(2(M_1(h)M_2(h)-\abs{P(h)}^2)\bigr)^\gamma
 \notag\\[-1mm]
 &\quad+2^\gamma M_0(h)^{2-\gamma}M_2(h)^\gamma.
 \label{eq:total-rate}
\end{align}
The same assertions hold when \(h(x)\,dx\) is replaced by a finite
nonnegative Borel measure with these moments.
\end{lemma}

\begin{proof}
By Cauchy--Schwarz inequality, \(\abs{P(h)}^2\leq M_1(h)M_2(h)\).  We claim that
\begin{equation}
 \begin{aligned}
 \int_{X^2}S h(x)h(x_1)\,dx_1\,dx
 &=2M_0(h)M_1(h),\\
 \int_{X^2}SE h(x)h(x_1)\,dx_1\,dx
 &=2\bigl(M_1(h)M_2(h)-\abs{P(h)}^2\bigr).
 \end{aligned}
 \label{eq:joint-rate-identities}
\end{equation}
Indeed, the first identity follows directly from $S=m+m_1$. Since
\(SE=mm_1|v-v_1|^2\), Cauchy--Schwarz gives
\begin{equation}
 \int_Xm\abs v\,h(x)\,dx
 \leq M_1(h)^{1/2}M_2(h)^{1/2}<\infty.
 \label{eq:momentum-first-moment}
\end{equation}
Consequently, the cross term is absolutely integrable, since
\begin{equation}
 \int_{X^2}mm_1\abs{v\cdot v_1}
 h(x)h(x_1)\,dx_1\,dx
 \leq
 \left(\int_Xm\abs v\,h(x)\,dx\right)^2<\infty.
\end{equation}
Fubini's theorem therefore gives
\begin{align*}
 \int_{X^2}SEh(x)h(x_1)
 &=\int_{X^2}mm_1
 \bigl(|v|^2+|v_1|^2-2v\cdot v_1\bigr)
 h(x)h(x_1)\,dx_1\,dx\\
 &=2M_1(h)M_2(h)-2|P(h)|^2,
\end{align*}
which proves the second. Using \(SE^\gamma=S^{1-\gamma}(SE)^\gamma\), H\"older's inequality on
\(X^2\) and \eqref{eq:joint-rate-identities} yield
\begin{equation*}
 \int_{X^2}SE^\gamma h(x)h(x_1)
 \leq\left(\int_{X^2}S h(x)h(x_1)\right)^{1-\gamma}
 \left(\int_{X^2}SEh(x)h(x_1)\right)^\gamma,
\end{equation*}
which is \eqref{eq:weighted-rate}.  Since
\(E\leq m|v|^2+m_1|v_1|^2\),  H\"older's
inequality gives
\begin{equation*}
 \int_{X^2}E^\gamma h(x)h(x_1)
 \leq M_0(h)^{2(1-\gamma)}
 \left(\int_{X^2}(m|v|^2+m_1|v_1|^2)h(x)h(x_1)\right)^\gamma\leq 2^\gamma M_0(h)^{2-\gamma}M_2(h)^\gamma,
\end{equation*}
which is \eqref{eq:unweighted-rate}.  Adding the two estimates proves
\eqref{eq:total-rate}.  For a finite nonnegative Borel measure, the
same Cauchy--Schwarz estimate as
\eqref{eq:momentum-first-moment} makes the signed cross term
absolutely integrable; Tonelli's theorem applies to the nonnegative
terms and Fubini's theorem to that cross term.  The preceding argument
therefore remains valid verbatim for such measures.
\end{proof}

Fix a nonincreasing function
\(\chi\in C([0,\infty);[0,1])\) such that
\begin{equation}\label{eq:chi-properties}
 \chi=1\quad\hbox{on }[0,1],
 \qquad
 \chi=0\quad\hbox{on }[2,\infty),
\end{equation}
and set
\begin{equation}
 \chi_N(S)=\chi(S/N),
 \qquad
 a_N(m,m_1,\alpha)=\chi_N(m+m_1)a(m,m_1,\alpha).
 \label{eq:mass-cutoff}
\end{equation}
For every fixed \(N\), \(a_N\) is continuous,
nonnegative, and satisfies
\begin{equation}
 a_N(m,m_1,\alpha)
 =a_N(m_1,m,\alpha)
 =a_N(m,m_1,1-\alpha).
 \label{eq:cutoff-symmetries}
\end{equation}
Moreover, \(\chi_N(S)\neq0\) implies \(S<2N\), and hence
\begin{equation}
 \norm{a_N}_{L^\infty((0,\infty)^2\times(0,1))}
 \leq a_+(1+2N)<\infty.
 \label{eq:cutoff-boundedness}
\end{equation}
Finally,
\begin{equation}
 0\leq a_N(m,m_1,\alpha)\leq a(m,m_1,\alpha)
 \leq a_+(1+m+m_1).
 \label{eq:cutoff-domination}
\end{equation}

\begin{proposition}[Global solutions of the truncated equations]
\label{prop:truncated-solutions}
Let \(\mathbf Q_N\) be the collision operator obtained from
\eqref{eq:model-static-collision-form} by replacing \(a\) with
\(a_N\). For every \(N\geq1\), the Cauchy problem
\begin{equation}\label{eq:truncated-Cauchy}
 \partial_t f_N=\mathbf Q_N(f_N,f_N),
 \qquad
 f_N(0)=f_0,
\end{equation}
has a global nonnegative \(L^1\)-integral weak solution
\begin{equation}
 f_N\in W^{1,\infty}(0,\infty;L^1(X))
 \cap L^\infty
 (0,\infty;L^1(X;(1+m+m|v|^2)\,dx)).
 \label{eq:truncated-solution-class}
\end{equation}
For all \(t\geq0\),
\begin{equation}\label{eq:truncated-moments}
 M_0(f_N(t))=M_0(f_0),
 \qquad
 M_1(f_N(t))=M_1(f_0),
 \qquad
 M_2(f_N(t))\leq M_2(f_0).
\end{equation}
For every bounded Borel test function \(\psi\), the weak identity
\begin{align}
 \int_X\psi(x)f_N(t,x)\,dx
 -\int_X\psi(x)f_0(x)\,dx
 =\int_0^t\mathcal Q_N(f_N(s),f_N(s))[\psi]\,ds
 \label{eq:truncated-weak-identity}
\end{align}
holds, where
\begin{align}
 \mathcal Q_N(h,h)[\psi]
 =\frac12\int_{X^2}\int_0^1\int_{\Sph}
 &a_N(m,m_1,\alpha)E^\gamma
 b\left(\frac{u}{\abs u}\cdot\omega\right)
 \Delta\psi\,d\omega\,d\alpha\,
 h(x)h(x_1)\,dx_1\,dx.
 \label{eq:truncated-collision-form}
\end{align}
\end{proposition}
\begin{proof}
By \eqref{eq:cutoff-symmetries}--\eqref{eq:cutoff-boundedness},
\(a_N\) is bounded, continuous, nonnegative, and has both collision
symmetries required in \cite[Theorem~2.2]{LuoLiu2026}.  That theorem
therefore gives all the stated conclusions.
\end{proof}

The following uniform estimate is needed in
\cref{sec:passage-limit}.  Set
\begin{align}
 C_{\mathrm{rate}}
 :=a_+\norm b_{L^1}
 [
 (2M_0(f_0)M_1(f_0))^{1-\gamma}
 (2M_1(f_0)M_2(f_0))^\gamma
 +2^\gamma M_0(f_0)^{2-\gamma}M_2(f_0)^\gamma
 ].
 \label{eq:Crate-definition}
\end{align}

\begin{lemma}[Uniform collision rate and time regularity]
\label{lem:uniform-rate-time}
For every \(N\geq1\) and \(s,t\geq0\),
\begin{align}
 &\int_{X^2}\int_0^1\int_{\Sph}
 a_N(m,m_1,\alpha)E^\gamma
 b\left(\frac{u}{\abs u}\cdot\omega\right)
 f_N(t,x)f_N(t,x_1)\,d\omega\,d\alpha\,dx_1\,dx
 \leq C_{\mathrm{rate}},
 \label{eq:uniform-collision-rate}\\
 &\norm{\mathbf Q_N(f_N(t),f_N(t))}_{L^1(X)}
 \leq2C_{\mathrm{rate}},
 \qquad
 \norm{f_N(t)-f_N(s)}_{L^1(X)}
 \leq2C_{\mathrm{rate}}\abs{t-s}.
 \label{eq:uniform-time-Lipschitz}
\end{align}
\end{lemma}

\begin{proof}
By \eqref{eq:cutoff-domination}, integration in \(\alpha\) and
\(\omega\), \cref{lem:joint-rate} and
\eqref{eq:truncated-moments} show that
\begin{align*}
   &\int_{X^2}\int_0^1\int_{\Sph}
 a_N(m,m_1,\alpha)E^\gamma
 b\left(\frac{u}{\abs u}\cdot\omega\right)
 f_N(t,x)f_N(t,x_1)\,d\omega\,d\alpha\,dx_1\,dx\\&\qquad \leq a_+\norm b_{L^1}
 \int_{X^2}(1+S)E^\gamma f_N(t,x)f_N(t,x_1)\,dx_1\,dx
 \leq C_{\mathrm{rate}}.
\end{align*}

For \(\norm\psi_\infty\leq1\), one has
\(\abs{\Delta\psi}\leq4\). The factor \(1/2\) in
\eqref{eq:truncated-collision-form} and
\eqref{eq:uniform-collision-rate} yield
\(\norm{\mathbf Q_N(f_N,f_N)}_{L^1}\leq2C_{\mathrm{rate}}\).
Integrating \eqref{eq:truncated-Cauchy} in \(L^1(X)\) between \(s\)
and \(t\) proves the last estimate.
\end{proof}

For later mass tests, if \(\varphi:(0,\infty)\to\R\) is bounded and
Borel, set
\begin{align}
 \mathcal A_\varphi(m,m_1)
 =\int_0^1a(m,m_1,\alpha)
 \bigl[&\varphi(\alpha S)+\varphi((1-\alpha)S)-\varphi(m)-\varphi(m_1)\bigr]\,d\alpha.
 \label{eq:A-varphi}
\end{align}
Since a mass-only test is independent of \(\omega\),
\eqref{eq:truncated-weak-identity} implies
\begin{align}
 &\int_X\varphi(m)f_N(t,x)\,dx
 -\int_X\varphi(m)f_N(s,x)\,dx
 \notag\\
 &\quad=\frac{1}{2}\norm{b}_{L^1}\int_s^t\int_{X^2}
 \chi_N(S)E^\gamma\mathcal A_\varphi(m,m_1)
 f_N(\tau,x)f_N(\tau,x_1)\,dx_1\,dx\,d\tau.
 \label{eq:mass-test-identity}
\end{align}

\subsection{Uniform bounded-mass reservoir}The no-gelation mechanism needs a uniformly positive number of
particles in a fixed bounded mass interval.  Conservation of total
mass controls the number at large mass, while the next estimate rules
out accumulation at \(m=0\).  Only the pointwise upper bound
\eqref{eq:model-linear-upper} is used in the proof.

For \(r>0\), define
\begin{equation}\label{eq:FNr-definition}
 F_N(r,t)=\int_{\{m<r\}}f_N(t,m,v)\,dm\,dv.
\end{equation}
The small-mass bootstrap in
\cite[Lemma~8.5]{LuoLiu2026} uses only the linear growth bound for the
mass-exchange rate, the bounded-Borel mass-test identity, and uniform
control of \(M_0\) and \(M_2\).  In view of
\eqref{eq:model-linear-upper}, \eqref{eq:mass-test-identity}, and
\eqref{eq:truncated-moments}, it applies to the present approximation
family and gives the following statement.

\begin{proposition}[Uniform exclusion of the zero-mass boundary]
\label{prop:small-mass}
There is a constant \(C\), independent of \(N,r,\rho,s,t\), such that
for \(0<r<\rho<1\) and \(0\leq s\leq t\),
\begin{equation}
 F_N(r,t)-F_N(r,s)
 \leq C\int_s^tF_N(\rho,\tau)^{2-\gamma}\,d\tau
 {}+C(t-s)r(1+\rho^{-1}).
 \label{eq:small-mass-bootstrap}
\end{equation}
Consequently, for every \(T>0\),
\begin{equation}\label{eq:small-mass-uniform}
 \lim_{r\downarrow0}
 \sup_{N\geq1}\sup_{0\leq t\leq T}F_N(r,t)=0.
\end{equation}
\end{proposition}

\begin{proof}
This adapts the small-mass argument of
\cite[Lemma~8.5]{LuoLiu2026} to the present mass cutoff.  We
give all details because the cited lemma is stated there inside a
higher-moment construction.

Since  \(\varphi(m)=\one_{\{m<r\}}\) is bounded and Borel, \eqref{eq:mass-test-identity} gives
\begin{align}
 F_N(r,t)-F_N(r,s)
 &=\frac{1}{2}\norm{b}_{L^1}\int_s^t\int_{X^2}
 \chi_N(S)E^\gamma\mathcal A_\varphi(m,m_1)
 f_N(\tau,x)f_N(\tau,x_1)\,dx_1\,dx\,d\tau,
 \label{eq:small-mass-identity}
\end{align}
where
\begin{equation}
 \mathcal A_\varphi(m,m_1)
 \leq\int_0^1a(m,m_1,\alpha)
 \bigl[\one_{\{\alpha S<r\}}+\one_{\{(1-\alpha)S<r\}}\bigr]\,d\alpha.
\end{equation}

\noindent\underline{\textbf{Step 1: the part with \(S<\rho\).}} On \(\{S<\rho\}\), both incoming masses are below \(\rho\) and
\(a\leq2a_+\).  Since \(0<\gamma<1\),
\(E^\gamma\leq(m|v|^2)^\gamma+(m_1|v_1|^2)^\gamma\).
H\"older's inequality yields
\begin{equation}
 \int_{\{m<\rho\}}(m|v|^2)^\gamma f_N(\tau,x)\,dx
 \leq M_2(f_0)^\gamma F_N(\rho,\tau)^{1-\gamma}.
\end{equation}
Since \(\{S<\rho\}\subset\{m<\rho,\ m_1<\rho\}\), it follows that
\begin{align*}
 &\int_{\{S<\rho\}}E^\gamma
 f_N(\tau,x)f_N(\tau,x_1)\,dx_1\,dx\leq
 2F_N(\rho,\tau)
 \int_{\{m<\rho\}}(m|v|^2)^\gamma f_N(\tau,x)\,dx\leq
 2M_2(f_0)^\gamma F_N(\rho,\tau)^{2-\gamma}.
\end{align*}
After integration in \(\alpha\) and \(\omega\), the contribution of
\(\{S<\rho\}\) is therefore bounded by
\(C F_N(\rho,\tau)^{2-\gamma}\).

\noindent\underline{\textbf{Step 2: the part with \(S\geq\rho\).}}
For a fixed \(S\), the two sets
\(\{\alpha:\alpha S<r\}\) and
\(\{\alpha:(1-\alpha)S<r\}\) have total length at most \(2r/S\).
Consequently, \eqref{eq:model-linear-upper} gives
\begin{align*}
 \int_0^1a(m,m_1,\alpha)
 \bigl[\one_{\{\alpha S<r\}}+\one_{\{(1-\alpha)S<r\}}\bigr]\,d\alpha\leq2a_+r\frac{1+S}{S}
 \leq2a_+r(1+\rho^{-1}).
\end{align*}
Hence the contribution of \(\{S\geq\rho\}\) is bounded by
\begin{align*}
 &Ca_+r(1+\rho^{-1})
 \int_{X^2}E^\gamma
 f_N(\tau,x)f_N(\tau,x_1)\,dx_1\,dx\leq
 Cr(1+\rho^{-1})
 M_0(f_0)^{2-\gamma}M_2(f_0)^\gamma
 \leq Cr(1+\rho^{-1}),
\end{align*}
where \(C\) is independent of \(N,r,\rho,\tau\).

\noindent\underline{\textbf{Step 3: the short-time bootstrap.}}
Take \(\rho=\sqrt r\).  Let
\(I=[t_0,t_0+\tau]\subset[0,T]\) and suppose
\(\displaystyle\lim_{r\downarrow0}\sup_NF_N(r,t_0)=0\).  Define
\begin{equation}
 L_I=\limsup_{r\downarrow0}\sup_{N\geq1}\sup_{t\in I}F_N(r,t).
\end{equation}
Using \eqref{eq:small-mass-bootstrap} with \(s=t_0\),
\(\rho=\sqrt r\), and then taking the supremum over
\(N\geq1\) and \(t\in I\), we obtain
\begin{align}
 \sup_{N\geq1}\sup_{t\in I}F_N(r,t)
 &\leq \sup_{N\geq1}F_N(r,t_0)
 +C\tau
 \left(\sup_{N\geq1}\sup_{t\in I}
 F_N(\sqrt r,t)\right)^{2-\gamma}
 +C\tau r(1+r^{-1/2}).
 \label{eq:small-mass-limsup-preparation}
\end{align}
The first term on the right tends to zero by the assumed starting
property, while \(r(1+r^{-1/2})\to0\).  Since \(\sqrt r\downarrow0\)
as \(r\downarrow0\), taking the limsup in
\eqref{eq:small-mass-limsup-preparation} gives
\(L_I\leq C\tau L_I^{2-\gamma}\).  If \(M_0(f_0)=0\), all solutions
vanish.  Otherwise \(0\leq L_I\leq M_0(f_0)\); choose \(\tau>0\) so
that \(C\tau M_0(f_0)^{1-\gamma}<1\).  If \(L_I>0\), division by
\(L_I\) would give
\(1\leq C\tau L_I^{1-\gamma}<1\), a contradiction.  Thus \(L_I=0\).
At \(t_0=0\), the starting property follows from \(f_0\in L^1(X)\).
The same \(\tau\) works at every later left endpoint, so finitely many
iterations cover \([0,T]\) and prove
\eqref{eq:small-mass-uniform}.
\end{proof}

The preceding exclusion of \(m=0\), together with conservation of
particle number and total mass, now supplies collision partners in a
fixed bounded mass interval at every time.

\begin{corollary}[A uniformly populated bounded-mass set]
\label{cor:bounded-mass-set}
Assume \(M_0(f_0)>0\).  For every \(T>0\), there exist
\(0<r_T<R_T<\infty\) and \(\eta_T>0\), independent of \(N\), such
that, with
\begin{equation}\label{eq:BT-definition}
 \mathcal B_T=\{(m,v)\in X:r_T\leq m\leq R_T\},
\end{equation}
one has
\begin{equation}\label{eq:bounded-mass-set}
 \inf_{N\geq1}\inf_{0\leq t\leq T}
 \int_{\mathcal B_T}f_N(t,x)\,dx\geq\eta_T.
\end{equation}
\end{corollary}
\begin{proof}
 By
\eqref{eq:small-mass-uniform}, choose \(r_T\in(0,1)\) so that
\begin{equation*}
 \sup_{N\geq1}\sup_{0\leq t\leq T}
 \int_{\{m<r_T\}}f_N(t,x)\,dx\leq\frac14M_0(f_0).
\end{equation*}
Take \(R_T=\max\{2r_T,4M_1(f_0)/M_0(f_0)\}\).  Mass conservation gives
\begin{equation*}
 \sup_{N\geq 1}\sup_{0\leq t\leq T}\int_{\{m>R_T\}}f_N(t,x)\,dx
 \leq\frac{M_1(f_0)}{R_T}
 \leq\frac14M_0(f_0),
\end{equation*}
and particle-number conservation yields
\begin{equation*}
 \int_{\mathcal B_T}f_N(t,x)\,dx
 \geq M_0(f_0)-\frac14M_0(f_0)-\frac14M_0(f_0)=\frac12M_0(f_0)  .
\end{equation*}
Thus \eqref{eq:bounded-mass-set} holds with
\(\eta_T=M_0(f_0)/2\).
\end{proof}

\section{Uniform No Gelation Mass Estimate}\label{sec:no-gelation}
For each cutoff parameter \(N\geq1\), let
\(f_N=f_N(t,x)\geq0\) denote the corresponding approximate
solution with initial datum \(f_0\).  Fix \(T>0\).  Our aim is to establish the uniform tightness of the mass measure associated with the truncated solutions. 

\begin{proposition}[No Gelation Mass Estimate]\label{prop:no-gelation}
  For every \(T>0\),
  \begin{equation}
    \lim_{R\to\infty}\sup_{N\geq1}\sup_{0\leq t\leq T}\int_{\{m>R\}}m f_N(t,x) dx=0.\label{eq:no-gelation}
  \end{equation}
\end{proposition}

The proof assigns a family of tail-adapted convex tokens to each
particle, giving a propagated weighted moment.  The first subsection computes their exact signed change
under redistribution, while the second compares the frequencies of the
collisions which can create tokens with those of collisions which
destroy them.

\subsection{Tail-adapted tokens and loss of splitting}\label{subsec:tokens}

Choose \(L_0\geq1\) so large that
\begin{equation}\label{eq:initial-tail-small}
 \mathcal T_0:=\int_{\{m>L_0\}}m f_0(x)\,dx\leq\frac12,
\end{equation}
and use dyadic levels
\begin{equation}\label{eq:levels-tails}
 L_k=2^kL_0,
 \qquad
 \mathcal T_k=\int_{\{m>L_k\}}m f_0(x)\,dx,
 \qquad k\geq0.
\end{equation}
Since \(L_{k+1}=2L_k>L_k\),
\begin{equation}
 \one_{\{m>L_{k+1}\}}\leq\one_{\{m>L_k\}},
 \qquad
 \one_{\{m>L_k\}}\longrightarrow0
 \quad \text{ as }k\to\infty
 \label{eq:initial-tail-indicators}
\end{equation}
for every \(m>0\).  Moreover,
\begin{equation}
 0\leq m f_0(x)\one_{\{m>L_k\}}
 \leq m f_0(x),
 \qquad
 \int_Xm f_0(x)\,dx=M_1(f_0)<\infty.
 \label{eq:initial-tail-dominating-function}
\end{equation}
Thus monotone set inclusion and dominated convergence give
\begin{equation}
 0\leq\mathcal T_{k+1}\leq\mathcal T_k,
 \qquad
 \lim_{k\to\infty}\mathcal T_k=0.
 \label{eq:initial-tail-monotone-limit}
\end{equation}

\begin{lemma}[Coefficients adapted to the initial mass tail]
\label{lem:q-coefficients}
There is a positive sequence \(\{q_k\}_{k\geq0}\) such that
\begin{equation}\label{eq:q-properties}
 2^{-1/2}q_k\leq q_{k+1}\leq q_k,
 \qquad
 \sum_{k=0}^\infty q_k=\infty,
 \qquad
 \sum_{k=0}^\infty q_k\mathcal T_k<\infty.
\end{equation}
\end{lemma}
\begin{proof}
Set \(c=2^{-1/2}\). We first construct a decreasing sequence of
unit-mass plateaux and then replace its downward jumps by geometric
transitions. See \cref{fig:q-coefficient-construction} for a schematic illustration.

Choose integers
\(
 0=n_0<n_1<n_2<\cdots
\)
such that
\begin{equation}\label{eq:q-block-indices}
 \mathcal T_{n_j}\leq 2^{-(j+1)},
 \qquad
 d_j:=n_{j+1}-n_j
 \quad\text{is nondecreasing in }j.
\end{equation}
Such a choice is possible because \(\mathcal T_k\to0\). First take
\(n_0=0\), and choose \(n_1>n_0\) sufficiently large that
\(\mathcal T_{n_1}\leq2^{-2}\).  After
\(n_{j-1}<n_j\) have been chosen for some \(j\geq1\), one may take
\(n_{j+1}\) sufficiently
large that
\(
 \mathcal T_{n_{j+1}}\leq2^{-(j+2)},\) and \(
 n_{j+1}-n_j\geq n_j-n_{j-1}.
\)

On the block
\(
 B_j:=\{n_j,n_j+1,\ldots,n_{j+1}-1\},
\)
define the auxiliary coefficients
\begin{equation}\label{eq:q-rough-plateaux}
 \widehat q_k=1/{d_j},
 \qquad k\in B_j.
\end{equation}
Because \(d_j\) is nondecreasing, \(\widehat q_k\) is positive and
nonincreasing.  Moreover, every plateau carries exactly 
\(
 \sum_{k\in B_j}\widehat q_k=1.
\)
Hence
\begin{equation}\label{eq:q-rough-properties}
 \sum_{k=0}^{\infty}\widehat q_k=\infty,
 \qquad
 \sum_{k=0}^{\infty}\widehat q_k\mathcal T_k
 \leq\sum_{j=0}^{\infty}\mathcal T_{n_j}
 \leq\sum_{j=0}^{\infty}2^{-(j+1)}<\infty.
\end{equation}

We now smooth the downward jumps of \(\widehat q_k\) by setting
\begin{equation}\label{eq:q-geometric-majorant}
 q_k:=\max_{0\leq i\leq k}c^{\,k-i}\widehat q_i.
\end{equation}
Equivalently,
\(
 q_0=\widehat q_0, q_{k+1}=\max\{\widehat q_{k+1},cq_k\}.
\)
Since
\(\widehat q_{k+1}\leq\widehat q_k\leq q_k\), this recursion gives
\begin{equation*}
 cq_k\leq q_{k+1}\leq q_k.
\end{equation*}
Furthermore, \(q_k\geq\widehat q_k\), so
\begin{equation*}
 \sum_{k=0}^{\infty}q_k
 \geq\sum_{k=0}^{\infty}\widehat q_k
 =\infty.
\end{equation*}

Finally, Tonelli's theorem and the
monotonicity of \(\mathcal T_k\) imply
\begin{align*}
 \sum_{k=0}^{\infty}q_k\mathcal T_k
 &\leq
 \sum_{k=0}^{\infty}\sum_{i=0}^{k}
 c^{\,k-i}\widehat q_i\mathcal T_k =\sum_{i=0}^{\infty}\widehat q_i
   \sum_{r=0}^{\infty}c^r\mathcal T_{i+r}\leq
 \frac1{1-c}
 \sum_{i=0}^{\infty}\widehat q_i\mathcal T_i
 <\infty.
\end{align*}
This proves all the assertions.
\end{proof}
\begin{figure}[t]
\centering
\begin{tikzpicture}[
  x=.48cm,y=.86cm,
  axis/.style={-{Stealth[length=2mm]},thin},
  guide/.style={gray!55,densely dotted,thin},
  raw/.style={line width=1.5pt,dashed},
  smooth/.style={qred,line width=1.7pt,line join=round},
  every node/.style={font=\small}
]
  % Colored blocks
  \fill[stageblue!10]   (.7,.6) rectangle (5.3,6.55);
  \fill[stageteal!10]   (5.7,.6) rectangle (11.3,6.55);
  \fill[stageorange!11] (11.7,.6) rectangle (17.3,6.55);
  \fill[stagepurple!10] (17.7,.6) rectangle (24.3,6.55);

  % Axes and block boundaries
  \draw[axis] (.5,.75) -- (24.8,.75) node[right] {$k$};
  \draw[axis] (.5,.75) -- (.5,7.55);
  \node[rotate=90] at (-.05,4.15) {coefficient};

  \foreach \x in {5.5,11.5,17.5}
    \draw[guide] (\x,.75) -- (\x,6.55);

  \node[stageblue!70!black]   at (3,6.72) {$B_j$};
  \node[stageteal!70!black]   at (8.5,6.72) {$B_{j+1}$};
  \node[stageorange!80!black] at (14.5,6.72) {$B_{j+2}$};
  \node[stagepurple!75!black] at (21,6.72) {$B_{j+3}$};

  % Raw unit-mass plateaux
  \draw[stageblue,raw]   (1,5.75) -- (5,5.75);
  \draw[stageteal,raw]   (6,4.75) -- (11,4.75);
  \draw[stageorange,raw] (12,3.45) -- (17,3.45);
  \draw[stagepurple,raw] (18,2.25) -- (24,2.25);

  \node[stageblue!70!black,above] at (3,5.75)
    {$\widehat q_k=d_j^{-1}$};
  \node[stageteal!70!black,above] at (8.5,4.75)
    {$\widehat q_k=d_{j+1}^{-1}$};
  \node[stageorange!80!black,above] at (14.5,3.45)
    {$\widehat q_k=d_{j+2}^{-1}$};
  \node[stagepurple!75!black,above] at (21,2.25)
    {$\widehat q_k=d_{j+3}^{-1}$};

  % Geometric majorant
  \draw[smooth] plot coordinates {
    (1,5.75) (2,5.75) (3,5.75) (4,5.75) (5,5.75)
    (6,5.32) (7,4.98) (8,4.75) (9,4.75) (10,4.75) (11,4.75)
    (12,4.18) (13,3.75) (14,3.45) (15,3.45) (16,3.45) (17,3.45)
    (18,3.00) (19,2.62) (20,2.25) (21,2.25)
    (22,2.25) (23,2.25) (24,2.25)
  };

  % Legend
  \draw[qred,line width=1.7pt] (2,7.45) -- (3,7.45);
  \node[anchor=west] at (3.15,7.45)
    {$q_k=\max_{0\leq i\leq k}
      2^{-(k-i)/2}\widehat q_i$};

  \draw[stageblue,raw] (16.6,7.45) -- (17.6,7.45);
  \node[anchor=west] at (17.75,7.45)
    {raw plateaux \(\widehat q_k\)};

  % Block properties
  \node[align=center,font=\footnotesize] at (3,3.85)
    {$\displaystyle\sum_{k\in B_j}\widehat q_k=1$\\[-1mm]
     \(\mathcal T_k\leq2^{-(j+1)}\)};

  \node[align=center,font=\footnotesize] at (8.5,2.75)
    {$\displaystyle\sum_{k\in B_{j+1}}\widehat q_k=1$\\[-1mm]
     \(\mathcal T_k\leq2^{-(j+2)}\)};

  \node[align=center,font=\footnotesize] at (14.5,1.72)
    {$\displaystyle\sum_{k\in B_{j+2}}\widehat q_k=1$\\[-1mm]
     \(\mathcal T_k\leq2^{-(j+3)}\)};

  \node[align=center,font=\footnotesize] at (21,1.52)
    {$\displaystyle\sum_{k\in B_{j+3}}\widehat q_k=1$\\[-1mm]
     \(\mathcal T_k\leq2^{-(j+4)}\)};

  % Transition annotation
  \node[qred,align=center,font=\footnotesize]
    (bridge) at (14.4,5.45)
    {geometric bridge\\
     \(q_{k+1}\geq2^{-1/2}q_k\)};
  \draw[qred,-{Stealth[length=1.6mm]},thin]
    (bridge.south west) -- (12.15,4.18);

  % Block lengths
  \draw[decorate,decoration={brace,mirror,amplitude=4pt}]
    (1,.50) -- (5,.50)
    node[midway,below=5pt] {$d_j=n_{j+1}-n_j$};

  \draw[decorate,decoration={brace,mirror,amplitude=4pt}]
    (6,.50) -- (11,.50)
    node[midway,below=5pt] {$d_{j+1}$};

  \draw[decorate,decoration={brace,mirror,amplitude=4pt}]
    (12,.50) -- (17,.50)
    node[midway,below=5pt] {$d_{j+2}$};

  \draw[decorate,decoration={brace,mirror,amplitude=4pt}]
    (18,.50) -- (24,.50)
    node[midway,below=5pt] {$d_{j+3}$};
\end{tikzpicture}
\caption{Tail-adapted block coefficients and their geometric majorant.
The colored dashed plateaux form the auxiliary sequence
\(\widehat q_k\).  Every block
\(B_j=[n_j,n_{j+1})\cap\mathbb N_0\) has
\(\widehat q\)-mass one, while the initial tail on that block is at
most \(2^{-(j+1)}\).  The red sequence \(q_k\) replaces each downward
jump by an admissible geometric transition.  }
\label{fig:q-coefficient-construction}
\end{figure}
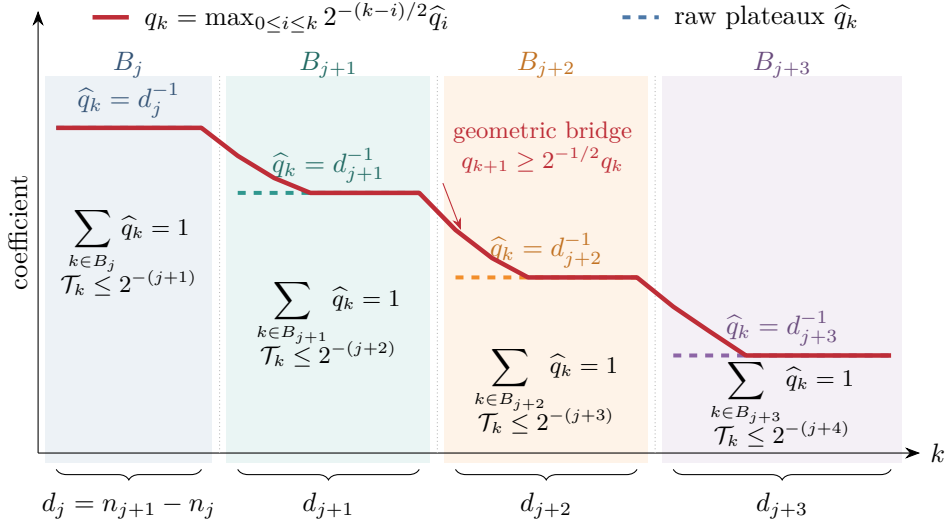

For each level \(L_k\), the quantity \((m-L_k)_+\) measures the excess
mass above \(L_k\), which we regard as the mass token carried at that
level.  Weighting these levelwise tokens by the coefficients constructed
in \cref{lem:q-coefficients}, we define the tail-adapted mass-token weight
\begin{equation}\label{eq:Phi-definition}
   \Phi(m):=m+\sum_{k=0}^{\infty}q_k(m-L_k)_+.
\end{equation}
Since \(L_k\to\infty\), for each fixed \(m>0\) the sum contains only
finitely many nonzero terms.  Hence \(\Phi\) is well defined.

\begin{lemma}[Properties of the tail-adapted mass-token weight]
\label{lem:Phi-properties}
The function \(\Phi\) is finite, continuous, nonnegative, increasing,
and convex.  Moreover,
\begin{equation}\label{eq:Phi-initial}
 \int_X\Phi(m)f_0(x)\,dx<\infty,
\end{equation}
the quotient \(m\mapsto\Phi(m)/m\) is nondecreasing, and
\begin{equation}\label{eq:Phi-superlinear}
 \lim_{m\to\infty}\frac{\Phi(m)}m=\infty.
\end{equation}
\end{lemma}

\begin{proof}
Since \(L_k\to\infty\), the series defining \(\Phi\) is locally
finite.  Indeed, for any \(M>0\) and \(m\in[0,M]\), all terms with
\(L_k\geq M\) vanish, so that
\begin{equation*}
 \Phi(m)=m+\sum_{\{k:L_k<M\}}q_k(m-L_k)_+.
\end{equation*}
This is a finite sum of continuous, nonnegative, increasing, and convex
functions.  Since \(M>0\) is arbitrary, \(\Phi\) has these properties
on \([0,\infty)\); in particular, it is finite everywhere.

Since \(f_0\geq0\), Tonelli's theorem and the definition of
\(\mathcal T_k\) give
\begin{equation}
 \int_X\Phi(m)f_0(x)\,dx
 =M_1(f_0)
   +\sum_{k=0}^{\infty}q_k\int_X(m-L_k)_+f_0(x)\,dx
 \leq M_1(f_0)+\sum_{k=0}^{\infty}q_k\mathcal T_k<\infty.
\end{equation}
This proves \eqref{eq:Phi-initial}.

For \(m>0\), we may write
\begin{equation}
 \frac{\Phi(m)}m
 =1+\sum_{k=0}^{\infty}q_k
      \left(1-\frac{L_k}{m}\right)_+.
\end{equation}
For every \(k\), the function
\(m\mapsto(1-L_k/m)_+\) is nondecreasing on \((0,\infty)\).
Hence \(m\mapsto\Phi(m)/m\) is nondecreasing as well. Finally, let \(A>1\).  Since
\(\sum_{k=0}^{\infty}q_k=\infty\), there exists \(J\) such that
\(\sum_{k=0}^{J}q_k\geq2(A-1)\).  If \(m\geq2L_J\), then
\(L_k/m\leq1/2\) for every \(k\leq J\), and consequently
\begin{equation}
 \frac{\Phi(m)}m
 \geq1+\sum_{k=0}^{J}q_k
       \left(1-\frac{L_k}{m}\right)
 \geq1+\frac12\sum_{k=0}^{J}q_k
 \geq A.
\end{equation}
Thus, for every \(A>1\), \(\Phi(m)/m\) is at least \(A\)
whenever \(m\geq2L_J\).  This proves
\eqref{eq:Phi-superlinear}.
\end{proof}

For \(S>0\) and \(r\in(0,1)\), define the token loss under the split
\(S\mapsto(rS,(1-r)S)\) by
\begin{equation}\label{eq:Dr-definition}
 D_r(S)=\Phi(S)-\Phi(rS)-\Phi((1-r)S).
\end{equation}
\begin{lemma}[Exact loss of one hinge]
\label{lem:one-hinge}
For \(r\in(0,1)\) and \(z\geq0\), set
\begin{equation}\label{eq:dr-profile}
 d_r(z)=
 \begin{cases}
  \min\{r,1-r,z,1-z\},&0\leq z\leq1,\\
  0,&z>1.
 \end{cases}
\end{equation}
Then, for all \(L,S>0\),
\begin{align}
 (S-L)_+-(rS-L)_+-((1-r)S-L)_+=Sd_r\left(\frac LS\right).
 \label{eq:one-hinge-identity}
\end{align}
Consequently,
\begin{equation}\label{eq:Dr-exact}
 D_r(S)=S\sum_{\{k:L_k<S\}}q_k
 d_r\left(\frac{L_k}{S}\right)\geq0.
\end{equation}
\end{lemma}

\begin{proof}
Both sides of \eqref{eq:one-hinge-identity} are unchanged when
\(r\) is replaced by \(1-r\), so it suffices to take
\(0<r\leq1/2\).  If \(L\geq S\), then
\begin{equation}
 (S-L)_+=(rS-L)_+=((1-r)S-L)_+=0.
\end{equation}
If \(0<L<S\), put \(z=L/S\in(0,1)\).  Factoring out \(S\) in the
three terms on the left-hand side of
\eqref{eq:one-hinge-identity} gives
\begin{equation}
 (S-L)_+=S(1-z),
 \qquad
 (rS-L)_+=S(r-z)_+,
 \qquad
 ((1-r)S-L)_+=S(1-r-z)_+.
 \label{eq:one-hinge-scaling}
\end{equation}
Consequently,
\begin{align}
 &(1-z)-(r-z)_+-((1-r)-z)_+
 \notag\\
 &\quad=
 \begin{cases}
 (1-z)-(r-z)-(1-r-z)=z,
     &0<z\leq r,\\
 (1-z)-(1-r-z)=r,
     &r<z\leq1-r,\\
 1-z,
     &1-r<z<1,
 \end{cases}
 \notag\\
 &\quad=\min\{r,1-r,z,1-z\}=d_r(z).
 \label{eq:one-hinge-piecewise}
\end{align}
Multiplication by \(S\) proves \eqref{eq:one-hinge-identity}.
Using \eqref{eq:Phi-definition} and cancellation of the linear part,
\begin{align}
 D_r(S)
 &=\sum_{k=0}^{\infty}q_k
 \bigl[(S-L_k)_+-(rS-L_k)_+-((1-r)S-L_k)_+\bigr]
 \notag\\
 &=S\sum_{\{k:L_k<S\}}q_kd_r(L_k/S)\geq0,
\end{align}
which is \eqref{eq:Dr-exact}.
In particular, the definition of \(d_r\) gives
\begin{equation}
 D_r(S)=D_{1-r}(S),
 \qquad
 0\leq d_r(z)\leq\min\{r,1-r,z\}
 \quad \text{for }0\leq z\leq1.
 \label{eq:Dr-basic-consequences}
\end{equation}
We then complete the proof.
\end{proof}

For \(S\geq L_0\), let \(K(S)\) be the unique integer satisfying
\begin{equation}\label{eq:active-index}
 L_{K(S)}\leq S<L_{K(S)+1}.
\end{equation}
For \(k \ge K(S) +1 \), one has $S<L_{K(S)+1} \le L_k$. Hence \(d_r(L_k/S)=0\).
\eqref{eq:Dr-exact} may equivalently be written as the finite sum
\begin{equation}
 D_r(S)=S\sum_{k=0}^{K(S)}q_k
 d_r\left(\frac{L_k}{S}\right).
 \label{eq:Dr-finite-active-sum}
\end{equation}
If \(S=L_{K(S)}\), the last summand in
\eqref{eq:Dr-finite-active-sum} vanishes, so
\begin{equation*}
D_r(S)=S\sum_{\{k:L_k<S\}}q_k
 d_r\left(\frac{L_k}{S}\right)\,.
\end{equation*}

\begin{lemma}[Small-share upper bound and share-weighted lower bound]
\label{lem:Dr-bounds}
There is \(C<\infty\) such that, whenever
\(S\geq L_0\), \(K=K(S)\), and \(0<r\leq1/2\),
\begin{equation}\label{eq:Dr-upper}
 0\leq D_r(S)\leq
 C S q_K\sqrt r.
\end{equation}
Whenever \(S\geq L_1=2L_0\),
\begin{equation}\label{eq:Dr-weighted-lower}
 D_\alpha(S)\geq \frac12 S q_{K}
 \min\{\alpha,1-\alpha\},
 \qquad 0<\alpha<1.
\end{equation}
\end{lemma}

\begin{proof}
Write \(K=K(S)\) and \(n=K-k\).  Equations \eqref{eq:q-properties}
and \eqref{eq:active-index} imply
\begin{equation}\label{eq:level-comparison}
 q_{K-n}\leq2^{n/2}q_K,
 \qquad
 \frac{L_{K-n}}S=2^{-n}\frac{L_{K}}S\leq2^{-n}.
\end{equation}
Since \(d_r(z)\leq\min\{r,z\}\), each term in
\eqref{eq:Dr-exact} satisfies
\begin{align}
 \frac{q_{K-n}}{q_K}
 d_r\left(\frac{L_{K-n}}S\right)
 \leq2^{n/2}
 \min\left\{r,\frac{L_{K-n}}S\right\}
 \leq\min\{r2^{n/2},2^{-n/2}\}.
 \label{eq:one-level-Dr-upper}
\end{align}
Summing \eqref{eq:one-level-Dr-upper} for \(0\leq n\leq K\)
gives
\begin{equation}\label{eq:hinge-geometric-sum}
 \frac{D_r(S)}{Sq_K}
 \leq\sum_{n=0}^K
 \min\{r2^{n/2},2^{-n/2}\}.
\end{equation}
Choose \(n_r\geq0\) so that
\(
 2^{-(n_r+1)}<r\leq2^{-n_r}.
 \)
If \(0\leq n\leq n_r\), then
\(r\leq2^{-n_r}\leq2^{-n}\), whereas, if \(n\geq n_r+1\), then
\(2^{-n}\leq2^{-(n_r+1)}<r\).  Hence
\begin{equation}
 \min\{r2^{n/2},2^{-n/2}\}
 =\begin{cases}
 r2^{n/2},&0\leq n\leq n_r,\\
 2^{-n/2},&n\geq n_r+1,
 \end{cases}
 \label{eq:Dr-crossover}
\end{equation}
and, since all summands are nonnegative,
\begin{align}
 \sum_{n=0}^{K}\min\{r2^{n/2},2^{-n/2}\}
 &\leq\sum_{n=0}^{n_r}r2^{n/2}
 +\sum_{n=n_r+1}^{\infty}2^{-n/2}.
 \label{eq:Dr-split-sum}
\end{align}
Then
\begin{align}
 \sum_{n=0}^{n_r}r2^{n/2}
 &=r\frac{2^{(n_r+1)/2}-1}{2^{1/2}-1}
 \leq\frac{2^{1/2}}{2^{1/2}-1}\sqrt r,
 \label{eq:Dr-geometric-first}\\
 \sum_{n=n_r+1}^{\infty}2^{-n/2}
 &=\frac{2^{-(n_r+1)/2}}{1-2^{-1/2}}
 \leq\frac{1}{1-2^{-1/2}}\sqrt r.
 \label{eq:Dr-geometric-second}
\end{align}
Indeed, the first inequality uses \(2^{n_r}\leq r^{-1}\), while the
second uses \(2^{-(n_r+1)}<r\).  Equations
\eqref{eq:hinge-geometric-sum}--\eqref{eq:Dr-geometric-second} prove
\eqref{eq:Dr-upper}.

If \(S\geq L_1\), then \(K=K(S)\geq1\).  From
\(L_K\leq S<L_{K+1}\) and \(L_k=2^kL_0\),
\begin{equation}
 \frac{L_{K-1}}S
 =\frac{L_K}{2S}\leq\frac12,
 \qquad
 \frac{L_{K-1}}S
 =\frac{L_K}{2S}>\frac{L_K}{2L_{K+1}}=\frac14.
 \label{eq:previous-level-position-calculation}
\end{equation}
Hence
\begin{equation}\label{eq:profile-rho-lower}
 d_\alpha(L_{K-1}/S)=\min\{\alpha,1-\alpha,L_{K-1}/S,1-L_{K-1}/S\}
 \geq\min\{\alpha,1-\alpha,1/4\}
 \geq\frac12\min\{\alpha,1-\alpha\}.
\end{equation}
Retaining only the hinge \(K-1\) in \eqref{eq:Dr-exact} and using
\(q_{K-1}\geq q_K\),
\begin{align}
 D_\alpha(S)
 &\geq S q_{K-1}d_\alpha(L_{K-1}/S)
 \geq\frac12S q_{K-1}\min\{\alpha,1-\alpha\}
 \geq\frac12S q_K\min\{\alpha,1-\alpha\},
\end{align}
which proves \eqref{eq:Dr-weighted-lower}.
\end{proof}
For the incoming pair, put
\begin{equation}\label{eq:incoming-share}
 \vartheta
 =\frac{\min\{m,m_1\}}{m+m_1}
 \in(0,1/2].
\end{equation}
The notation \(\mathcal A_\varphi\) is first introduced in
\eqref{eq:A-varphi} for bounded \(\varphi\), but since for each fixed incoming
pair \((m,m_1)\), the four masses
\(\alpha S\), \((1-\alpha)S\), \(m\), and \(m_1\) belong to
\([0,S]\), where \(\Phi\) is bounded, \(\mathcal A_\Phi(m,m_1)\) is well defined.  Using
\eqref{eq:model-rate-disintegration}, we therefore set
\begin{equation}
 \mathcal A_\Phi(m,m_1)
 :=
 \lambda(m,m_1)\int_0^1g(\alpha)
 [
 \Phi(\alpha S)+\Phi((1-\alpha)S)
 -\Phi(\vartheta S)-\Phi((1-\vartheta)S)
 ]\,d\alpha.
 \label{eq:A-Phi}
\end{equation}
Since
\begin{align}
 \Phi(m)+\Phi(m_1)
 &=\Phi(S)-D_{\vartheta}(S),
 \label{eq:incoming-Phi-sum}\\
 \Phi(\alpha S)+\Phi((1-\alpha)S)
 &=\Phi(S)-D_\alpha(S),
 \label{eq:outgoing-Phi-sum}
\end{align}
and \eqref{eq:model-fixed-share}--\eqref{eq:model-rate-disintegration}, we have
\begin{align} \label{eq:signed-share-identity}
 \mathcal A_\Phi(m,m_1)
 &=\int_0^1\lambda(m,m_1)g(\alpha)
 \bigl[D_{\vartheta}(S)-D_\alpha(S)\bigr] \,d\alpha
 \notag\\
 &=\lambda(m,m_1)D_{\vartheta}(S)
 \int_0^1g(\alpha)\,d\alpha
 -\lambda(m,m_1)\int_0^1g(\alpha)D_\alpha(S)\,d\alpha\notag\\ &=\lambda(m,m_1)\left[D_{\vartheta}(S) - \int_0^1g(\alpha)D_\alpha(S)\,d\alpha\right].
\end{align}
Substituting \eqref{eq:Dr-exact} into both terms of
\eqref{eq:signed-share-identity} gives the hinge-by-hinge identity
\begin{align}
 \mathcal A_\Phi(m,m_1)
 &=\lambda(m,m_1)S
 \sum_{\{k:L_k<S\}}q_k
 \left[
 d_{\vartheta}\left(\frac{L_k}{S}\right)
 -\int_0^1g(\alpha)
 d_\alpha\left(\frac{L_k}{S}\right)\,d\alpha
 \right].
 \label{eq:signed-hinge-by-hinge}
\end{align}

For the remainder of this section, set
\begin{equation}
 G_1:=\mathbb E_{\alpha\sim g}[\min\{\alpha,1-\alpha\}]=\int_0^1\min\{\alpha,1-\alpha\}g(\alpha)\,d\alpha
 \in(0,1/2].
\end{equation}
The strict positivity of \(G_1\) follows from
\(g\geq0\), \(g\in C([0,1])\), and \(\displaystyle\int_0^1g=1\): the function
\(g\) is positive on a nonempty interval meeting \((0,1)\), where
\(\min\{\alpha,1-\alpha\}>0\).

\begin{proposition}[Signed redistribution estimate]
\label{prop:signed-redistribution}
There are \(\vartheta_0\in(0,1/4]\) and \(C<\infty\) such that, for
\(S\geq L_1\),
\begin{align}
 \mathcal A_\Phi(m,m_1)
 &\leq S q_{K(S)}\lambda(m,m_1)
 \left(C\sqrt{\vartheta}-\frac{G_1}{2}\right),
 \label{eq:general-signed}\\
 \vartheta\leq \vartheta_0
 &\quad\Longrightarrow\quad
 \mathcal A_\Phi(m,m_1)
 \leq-\frac{G_1}{4}
 S q_{K(S)}\lambda(m,m_1),
 \label{eq:general-negative}\\
 \bigl(\mathcal A_\Phi(m,m_1)\bigr)_+
 &\leq C S q_{K(S)}\lambda(m,m_1).
 \label{eq:general-positive}
\end{align}
\end{proposition}

\begin{proof}
For \(S\geq L_1\), \eqref{eq:Dr-upper} gives
\begin{equation}
 D_{\vartheta}(S)
 \leq C S q_{K(S)}\sqrt{\vartheta},
 \label{eq:signed-incoming-upper}
\end{equation}
while \eqref{eq:Dr-weighted-lower} and the definition of \(G_1\) give
\begin{align}
 \int_0^1g(\alpha)D_\alpha(S)\,d\alpha
 \geq\frac12S q_{K(S)}
 \int_0^1\min\{\alpha,1-\alpha\}g(\alpha)\,d\alpha=\frac{G_1}{2}S q_{K(S)}.
 \label{eq:signed-outgoing-lower}
\end{align}
Substituting \eqref{eq:signed-incoming-upper} and
\eqref{eq:signed-outgoing-lower} into
\eqref{eq:signed-share-identity} yields
\begin{align}
 \mathcal A_\Phi(m,m_1)
 &\leq\lambda(m,m_1)
 \left[C S q_{K(S)}\sqrt{\vartheta}
 -\frac{G_1}{2}S q_{K(S)}\right]
 \notag\\
 &=S q_{K(S)}\lambda(m,m_1)
 \left(C\sqrt{\vartheta}-\frac{G_1}{2}\right),
\end{align}
which proves \eqref{eq:general-signed}.  Since \(G_1>0\) and
\(C>0\), set
\begin{equation}\label{eq:vartheta0-choice}
 \vartheta_0=\min\left\{\frac14,
 \left(\frac{G_1}{4C}\right)^2\right\}
 \in(0,1/4],
 \qquad
 C\sqrt{\vartheta_0}\leq\frac{G_1}{4}.
\end{equation}
If \(\vartheta\leq \vartheta_0\), then
\begin{align}
 C\sqrt{\vartheta}-\frac{G_1}{2}
 &\leq C\sqrt{\vartheta_0}-\frac{G_1}{2}
 \leq-\frac{G_1}{4},
\end{align}
which proves \eqref{eq:general-negative}.  Finally,
\eqref{eq:signed-share-identity}, \(\lambda(m,m_1)>0\),
\(D_\alpha(S)\geq0\), and
\eqref{eq:signed-incoming-upper} imply, term by term,
\begin{align}
 \bigl(\mathcal A_\Phi(m,m_1)\bigr)_+
 &=\lambda(m,m_1)
 \left[D_{\vartheta}(S)
 -\int_0^1g(\alpha)D_\alpha(S)\,d\alpha\right]_+
 \notag\\
 &\leq\lambda(m,m_1)D_{\vartheta}(S)
 \notag\\
 &\leq C S q_{K(S)}\lambda(m,m_1)
 \sqrt{\vartheta}
 \leq C S q_{K(S)}\lambda(m,m_1),
\end{align}
which is \eqref{eq:general-positive}.
\end{proof}

\subsection{Comparison with bounded-mass partners and propagation of the
\texorpdfstring{\(\Phi\)}{Phi}-moment}

The purpose of this subsection is to propagate the tail-adapted
\(\Phi\)-moment constructed in \cref{subsec:tokens}.  The evolution of
this moment is governed by the signed increment
\(\mathcal A_\Phi(m,m_1)\).  Collisions with
\(\mathcal A_\Phi>0\) increase the moment, whereas collisions with
\(\mathcal A_\Phi<0\) decrease it.  We shall show that the positive
contribution is uniformly bounded on bounded mass ranges, and is
controlled at large mass by a part of the negative contribution.

\begin{theorem}[Uniform propagation of the tail-adapted moment]
\label{thm:Phi-propagation}
For every \(T>0\), there is \(C_T<\infty\), independent of \(N\), such
that
\begin{equation}
 \sup_{N\geq1}\sup_{0\leq t\leq T}
 \int_X\Phi(m)f_N(t,x)\,dx
 \leq \int_X\Phi(m)f_0(x)\,dx+C_T.
 \label{eq:Phi-propagation}
\end{equation}
\end{theorem}

We first connect the quantity \(\mathcal A_\Phi\) with the time
evolution of the approximate solutions. 

\begin{lemma}[Evolution identity for the \(\Phi\)-moment]
\label{lem:Phi-evolution}
For every \(N\geq1\), the map
\begin{equation*}
 t\longmapsto\int_X\Phi(m)f_N(t,x)\,dx
\end{equation*}
is locally absolutely continuous.  For almost every \(t>0\),
\begin{align}
 \frac{d}{dt}\int_X\Phi(m)f_N(t,x)\,dx
 =\frac{1}2\norm{b}_{L^1}
 \int_{X^2}\chi_N(S)E^\gamma\mathcal A_\Phi(m,m_1)
 f_N(t,x)f_N(t,x_1)\,dx_1\,dx.
 \label{eq:Phi-moment-derivative}
\end{align}
\end{lemma}

\begin{proof}
For a truncation height \(H>0\), set
\(\Phi^{[H]}=\min\{\Phi,H\}\).  Since \(\Phi^{[H]}\) is bounded and
Borel, \eqref{eq:mass-test-identity} gives, for \(0\leq s\leq t\),
\begin{align}
 &\int_X\Phi^{[H]}(m)f_N(t,x)\,dx
 -\int_X\Phi^{[H]}(m)f_N(s,x)\,dx
 \notag\\
 &\quad=\frac{1}2\norm{b}_{L^1}
 \int_s^t\int_{X^2}\chi_N(S)E^\gamma
 \mathcal A_{\Phi^{[H]}}(m,m_1)
 f_N(\tau,x)f_N(\tau,x_1)\,dx_1\,dx\,d\tau.
 \label{eq:truncated-Phi-identity}
\end{align}
If \(\chi_N(S)\neq0\), then \(S<2N\).  Both incoming masses and both
outgoing masses therefore belong to \([0,2N]\).  Since \(\Phi\) is
continuous and increasing, the finite number
\(
 H_N:=1+\Phi(2N)
\)
satisfies: for every \(H\geq H_N\),
\begin{equation}
 \mathcal A_{\Phi^{[H]}}(m,m_1)=\mathcal A_\Phi(m,m_1)
 \qquad\text{whenever }\chi_N(S)\neq0.
 \label{eq:truncated-Phi-agreement}
\end{equation}
On the same active set, \(\displaystyle\int_0^1g(\alpha)\,d\alpha=1\),
\eqref{eq:model-intensity-comparison}, and
\eqref{eq:model-regular-variation} give
\begin{equation}
 \abs{\mathcal A_\Phi(m,m_1)}
 \leq4\lambda_+\kappa_+(1+2N)\Phi(2N).
 \label{eq:active-Phi-bound}
\end{equation}
For \(H\geq H_N\), \eqref{eq:unweighted-rate} and
\eqref{eq:truncated-moments} give
\begin{equation}
 \int_{X^2}E^\gamma
 f_N(\tau,x)f_N(\tau,x_1)\,dx_1\,dx
 \leq2^\gamma M_0(f_0)^{2-\gamma}M_2(f_0)^\gamma,
 \label{eq:truncated-unweighted-rate}
\end{equation}
\begin{align}
   &\frac{1}2\norm{b}_{L^1}
 \int_s^t\int_{X^2}\chi_N(S)E^\gamma
 |\mathcal A_{\Phi^{[H]}}(m,m_1)|
 f_N(\tau,x)f_N(\tau,x_1)\,dx_1\,dx\,d\tau
 \notag\\
 &\quad\leq
 2^{1+\gamma}(t-s)\norm{b}_{L^1(\Sph)}
 \lambda_+\kappa_+(1+2N)\Phi(2N)
 M_0(f_0)^{2-\gamma}M_2(f_0)^\gamma.
\end{align}

Take \(s=0\) in \eqref{eq:truncated-Phi-identity} and let
\(H\to\infty\).  Monotone convergence,
\eqref{eq:Phi-initial}, and
\eqref{eq:truncated-Phi-agreement} yield
\begin{align}
 &\int_X\Phi(m)f_N(t,x)\,dx-\int_X\Phi(m)f_0(x)\,dx
 \notag\\
 &\quad=\frac{1}2\norm{b}_{L^1}
 \int_0^t\int_{X^2}\chi_N(S)E^\gamma\mathcal A_\Phi
 f_N(\tau,x)f_N(\tau,x_1)\,dx_1\,dx\,d\tau.
 \label{eq:integrated-Phi-identity}
\end{align}
Subtracting \eqref{eq:integrated-Phi-identity} at times \(s\) and
\(t\) proves local absolute continuity and
\eqref{eq:Phi-moment-derivative}.
\end{proof}

Fix \(T>0\), and let
\(\mathcal B_T=\{(m,v):r_T\leq m\leq R_T\}\) and \(\eta_T>0\) be
given by \cref{cor:bounded-mass-set}.  If \(M_0(f_0)=0\), then
\(f_0=0\) almost everywhere and \(f_N\equiv0\), so all the estimates
below are immediate.  We henceforth assume \(M_0(f_0)>0\).

Consider the signed collision integral on the right-hand side of
\eqref{eq:Phi-moment-derivative}.  Set
\(
 M:=\max\{m,m_1\},
\)
and introduce a dyadic index \(J=J(T)\), whose size will be fixed after
the comparison estimates below.  Now we divide the incoming space into four classes, as summarized in \cref{tab:four-Phi-classes}:
\begin{align}
  X^2 &=\Omega_{\mathrm I}\cup\Omega_{\mathrm {II}}\cup\Omega_{\mathrm {III}}\cup\Omega_{\mathrm {IV}},\\
 \Omega_{\mathrm I}
 &=\{M<\vartheta_0L_J\},
 \label{eq:class-I}\\
 \Omega_{\mathrm {II}}
 &=\{M\geq\vartheta_0L_J,\ \vartheta\leq\vartheta_0\},
 \label{eq:class-II}\\
 \Omega_{\mathrm {III}}
 &=\{\vartheta_0L_J\leq M<L_J,\ \vartheta>\vartheta_0\},
 \label{eq:class-III}\\
 \Omega_{\mathrm {IV}}
 &=\{M\geq L_J,\ \vartheta>\vartheta_0\}.
 \label{eq:class-IV}
\end{align}
Classes I and III form a bounded mass range and will be estimated
directly.  All pairs in Class II have a negative contribution of $\mathcal A_\Phi(m,m_1)$, and a suitable part will be retained to control the positive contribution of Class IV. Indeed, we divide Class IV into dyadic shells $\mathcal H_j=\{L_j\leq M<L_{j+1}\}$ for $j\geq J$ and compare the positive contribution of each shell, which gains a factor \(L_j^{\gamma-1}\). Since \(\gamma<1\), choosing \(J\) sufficiently large will make the total positive contribution of  Class IV absorbable by the retained part of Class II.  The comparison estimates and the precise choice of \(J\) are given below.

\begin{table}[h!]
\centering
\small
\renewcommand{\arraystretch}{1.25}
\begin{tabular}{@{}c p{0.28\textwidth} p{0.55\textwidth}@{}}
\toprule
Class & Incoming-pair geometry & Role in the estimate\\
\midrule
I
& \(M<\vartheta_0L_J\)
& Bounded incoming masses.  We can control the positive part directly by
the physical moment bounds.\\

II
& \(M\geq\vartheta_0L_J\), \(\vartheta\leq\vartheta_0\)
& Unequal incoming masses with 
\(\mathcal A_\Phi\leq0\).  Retain the collisions in which the large
particle meets \(\mathcal B_T\) as the negative contribution used
against Class IV.\\

III
& \(\vartheta_0L_J\leq M<L_J\), \(\vartheta>\vartheta_0\)
& Comparable masses in the bounded transition range.  We can control the
positive part directly as in Class I.\\

IV
& \(M\geq L_J\), \(\vartheta>\vartheta_0\)
& Large comparable masses (the main difficulty).  We can decompose the region into dyadic shells and compare
their positive contribution with the negative contribution retained
from Class II.\\
\bottomrule
\end{tabular}
\caption{The four incoming-pair configurations and their roles in the
signed \(\Phi\)-estimate.}
\label{tab:four-Phi-classes}
\end{table}

The next two lemmas provide the scale and collision-frequency
comparisons needed for Class IV.  After proving them, we fix \(J\) and
estimate the four classes separately.

\begin{lemma}[Comparison of the signed-increment scale]
\label{lem:Phi-scale-comparison}
Fix \(0<c<C<\infty\).  There are constants
\(0<C_1(c,C)\leq C_2(c,C)<\infty\) such that, for every sufficiently
large dyadic level \(R=L_j\), whenever
\(cR\leq S\leq CR\),
\begin{equation}
 C_1\kappa(R)R q_j
 \leq S q_{K(S)}\lambda(m,m_1)
 \leq C_2\kappa(R)R q_j.
 \label{eq:Phi-scale-comparison}
\end{equation}
\end{lemma}
\begin{proof}
Set \(K=K(S)\).  We first compare \(\kappa(S)\) with
\(\kappa(R)\).  By assumption \textnormal{(RV)}, the function
\(\kappa\) is regularly varying with index \(p\).  Hence the uniform
convergence theorem for regularly varying functions
\cite[Theorem~1.5.2]{BinghamGoldieTeugels1987} gives
\begin{equation}
 \sup_{\zeta\in[c,C]}
 \left|
 \frac{\kappa(\zeta R)}{\kappa(R)}-\zeta^p
 \right|
 \longrightarrow0
 \qquad\text{as }R\to\infty.
 \label{eq:kappa-uniform-compact-scale}
\end{equation}
Choose \(R_0\) so large that for \(R\geq R_0\),
the supremum in \eqref{eq:kappa-uniform-compact-scale} is at most
\(c^p/2\).  Hence
\begin{equation}
 \frac12c^p\kappa(R)
 \leq\kappa(\zeta R)
 \leq\left(C^p+\frac12c^p\right)\kappa(R).
 \label{eq:kappa-compact-scale}
\end{equation}
Taking \(\zeta=S/R\) yields
\begin{equation}
 \frac12c^p\kappa(R)
 \leq\kappa(S)
 \leq\left(C^p+\frac12c^p\right)\kappa(R).
 \label{eq:kappa-S-R-comparison}
\end{equation}
We next compare \(K=K(S)\) with \(j\).  Since
\(R=L_j\), \(L_K\leq S<L_{K+1}\), and \(L_k=2^kL_0\),
\begin{equation}
 2^{K-j}
 =\frac{L_K}{L_j}
 \leq\frac{S}{R}
 <\frac{L_{K+1}}{L_j}
 =2^{K-j+1}.
 \label{eq:active-index-ratio}
\end{equation}
Combining \eqref{eq:active-index-ratio} with
\(c\leq S/R\leq C\) gives
\begin{equation}
 K-j\leq\log_2C,
 \qquad
 j-K<1+\log_2(c^{-1}).
 \label{eq:active-index-two-sided}
\end{equation}
Therefore with
\begin{equation}
 h=h(c,C)
 :=1+\left\lceil
 \log_2\max\{C,c^{-1}\}
 \right\rceil,
 \label{eq:active-index-distance}
\end{equation}
we have
\(
 |K-j|\leq h.
\)
It remains to compare \(q_K\) with \(q_j\).  If \(K\geq j\), iterating
\(2^{-1/2}q_k\leq q_{k+1}\leq q_k\) from \(j\) to \(K\) gives
\(2^{-(K-j)/2}q_j\leq q_K\leq q_j.\)
If \(K<j\), iteration from \(K\) to \(j\) gives \(
 q_j\leq q_K\leq2^{(j-K)/2}q_j.\) 
Together with this bound, these two cases
imply
\(
 2^{-h/2}q_j\leq q_K\leq2^{h/2}q_j.
\)

Finally, using \(cR\leq S\leq CR\),
\eqref{eq:model-intensity-comparison} and
\eqref{eq:kappa-S-R-comparison}, we obtain
\begin{equation}
 S q_K\lambda(m,m_1)\geq\frac{\lambda_-}{2}\,
   c^{p+1}2^{-h/2}\,
   \kappa(R)R q_j,
 \label{eq:Phi-scale-lower-detailed}
\end{equation}
and
\begin{equation}
 S q_K\lambda(m,m_1)\leq \lambda_+C\,2^{h/2}
   \left(C^p+\frac12c^p\right)
   \kappa(R)R q_j.
 \label{eq:Phi-scale-upper-detailed}
\end{equation}
Thus \eqref{eq:Phi-scale-comparison} holds with
\begin{equation}
 C_1
 =\frac{\lambda_-}{2}c^{p+1}2^{-h/2},
 \qquad
 C_2
 =\lambda_+C\,2^{h/2}
 \left(C^p+\frac12c^p\right).
\end{equation}
\end{proof}

\begin{lemma}[Large-comparable versus large--bounded comparison]
\label{lem:large-pair-comparison}
For all sufficiently large \(j\), define
\begin{equation}
 \mathcal H_j
 =\{(x,x_1)\in X^2:
 L_j\leq M<L_{j+1},\ \vartheta>\vartheta_0\}, \qquad I_j=[\vartheta_0L_j,2L_j].
 \label{eq:comparable-shell}
\end{equation}
There is \(K_T<\infty\), independent of
\(N\), \(j\), and \(t\in[0,T]\), such that
\begin{align}
 &\int_{\mathcal H_j}\chi_N(S)E^\gamma
 \bigl(\mathcal A_\Phi(m,m_1)\bigr)_+
 f_N(t,x)f_N(t,x_1)\,dx_1\,dx
 \notag\\
 &\quad\leq K_T L_j^{\gamma-1}
 \int_{\substack{m\in I_j\\x_1\in\mathcal B_T}}
 \chi_N(S)E^\gamma
 \bigl[-\mathcal A_\Phi(m,m_1)\bigr]
 f_N(t,x)f_N(t,x_1)\,dx_1\,dx,
 \label{eq:large-pair-comparison}
\end{align}
where the integrand on the right-hand side is nonnegative.
\end{lemma}
\begin{proof}
Throughout the proof, \(K_T\) denotes a constant depending only on
\(T\) and the fixed data, independent of \(N\), \(j\), and
\(t\in[0,T]\).  Notice that it may be enlarged finitely many times in
the proof.

\medskip
\noindent\textbf{\underline{Step 1: localization of the shell and the positive increment.}}
If \((x,x_1)\in\mathcal H_j\), then
\begin{equation}
 \min\{m,m_1\}=\vartheta S>\vartheta_0S
 \geq\vartheta_0M\geq\vartheta_0L_j,
 \qquad
 M<L_{j+1}=2L_j.
\end{equation}
Thus \(\mathcal H_j\subset\{m,m_1\in I_j\}\) and
\(L_j\leq S<4L_j\).  By
\eqref{eq:general-positive} and
\cref{lem:Phi-scale-comparison} with \(c=1\), \(C=4\),
\begin{equation}
 \bigl(\mathcal A_\Phi(m,m_1)\bigr)_+
 \leq C_*\kappa(L_j)L_jq_j
 \qquad\text{on }\mathcal H_j.
 \label{eq:large-comparable-positive-size}
\end{equation}
Here \(C_*>0\) is independent of \(N\), \(j\), and
\(t\in[0,T]\).

\medskip
\noindent\textbf{\underline{Step 2: pointwise comparison of collision frequencies.}}
Write \(x=(m,v)\), \(x_1=(m_1,v_1)\), and
\(z=(\mu,w)\in\mathcal B_T\).  Once
\(\vartheta_0L_j\geq R_T\), the bounds
\(m,m_1\in I_j\) imply
\begin{align}
 E(x,x_1)^\gamma
 &\leq(2L_j)^\gamma\abs{v-v_1}^{2\gamma},
 \label{eq:high-high-energy-upper}\\
 E(x,z)^\gamma
 &\geq\left(\frac{r_T}{2}\right)^\gamma
 \abs{v-w}^{2\gamma},
 \qquad
 E(x_1,z)^\gamma
 \geq\left(\frac{r_T}{2}\right)^\gamma
 \abs{v_1-w}^{2\gamma}.
 \label{eq:high-bounded-energy-lower}
\end{align}
Indeed, \(m\geq\vartheta_0L_j\geq R_T\geq\mu\) gives
\(
 {m\mu}/{(m+\mu)}\geq{\mu}/{2}\geq{r_T}/{2}.
\)
Since
\begin{equation}
 \abs{v-v_1}^{2\gamma}
 \leq2^{(2\gamma-1)_+}
 \bigl(\abs{v-w}^{2\gamma}+\abs{v_1-w}^{2\gamma}\bigr),
 \label{eq:three-velocity-inequality}
\end{equation}
\eqref{eq:high-high-energy-upper}--
\eqref{eq:three-velocity-inequality} give
\begin{equation}
 E(x,x_1)^\gamma
 \leq K_TL_j^\gamma
 \bigl(E(x,z)^\gamma+E(x_1,z)^\gamma\bigr).
 \label{eq:three-point-energy-comparison}
\end{equation}

Since \(m,m_1\geq\vartheta_0L_j\geq R_T\geq\mu\) and \(\chi_N\) is
nonincreasing,
\(
 \chi_N(m+m_1)\leq\chi_N(m+\mu),
 \chi_N(m+m_1)\leq\chi_N(m_1+\mu).
\)
Consequently,
\begin{align}
 \chi_N(m+m_1)E(x,x_1)^\gamma
 \leq K_TL_j^\gamma
 \bigl[
 \chi_N(m+\mu)E(x,z)^\gamma
 +\chi_N(m_1+\mu)E(x_1,z)^\gamma
 \bigr].
 \label{eq:pointwise-collision-comparison}
\end{align}

\medskip
\noindent\textbf{\underline{Step 3: averaging over bounded collision partners.}}
Fix \(x=(m,v)\) and \(x_1=(m_1,v_1)\) with
\(m,m_1\in I_j\).  Since the left-hand side of
\eqref{eq:pointwise-collision-comparison} is nonnegative, multiplying
\eqref{eq:pointwise-collision-comparison} by
\(f_N(t,z)/\eta_T\) and integrating over \(z\in\mathcal B_T\),
\cref{cor:bounded-mass-set} gives
\begin{align}
 &\chi_N(m+m_1)E(x,x_1)^\gamma
 \leq
 \frac{\chi_N(m+m_1)E(x,x_1)^\gamma}{\eta_T}
 \int_{\mathcal B_T}f_N(t,z)\,dz
 \notag\\
 &\quad\leq
 \frac{K_TL_j^\gamma}{\eta_T}
 \int_{\mathcal B_T}
 \Bigl[
   \chi_N(m+\mu)E(x,z)^\gamma
   +\chi_N(m_1+\mu)E(x_1,z)^\gamma
 \Bigr]f_N(t,z)\,dz.
 \label{eq:averaged-collision-comparison}
\end{align}
We now multiply this inequality by
\(f_N(t,x)f_N(t,x_1)\) and integrate over all
\(x,x_1\) whose mass components belong to \(I_j\).  Since every
integrand is nonnegative, Tonelli's theorem gives
\begin{align}
 &\int_{\{m,m_1\in I_j\}}
 \chi_N(S)E(x,x_1)^\gamma
 f_N(t,x)f_N(t,x_1)\,dx_1\,dx
 \notag\\
 &\quad\leq
 \frac{K_TL_j^\gamma}{\eta_T}
 \int_{\substack{m,m_1\in I_j\\z\in\mathcal B_T}}
 \chi_N(m+\mu)E(x,z)^\gamma
 f_N(t,x)f_N(t,x_1)f_N(t,z)\,dz\,dx_1\,dx
 \notag\\
 &\qquad+
 \frac{K_TL_j^\gamma}{\eta_T}
 \int_{\substack{m,m_1\in I_j\\z\in\mathcal B_T}}
 \chi_N(m_1+\mu)E(x_1,z)^\gamma
 f_N(t,x)f_N(t,x_1)f_N(t,z)\,dz\,dx_1\,dx.
 \label{eq:threefold-collision-comparison}
\end{align}
In the first integral on the right-hand side, the integrand involving
\(x\) and \(z\) is independent of \(x_1\).  Hence it factors as
\begin{align}
 &\int_{\substack{m,m_1\in I_j\\z\in\mathcal B_T}}
 \chi_N(m+\mu)E(x,z)^\gamma
 f_N(t,x)f_N(t,x_1)f_N(t,z)\,dz\,dx_1\,dx
 \notag\\
 &\quad=
 \left(\int_{\{m_1\in I_j\}}f_N(t,x_1)\,dx_1\right)
 \int_{\substack{m\in I_j\\z\in\mathcal B_T}}
 \chi_N(m+\mu)E(x,z)^\gamma
 f_N(t,x)f_N(t,z)\,dz\,dx.
 \label{eq:first-threefold-factorization}
\end{align}
Similarly, the second integral in
\eqref{eq:threefold-collision-comparison} factors as
\begin{align}
 &\left(\int_{\{m\in I_j\}}f_N(t,x)\,dx\right)
 \int_{\substack{m_1\in I_j\\z\in\mathcal B_T}}
 \chi_N(m_1+\mu)E(x_1,z)^\gamma
 f_N(t,x_1)f_N(t,z)\,dz\,dx_1.
\end{align}
Relabelling \(x_1\) as \(x\) in the last integral shows that the two
terms are equal.  Consequently,
\begin{align}
 &\int_{\{m,m_1\in I_j\}}\chi_N(S)E(x,x_1)^\gamma
 f_N(t,x)f_N(t,x_1)\,dx_1\,dx
 \notag\\
 &\quad\leq
 \frac{2K_TL_j^\gamma}{\eta_T}
 \left(\int_{\{m\in I_j\}}f_N(t,x)\,dx\right)
 \int_{\substack{m\in I_j\\z\in\mathcal B_T}}
 \chi_N(m+\mu)E(x,z)^\gamma
 f_N(t,x)f_N(t,z)\,dz\,dx.
 \label{eq:integrated-collision-comparison}
\end{align}
Besides, mass conservation gives
\begin{equation}
 \int_{\{m\in I_j\}}f_N(t,x)\,dx
 \leq\frac{M_1(f_0)}{\vartheta_0L_j}.
 \label{eq:large-interval-number-bound}
\end{equation}
Therefore
\begin{align}
 &\int_{\{m,m_1\in I_j\}}\chi_N(S)E^\gamma
 f_N(t,x)f_N(t,x_1)\,dx_1\,dx
 \notag\\
 &\quad\leq K_TL_j^{\gamma-1}
 \int_{\substack{m\in I_j\\x_1\in\mathcal B_T}}
 \chi_N(S)E^\gamma
 f_N(t,x)f_N(t,x_1)\,dx_1\,dx.
 \label{eq:collision-frequency-comparison}
\end{align}

\medskip
\noindent\textbf{\underline{Step 4: the negative increment and conclusion.}}
It remains to compare the signed increments
\(\mathcal A_\Phi(m,m_1)\).  If
\(m\in I_j\) and \(x_1\in\mathcal B_T\), then, for all sufficiently
large \(j\),
\(
 S=m+m_1\in[\vartheta_0L_j,3L_j]
\)
and
\(
 \vartheta
 \leq{R_T}/(\vartheta_0L_j+R_T)\leq\vartheta_0.
\)
Thus \eqref{eq:general-negative} and
\cref{lem:Phi-scale-comparison} with \(c=\vartheta_0\), \(C=3\) give
a constant \(c_*>0\), independent of \(N\), \(j\), and
\(t\in[0,T]\), such that
\begin{equation}
 -\mathcal A_\Phi(m,m_1)
 \geq c_*\kappa(L_j)L_jq_j
 \qquad
 \text{for }m\in I_j,\ x_1\in\mathcal B_T.
 \label{eq:large-bounded-negative-size}
\end{equation}
In particular, the integrand on the right-hand side of
\eqref{eq:large-pair-comparison} is nonnegative.  Combining
\eqref{eq:large-comparable-positive-size},
\eqref{eq:collision-frequency-comparison}, and
\eqref{eq:large-bounded-negative-size}, and enlarging \(K_T\) to
absorb \(C_*/c_*\), yields
\begin{align}
 &\int_{\mathcal H_j}\chi_N(S)E^\gamma
 \bigl(\mathcal A_\Phi(m,m_1)\bigr)_+
 f_N(t,x)f_N(t,x_1)\,dx_1\,dx
 \notag\\
 &\quad\leq
 C_*\kappa(L_j)L_jq_j
 \int_{\{m,m_1\in I_j\}}\chi_N(S)E^\gamma
 f_N(t,x)f_N(t,x_1)\,dx_1\,dx
 \notag\\
 &\quad\leq
 K_TL_j^{\gamma-1}\kappa(L_j)L_jq_j
 \int_{\substack{m\in I_j\\x_1\in\mathcal B_T}}
 \chi_N(S)E^\gamma
 f_N(t,x)f_N(t,x_1)\,dx_1\,dx
 \notag\\
 &\quad\leq
 K_TL_j^{\gamma-1}
 \int_{\substack{m\in I_j\\x_1\in\mathcal B_T}}
 \chi_N(S)E^\gamma
 \bigl[-\mathcal A_\Phi(m,m_1)\bigr]
 f_N(t,x)f_N(t,x_1)\,dx_1\,dx.
 \label{eq:large-pair-comparison-expanded}
\end{align}
This is \eqref{eq:large-pair-comparison}.
\end{proof}

In the following estimates of Class IV, we need to divide $\Omega_{\rm IV}$ into the disjoint shells $\mathcal H_j$ in \cref{lem:large-pair-comparison}. Although \(\mathcal H_j\) are disjoint, the comparison
regions \(I_j\times\mathcal B_T\) appearing on the right-hand side of
\eqref{eq:large-pair-comparison} may overlap.  Thus, when the
shellwise estimates are summed, the same large--bounded collision may
be used for several different shells.  Luckily, since 
\(m\in I_j\) implies 
\(\log_2({m}/{L_0})-1
 \leq j\leq
 \log_2({m}/{\vartheta_0L_0}),
\)
\begin{equation}
 \sum_{j=0}^{\infty}\one_{\{m\in I_j\}}
 \leq2+\left\lceil\log_2\frac2{\vartheta_0}\right\rceil.
 \label{eq:comparison-interval-overlap}
\end{equation}
Hence each fixed mass belongs to
only finitely many of the intervals \(I_j\). Consequently, the sum of all
comparison integrals can later be controlled by a fixed multiple of
one retained negative contribution from Class II.   

Fix a constant \(K_T\) for which
\eqref{eq:large-pair-comparison} holds for every sufficiently large
\(j\).  We can now specify the choice of \(J\). Choose \(J=J(T)\) so
large that \cref{lem:large-pair-comparison} holds for every
\(j\geq J\) and
\begin{align}
 \vartheta_0L_J&\geq\max\{R_T,L_1\},\quad
 \frac{R_T}{\vartheta_0L_J+R_T}\leq\vartheta_0,
 \quad K_T\left(2+\left\lceil\log_2\frac2{\vartheta_0}\right\rceil\right)
 L_J^{\gamma-1}\leq\frac12.
 \label{eq:J-absorption-condition}
\end{align}
Such a choice is possible because \(L_J\to\infty\) and
\(\gamma<1\).

Now we estimate the four areas listed in
\cref{tab:four-Phi-classes} respectively.  In the estimates below,
\(C_T\) may change from line to line and is independent of \(N\) and
\(t\in[0,T]\).

\noindent\textbf{\underline{Class I: bounded incoming masses.}}
If \((x,x_1)\in\Omega_{\mathrm I}\), then
\(S<2\vartheta_0L_J\).  From
\eqref{eq:signed-share-identity}, \(D_\alpha\geq0\), and
\(D_\vartheta(S)\leq\Phi(S)\),
\begin{align}
 \bigl(\mathcal A_\Phi(m,m_1)\bigr)_+
 \leq\lambda(m,m_1)D_\vartheta(S)
 \leq\lambda_+\kappa_+(1+2\vartheta_0L_J)
 \Phi(2\vartheta_0L_J)
 \leq C_T.
 \label{eq:class-I-pointwise}
\end{align}
Since \(\chi_N\leq1\) and \eqref{eq:unweighted-rate} holds uniformly
in \(N\),
\begin{align}
 &\int_{\Omega_{\mathrm I}}\chi_N(S)E^\gamma
 \bigl(\mathcal A_\Phi(m,m_1)\bigr)_+
 f_N(t,x)f_N(t,x_1)\,dx_1\,dx
 \notag\\
 &\quad\leq C_T\int_{X^2}E^\gamma
 f_N(t,x)f_N(t,x_1)\,dx_1\,dx
 \leq C_T2^\gamma
 M_0(f_0)^{2-\gamma}M_2(f_0)^\gamma.
 \label{eq:class-I-bound}
\end{align}
\noindent\textbf{\underline{Class II: unequal incoming masses.}}
On \(\Omega_{\mathrm {II}}\), one has
\(S\geq\vartheta_0L_J\geq L_1\) and
\(\vartheta\leq\vartheta_0\).  Hence
\eqref{eq:general-negative} gives
\begin{equation}
 \mathcal A_\Phi(m,m_1)
 \leq-\frac{G_1}{4}S q_{K(S)}
 \lambda(m,m_1)\leq0.
 \label{eq:class-II-negative}
\end{equation}
The Class IV comparison uses part of this negative contribution
where the first particle has mass at least \(\vartheta_0L_J\) and the
second belongs to \(\mathcal B_T\).  Denote its nonnegative magnitude by
\begin{align}
 \mathcal D_N(t)
 :=\int_{\substack{m\geq\vartheta_0L_J\\x_1\in\mathcal B_T}}
 \chi_N(S)E^\gamma
 \bigl[-\mathcal A_\Phi(m,m_1)\bigr]
 f_N(t,x)f_N(t,x_1)\,dx_1\,dx.
 \label{eq:retained-negative-definition}
\end{align}
For each fixed \(N\), the quantity in
\eqref{eq:retained-negative-definition} is finite.  Indeed,
\(\chi_N(S)\neq0\) implies \(S<2N\), and
\begin{equation}
 \chi_N(S)\abs{\mathcal A_\Phi(m,m_1)}
 \leq
 4\lambda_+\kappa_+(1+2N)
 \sup_{0\leq r\leq2N}\Phi(r).
 \label{eq:retained-negative-finiteness}
\end{equation}
Thus \(\mathcal D_N(t)<\infty\) by
\eqref{eq:unweighted-rate}.
Moreover, if \(m\geq\vartheta_0L_J\) and
\(x_1\in\mathcal B_T\), then
\(m_1\leq R_T\leq m\) and
\(
 \vartheta
 \leq{R_T}/({\vartheta_0L_J+R_T})
 \leq\vartheta_0.\)
Thus the integration region in
\eqref{eq:retained-negative-definition} is contained in
\(\Omega_{\mathrm {II}}\), and
\begin{align}
 &\int_{\Omega_{\mathrm {II}}}\chi_N(S)E^\gamma
 \mathcal A_\Phi(m,m_1)
 f_N(t,x)f_N(t,x_1)\,dx_1\,dx
 \leq -\mathcal D_N(t).
 \label{eq:class-II-bound}
\end{align}

\noindent\textbf{\underline{Class III: bounded comparable masses.}}
On \(\Omega_{\mathrm {III}}\), the sign of
\(\mathcal A_\Phi\) is not fixed, but \(M<L_J\) and hence
\(S<2L_J\).  As in \eqref{eq:class-I-pointwise},
\begin{equation}
 \bigl(\mathcal A_\Phi(m,m_1)\bigr)_+
 \leq\lambda_+\kappa_+(1+2L_J)\Phi(2L_J)
 \leq C_T.
 \label{eq:class-III-pointwise}
\end{equation}
Therefore
\begin{align}
 &\int_{\Omega_{\mathrm {III}}}\chi_N(S)E^\gamma
 \mathcal A_\Phi(m,m_1)
 f_N(t,x)f_N(t,x_1)\,dx_1\,dx
 \notag\\
 &\quad\leq
 C_T\int_{X^2}E^\gamma
 f_N(t,x)f_N(t,x_1)\,dx_1\,dx
 \leq C_T2^\gamma
 M_0(f_0)^{2-\gamma}M_2(f_0)^\gamma.
 \label{eq:class-III-bound}
\end{align}
\noindent\textbf{\underline{Class IV: large comparable masses.}}
Since the shells \(\mathcal H_j\) are defined by
\(L_j\leq M<L_{j+1}\), they are pairwise disjoint and
\begin{equation*}
 \displaystyle \Omega_{\mathrm {IV}}
 =\bigsqcup_{j=J}^{\infty}\mathcal H_j.
\end{equation*}
We may therefore sum the positive contribution shell by shell.
Applying \cref{lem:large-pair-comparison} on each shell gives
\begin{align}
 &\int_{\Omega_{\mathrm {IV}}}\chi_N(S)E^\gamma
 \mathcal A_\Phi(m,m_1)
 f_N(t,x)f_N(t,x_1)\,dx_1\,dx
 \notag\\
 &\quad\leq
 \sum_{j=J}^{\infty}
 \int_{\mathcal H_j}\chi_N(S)E^\gamma
 \bigl(\mathcal A_\Phi(m,m_1)\bigr)_+
 f_N(t,x)f_N(t,x_1)\,dx_1\,dx
 \notag\\
 &\quad\leq K_T\sum_{j=J}^{\infty}L_j^{\gamma-1}
 \int_{\substack{m\in I_j\\x_1\in\mathcal B_T}}
 \chi_N(S)E^\gamma
 \bigl[-\mathcal A_\Phi(m,m_1)\bigr]
 f_N(t,x)f_N(t,x_1)\,dx_1\,dx.
 \label{eq:class-IV-shell-bound}
\end{align}
Because \(\gamma<1\), one has
\(L_j^{\gamma-1}\leq L_J^{\gamma-1}\) for every \(j\geq J\).
Moreover, \(I_j=[\vartheta_0L_j,2L_j]\) and
\(L_{j+1}=2L_j\), so consecutive intervals overlap and
\begin{equation}
 \bigcup_{j\geq J}I_j=[\vartheta_0L_J,\infty).
 \label{eq:comparison-interval-union}
\end{equation}
The integrand in the last line of
\eqref{eq:class-IV-shell-bound} is nonnegative.  Hence Tonelli's
theorem, \eqref{eq:comparison-interval-overlap}, and
\eqref{eq:comparison-interval-union} give
\begin{align}
 &\sum_{j=J}^{\infty}
 \int_{\substack{m\in I_j\\x_1\in\mathcal B_T}}
 \chi_N(S)E^\gamma
 \bigl[-\mathcal A_\Phi(m,m_1)\bigr]
 f_N(t,x)f_N(t,x_1)\,dx_1\,dx
 \notag\\
 &\qquad=\int_{\substack{m\geq\vartheta_0L_J\\x_1\in\mathcal B_T}}
 \left(\sum_{j=J}^{\infty}\one_{\{m\in I_j\}}\right)
 \chi_N(S)E^\gamma
 \bigl[-\mathcal A_\Phi(m,m_1)\bigr]
 f_N(t,x)f_N(t,x_1)\,dx_1\,dx
 \notag\\
 &\qquad\leq
 \left(2+\left\lceil\log_2\frac2{\vartheta_0}\right\rceil\right)
 \mathcal D_N(t).
 \label{eq:class-IV-overlap-sum}
\end{align}
Combining the preceding two estimates with
\eqref{eq:J-absorption-condition}, we obtain
\begin{align}
 &\int_{\Omega_{\mathrm {IV}}}\chi_N(S)E^\gamma
 \mathcal A_\Phi(m,m_1)
 f_N(t,x)f_N(t,x_1)\,dx_1\,dx
 \notag\\
 &\quad\leq
 K_T\left(2+\left\lceil\log_2\frac2{\vartheta_0}\right\rceil\right)
 L_J^{\gamma-1}\mathcal D_N(t)
 \leq\frac12\mathcal D_N(t).
 \label{eq:class-IV-bound}
\end{align}

The four classes can now be combined in the derivative identity.
\begin{proposition}[Uniform bound for the signed
\(\Phi\)-production]
\label{prop:signed-Phi-production}
For every \(T>0\), there is \(C_T<\infty\), independent of
\(N\) and \(t\in[0,T]\), such that
\begin{equation}
 \int_{X^2}\chi_N(S)E^\gamma\mathcal A_\Phi(m,m_1)
 f_N(t,x)f_N(t,x_1)\,dx_1\,dx
 \leq C_T.
 \label{eq:signed-Phi-production-bound}
\end{equation}
\end{proposition}

\begin{proof}
The partition
\eqref{eq:class-I}--\eqref{eq:class-IV} and
\eqref{eq:class-I-bound}, \eqref{eq:class-II-bound},
\eqref{eq:class-III-bound}, and \eqref{eq:class-IV-bound} give
\begin{align}
 \int_{X^2}\chi_N(S)E^\gamma\mathcal A_\Phi
 f_N(t,x)f_N(t,x_1)\,dx_1\,dx\leq C_T-\mathcal D_N(t)+C_T
 +\frac12\mathcal D_N(t)
 \leq2C_T.
 \label{eq:four-class-sum}
\end{align}
Renaming \(2C_T\) as \(C_T\) proves
\eqref{eq:signed-Phi-production-bound}.
\end{proof}

\begin{proof}[Proof of \cref{thm:Phi-propagation}]
By \cref{lem:Phi-evolution,prop:signed-Phi-production}, for almost every
\(t\in[0,T]\),
\begin{equation*}
 \frac{d}{dt}\int_X\Phi(m)f_N(t,x)\,dx
 \leq\frac12\norm{b}_{L^1(\Sph)}C_T.
\end{equation*}
Integrating from \(0\) to \(t\leq T\) gives
\begin{equation*}
 \int_X\Phi(m)f_N(t,x)\,dx
 \leq\int_X\Phi(m)f_0(x)\,dx
 +\frac{T}{2}\norm{b}_{L^1}C_T.
\end{equation*}
Taking the supremum over \(N\geq1\) and \(0\leq t\leq T\), and
absorbing the fixed factor into \(C_T\), proves
\eqref{eq:Phi-propagation}.
\end{proof}

\begin{proof}[Proof of \cref{prop:no-gelation}]
If \(M_0(f_0)=0\), then \(f_N\equiv0\), and the assertion is
immediate.  Assume \(M_0(f_0)>0\).  Since
\(m\mapsto\Phi(m)/m\) is nondecreasing,
\begin{align}
 \int_{\{m>R\}}m f_N(t,x)\,dx
 &\leq
 \left(\inf_{m>R}\frac{\Phi(m)}m\right)^{-1}
 \int_X\Phi(m)f_N(t,x)\,dx.
 \label{eq:Phi-tail-control}
\end{align}
By \cref{thm:Phi-propagation}, the last \(\Phi\)-moment is bounded
uniformly for \(N\geq1\) and \(0\leq t\leq T\), while
\begin{equation*}
 \inf_{m>R}\frac{\Phi(m)}m\longrightarrow\infty
 \qquad\text{as }R\to\infty
\end{equation*}
by \eqref{eq:Phi-superlinear}.  Taking the required suprema in
\eqref{eq:Phi-tail-control} and then letting \(R\to\infty\) gives
\begin{equation}
 \lim_{R\to\infty}\sup_{N\geq1}\sup_{0\leq t\leq T}
 \int_{\{m>R\}}m f_N(t,x)\,dx=0.
\end{equation}
This proves \cref{prop:no-gelation}.
\end{proof}

\section{Passage to the Limit and Global Existence}
\label{sec:passage-limit}
\subsection{Collision-rate tightness and narrow compactness}

\begin{lemma}[Localized collision-rate bound]
\label{lem:localized-collision-rate}
Let \(\mu\) be a finite nonnegative Borel measure on \(X\) with finite
\(M_0(\mu),M_1(\mu),M_2(\mu)\).  For every Borel set \(A\subset X\),
\begin{align}
 &\int_{A\times X}\int_0^1\int_{\Sph}
 a(m,m_1,\alpha)E^\gamma
 b\left(\frac{u}{\abs u}\cdot\omega\right)
 \,d\omega\,d\alpha\,d\mu(x_1)\,d\mu(x)
 \notag\\
 &\quad\leq a_+\norm b_{L^1(\Sph)}
 \Biggl[
 \bigl(M_0(\mu)\mu(A)\bigr)^{1-\gamma}
 \bigl(2M_0(\mu)M_2(\mu)\bigr)^\gamma
 \notag\\[-1mm]
 &\hspace{43mm}+
 \left(M_0(\mu)\int_A m\,d\mu+M_1(\mu)\mu(A)\right)^{1-\gamma}
 \bigl(2M_1(\mu)M_2(\mu)\bigr)^\gamma
 \Biggr].
 \label{eq:localized-collision-rate}
\end{align}
The same estimate holds with \(a_N\) in place of \(a\).
\end{lemma}

\begin{proof}
By \eqref{eq:model-linear-upper}, the left-hand side is bounded by
\begin{equation*}
 a_+\norm b_{L^1(\Sph)}
 \int_{A\times X}(1+S)E^\gamma\,d\mu(x_1)\,d\mu(x).
\end{equation*}
For the part without \(S\), H\"older's inequality on \(A\times X\)
gives
\begin{equation}
 \int_{A\times X}E^\gamma\,d\mu\,d\mu_1
 \leq\bigl(M_0(\mu)\mu(A)\bigr)^{1-\gamma}
 \left(\int_{A\times X}E\,d\mu\,d\mu_1\right)^\gamma,
\end{equation}
where the last integral is at most \(2M_0(\mu)M_2(\mu)\), by
\eqref{eq:model-reduced-energy}.  For the part with \(S\), write
\(SE^\gamma=S^{1-\gamma}(SE)^\gamma\).  A second application of
H\"older's inequality gives
\begin{align*}
 \int_{A\times X}SE^\gamma\,d\mu\,d\mu_1
 &\leq
 \left(\int_{A\times X}S\,d\mu\,d\mu_1\right)^{1-\gamma}
 \left(\int_{A\times X}SE\,d\mu\,d\mu_1\right)^\gamma\\
 &\leq
 \left(M_0(\mu)\int_A m\,d\mu+M_1(\mu)\mu(A)\right)^{1-\gamma}
 \bigl(2M_1(\mu)M_2(\mu)\bigr)^\gamma.
\end{align*}
Here the first factor was obtained by writing \(S=m+m_1\), while the
second was enlarged to \(X^2\) and then estimated by
\eqref{eq:joint-rate-identities}.  This proves
\eqref{eq:localized-collision-rate}.  The assertion for \(a_N\)
follows from \eqref{eq:cutoff-domination}.
\end{proof}

We denote by \(\mathcal M_+(X)\) the space of finite nonnegative Radon
measures on \(X\), endowed with the narrow topology. For \(T>0\), convergence
\(\mu_n\to\mu\) in \(C([0,T];\mathcal M_+(X)\text{-narrow})\)
means that
\begin{equation*}
 \sup_{0\leq t\leq T}
 \left|
 \int_X\phi\,d\mu_n(t)-\int_X\phi\,d\mu(t)
 \right|
 \longrightarrow0
 \qquad\text{for every }\phi\in C_b(X).
\end{equation*}
In this section, we regard \(f_N(t,x)\,dx\in\mathcal M_+(X)\) as a finite nonnegative Radon
measure and write it as \(\mu_N(t)\). 
For every Borel set \(D\subset X^2\), define
\begin{align}
 \mathcal R_N(D,t)
 :=\int_D\int_0^1\int_{\Sph}
 &a_NE^\gamma b\left(\frac{u}{\abs u}\cdot\omega\right)
 \,d\omega\,d\alpha\,
 d\mu_N(t)(x_1)\,d\mu_N(t)(x).
 \label{eq:truncated-collision-rate}
\end{align}
For a curve
\(\mu=(\mu_t)_{t\geq0}\in
C([0,\infty);\mathcal M_+(X)\text{-narrow})\), define
\begin{align}
 \mathcal R(D,t)
 :=\int_D\int_0^1\int_{\Sph}
 &a(m,m_1,\alpha)E(x,x_1)^\gamma
 b\left(\frac{u}{\abs u}\cdot\omega\right)
 \,d\omega\,d\alpha\,
 d\mu_t(x_1)\,d\mu_t(x).
 \label{eq:limit-collision-rate}
\end{align}
Finally, for \(0<r<R<\infty\), set
\begin{equation}
 K_{r,R}
 :=\{(m,v)\in X:r\leq m\leq R,\ \abs v\leq R\}.
\end{equation}
Thus
\begin{equation*}
 X^2\setminus K_{r,R}^2
 =
 \{(x,x_1)\in X^2:x\notin K_{r,R}
 \text{ or }x_1\notin K_{r,R}\}.
\end{equation*}

\begin{proposition}[Global narrow limit and moment passage]
\label{prop:global-narrow-limit}
There exist a strictly increasing sequence \(N_k\) and a narrowly
continuous curve \(t\mapsto\mu_t\) of finite nonnegative Radon measures
on \(X\) such that, for every \(T>0\),
\begin{equation}
 \mu_{N_k}\longrightarrow\mu
 \quad\hbox{in }C([0,T];\mathcal M_+(X)\text{-narrow}).
 \label{eq:narrow-compact-time}
\end{equation}
Moreover, \(\mu_0=f_0(x)\,dx\), and for every \(t\geq0\),
\begin{equation}
 M_0(\mu_t)=M_0(f_0),
 \qquad
 M_1(\mu_t)=M_1(f_0),
 \qquad
 M_2(\mu_t)\leq M_2(f_0).
 \label{eq:limit-physical-moments}
\end{equation}
For every \(T>0\),
\begin{equation}
 \sup_{0\leq t\leq T}\int_X\Phi(m)\,d\mu_t(x)
 \leq
 \sup_{N\geq1}\sup_{0\leq t\leq T}
 \int_X\Phi(m)f_N(t,x)\,dx<\infty,
 \label{eq:limit-Phi-bound}
\end{equation}
and hence
\begin{equation}
 \lim_{R\to\infty}\sup_{0\leq t\leq T}
 \int_{\{m>R\}}m\,d\mu_t(x)=0.
 \label{eq:measure-limit-no-gelation}
\end{equation}
Finally, 
\begin{equation}
 \lim_{r\downarrow0}\lim_{R\uparrow\infty}
 \sup_{0\leq t\leq T}
 \left[
 \sup_{N\geq1}\mathcal R_N\bigl(X^2\setminus K_{r,R}^2,t\bigr)
 +\mathcal R\bigl(X^2\setminus K_{r,R}^2,t\bigr)
 \right]=0.
 \label{eq:collision-rate-tightness}
\end{equation}
\end{proposition}
\begin{proof}
Fix \(T>0\).

\noindent\underline{\textbf{Step 1: uniform tightness of the one-particle
measures.}}
Fix \(\varepsilon>0\). By \eqref{eq:small-mass-uniform}, choose
\(r\in(0,1), R>r\) such that
\begin{equation*}\sup_{N\geq1}\sup_{0\leq t\leq T}F_N(r,t)<\frac{\varepsilon}{3},\quad \frac{M_1(f_0)}{R}<\frac{\varepsilon}{3},\quad \frac{M_2(f_0)}{rR^2}<\frac{\varepsilon}{3}.\end{equation*}

Since
\begin{equation*}K_{r,R}^c\subset\{m<r\}\cup\{m>R\}
\cup\{m\geq r,\ \abs v>R\},\end{equation*}
conservation of mass and kinetic energy
gives
\begin{align}
 \mu_N(t)(K_{r,R}^c)
 &\leq F_N(r,t)
 +\frac1R\int_Xm\,d\mu_N(t)
 +\frac1{rR^2}\int_Xm|v|^2\,d\mu_N(t) \notag\\
 &\leq F_N(r,t)+\frac{M_1(f_0)}R
 +\frac{M_2(f_0)}{rR^2}.
\end{align}
Thus
\begin{align}
 \sup_{N\ge 1}\sup_{0\leq t\leq T}\mu_N(t)(K_{r,R}^c)
 <\varepsilon.
 \label{eq:uniform-tightness-muN}
\end{align}
Hence
\(\{\mu_N(t):N\geq1,\ 0\leq t\leq T\}\) is uniformly tight.
Since \(\mu_N(t)(X)=M_0(f_0)\), Prokhorov's theorem yields relative
compactness in \(\mathcal M_+(X)\) endowed with the narrow topology.

\noindent\underline{\textbf{Step 2: compactness in global time.}}
For \(\nu,\widetilde\nu\in\mathcal M_+(X)\), define the bounded--Lipschitz distance
\begin{equation}
 d_{\mathrm{BL}}(\nu,\widetilde\nu)
 :=
 \sup_{\substack{
                  \norm{\psi}_{\infty}\leq1\\
                  \operatorname{Lip}(\psi)\leq1}}
 \left|\int_X\psi\,d\nu-\int_X\psi\,d\widetilde\nu\right|,
\end{equation}
which metrizes the narrow topology on \(\mathcal M_+(X)\). Since \(\mu_N(t)=f_N(t,x)\,dx\), for \(0\leq s,t\leq T\),
\begin{align}
 d_{\mathrm{BL}}\bigl(\mu_N(t),\mu_N(s)\bigr)
 &=
 \sup_{\substack{\norm{\psi}_{\infty}\leq1\\
                  \operatorname{Lip}(\psi)\leq1}}
 \left|
 \int_X\psi(x)\bigl(f_N(t,x)-f_N(s,x)\bigr)\,dx
 \right| \notag\\
 &\leq\norm{f_N(t)-f_N(s)}_{L^1(X)}
 \leq2C_{\mathrm{rate}}\abs{t-s},
 \label{eq:muN-equicontinuity}
\end{align}
where the last inequality follows from
\eqref{eq:uniform-time-Lipschitz}. Moreover, $\mu_N(t)(X)=M_0(f_0)$ and by Step~1,
\begin{equation}
 \mathcal K_T
 :=
 \overline{
 \{\mu_N(t):N\geq1,\ 0\leq t\leq T\}
 }^{\,\mathrm{narrow}}
 \quad\text{is compact in }
 (\mathcal M_+(X),d_{\mathrm{BL}}),
 \label{eq:compact-range}
\end{equation}
the metric-valued Arzelà--Ascoli theorem
give a subsequence \((N_j)\) and a curve
\(\mu:[0,T]\to\mathcal M_+(X)\) such that
\begin{equation}
  \lim_{j\to\infty}\sup_{0\leq t\leq T}
 d_{\mathrm{BL}}\bigl(\mu_{N_j}(t),\mu(t)\bigr)=0.
 \label{eq:uniform-BL-convergence}
\end{equation}
For \(0\leq s,t\leq T\), \eqref{eq:muN-equicontinuity} and
\eqref{eq:uniform-BL-convergence} yield
\begin{align}
 d_{\mathrm{BL}}\bigl(\mu(t),\mu(s)\bigr)
 &\leq
 d_{\mathrm{BL}}\bigl(\mu(t),\mu_{N_j}(t)\bigr)
 +d_{\mathrm{BL}}\bigl(\mu_{N_j}(t),\mu_{N_j}(s)\bigr)
 +d_{\mathrm{BL}}\bigl(\mu_{N_j}(s),\mu(s)\bigr)
 \notag\\
 &\leq
 d_{\mathrm{BL}}\bigl(\mu(t),\mu_{N_j}(t)\bigr)
 +2C_{\mathrm{rate}}\abs{t-s}
 +d_{\mathrm{BL}}\bigl(\mu_{N_j}(s),\mu(s)\bigr).
\end{align}
Letting \(j\to\infty\) gives
\(d_{\mathrm{BL}}(\mu(t),\mu(s))
\leq2C_{\mathrm{rate}}\abs{t-s}\); in particular,
\(\mu\in C([0,T];\mathcal M_+(X)\text{-narrow})\).

Now we upgrade \eqref{eq:uniform-BL-convergence} to uniform narrow
convergence against every test function in \(C_b(X)\). Fix
\(\psi\in C_b(X)\). Given
\(\varepsilon,\eta>0\), choose \(K_{r,R}\) such that
\eqref{eq:uniform-tightness-muN} holds. Since \(K_{r,R}\) is compact, there is a
bounded Lipschitz function \(g\) satisfying
\(\norm g_{\infty}\leq \|\psi\|_\infty\) and
\(\sup_{x\in K_{r,R}}\abs{\psi(x)-g(x)}<\eta\).
For every \(t\in[0,T]\), the Portmanteau theorem gives
\begin{equation}
 \mu(t)(K_{r,R}^c)
 \leq\liminf_{j\to\infty}\mu_{N_j}(t)(K_{r,R}^c)
 \leq\varepsilon,
 \qquad
 \mu(t)(X)
 =\lim_{j\to\infty}\mu_{N_j}(t)(X)
 =M_0(f_0).
\end{equation}
Consequently, denoting by $G=\max\{\norm g_{\infty},\operatorname{Lip}(g)\}$, inserting \(g\) and splitting the two error integrals over
\(K_{r,R}\) and \(K_{r,R}^c\), we obtain
\begin{align*}
 &\sup_{0\leq t\leq T}
 \left|
 \int_X\psi\,d\mu_{N_j}(t)-\int_X\psi\,d\mu(t)
 \right| \leq
 \sup_{0\leq t\leq T}
 \Bigg[
 \left|
 \int_Xg\,d\mu_{N_j}(t)-\int_Xg\,d\mu(t)
 \right|
 +\int_{K_{r,R}}\abs{\psi-g}\,d\mu_{N_j}(t)\notag\\ &
 \qquad\qquad\qquad\qquad\qquad\qquad+\int_{K_{r,R}}\abs{\psi-g}\,d\mu(t)
 +\int_{K_{r,R}^c}\abs{\psi-g}\,d\mu_{N_j}(t)
 +\int_{K_{r,R}^c}\abs{\psi-g}\,d\mu(t)
 \Bigg]
 \notag\\
 &\leq
 G
 \sup_{0\leq t\leq T}
 d_{\mathrm{BL}}\bigl(\mu_{N_j}(t),\mu(t)\bigr)
 +\eta\sup_{0\leq t\leq T}
 \bigl(\mu_{N_j}(t)(K_{r,R})+\mu(t)(K_{r,R})\bigr)
 \notag\\
 &
 +2\|\psi\|_\infty\sup_{0\leq t\leq T}
 \bigl(\mu_{N_j}(t)(K_{r,R}^c)
       +\mu(t)(K_{r,R}^c)\bigr)\leq
 G
 \sup_{0\leq t\leq T}
 d_{\mathrm{BL}}\bigl(\mu_{N_j}(t),\mu(t)\bigr)
 +2\eta M_0(f_0)+4\|\psi\|_\infty\varepsilon.
\end{align*}
First letting \(j\to\infty\), then \(\eta\downarrow0\) and
\(\varepsilon\downarrow0\), gives
\begin{equation}
 \lim_{j\to\infty}\sup_{0\leq t\leq T}
 \left|\int_X\psi\,d\mu_{N_j}(t)-\int_X\psi\,d\mu(t)\right|=0,
 \qquad \psi\in C_b(X).
 \label{eq:uniform-narrow-convergence}
\end{equation}
This proves \eqref{eq:narrow-compact-time}. Finally, apply the preceding extraction successively on
\([0,1],[0,2],\ldots\). Choose nested subsequences satisfying
\begin{equation}
 (N_j^{(k+1)})_{j\geq1}\subset(N_j^{(k)})_{j\geq1},
 \qquad
 \mu_{N_j^{(k)}}\longrightarrow\mu^{(k)}
 \quad\text{in }
 C([0,k];\mathcal M_+(X)\text{-narrow}).
\end{equation}
For every \(k\geq1\), uniqueness of the narrow limit gives
\(\mu^{(k+1)}|_{[0,k]}=\mu^{(k)}\). Setting
\(N_j=N_j^{(j)}\) and
\(\mu(t)=\mu^{(k)}(t)\) for \(0\leq t\leq k\), we obtain
\begin{equation}
 \mu_{N_j}\longrightarrow\mu
 \quad\text{in }
 C([0,T];\mathcal M_+(X)\text{-narrow})
\end{equation}
for every \(T>0\). Since \(\mu_{N_j}(0)=f_0(x)\,dx\), we also have
\begin{equation*}
 d_{\mathrm{BL}}\bigl(\mu(0),f_0(x)\,dx\bigr)
 \leq
 d_{\mathrm{BL}}\bigl(\mu(0),\mu_{N_j}(0)\bigr)
 \longrightarrow0,
\end{equation*}
and therefore \(\mu(0)=f_0(x)\,dx\).

\noindent\underline{\textbf{Step 3: passage of number, mass, and
energy.}}
Since \(1\in C_b(X)\), \eqref{eq:narrow-compact-time} and
\eqref{eq:truncated-moments} give
\begin{equation}
 M_0(\mu_t)
 =\lim_{k\to\infty}\int_X1\,d\mu_{N_k}(t)
 =M_0(f_0)
\end{equation}
for every \(t\in[0,T]\). We next pass to the total mass. For every \(R>0\), the function
\(m\mapsto m\wedge R\) belongs to \(C_b(X)\). Moreover,
\begin{align}
 0\leq M_1(f_0)-\int_X(m\wedge R)\,d\mu_N(t)=\int_{\{m>R\}}(m-R)\,d\mu_N(t)
 \leq\int_{\{m>R\}}m\,d\mu_N(t).
 \label{eq:truncated-mass-defect}
\end{align}
By \eqref{eq:narrow-compact-time},
\begin{equation*}
 \int_X(m\wedge R)\,d\mu_{N_k}(t)
 \longrightarrow
 \int_X(m\wedge R)\,d\mu_t
\end{equation*}
uniformly for \(t\in[0,T]\). Passing \(k\to\infty\) in
\eqref{eq:truncated-mass-defect} therefore yields
\begin{equation}
 0\leq
 M_1(f_0)-\int_X(m\wedge R)\,d\mu_t
 \leq
 \sup_{N\geq1}\sup_{0\leq s\leq T}
 \int_{\{m>R\}}m\,d\mu_N(s),
\end{equation}
where the right-hand side tends to zero as \(R\to\infty\) by
\cref{prop:no-gelation}. On the other hand,
\(m\wedge R\uparrow m\), so the monotone convergence theorem gives
\begin{equation}
 M_1(\mu_t)
 =\lim_{R\to\infty}\int_X(m\wedge R)\,d\mu_t
 =M_1(f_0).
\end{equation}
Finally, for every \(R>0\), the function
\((m|v|^2)\wedge R\in C_b(X)\). Hence
\eqref{eq:narrow-compact-time} and \eqref{eq:truncated-moments} give
\begin{align*}
 \int_X\bigl((m|v|^2)\wedge R\bigr)\,d\mu_t
 =\lim_{k\to\infty}
 \int_X\bigl((m|v|^2)\wedge R\bigr)\,d\mu_{N_k}(t)\leq\limsup_{k\to\infty}M_2(\mu_{N_k}(t))
 \leq M_2(f_0).
\end{align*}
Letting \(R\to\infty\) and applying monotone convergence once more,
we obtain
\begin{equation}
 M_2(\mu_t)
 =\int_Xm|v|^2\,d\mu_t
 \leq M_2(f_0).
\end{equation}
Since \(T>0\) was arbitrary, these conclusions hold for every
\(t\geq0\), proving \eqref{eq:limit-physical-moments}.

\noindent\underline{\textbf{Step 4: passage of the $\Phi$-moment.}}
Since $x\mapsto \Phi(m)$ is continuous and
nonnegative, for every \(t\in[0,T]\), the Portmanteau
theorem gives
\begin{align*}
 \int_X\Phi(m)\,d\mu_t
 \leq
 \liminf_{k\to\infty}
 \int_X\Phi(m)\,d\mu_{N_k}(t)\leq
 \sup_{N\geq1}\sup_{0\leq s\leq T}
 \int_X\Phi(m)f_N(s,x)\,dx.
\end{align*}
Taking the supremum over \(t\in[0,T]\) and using
\eqref{eq:Phi-propagation}, we obtain
\begin{equation}
 \sup_{0\leq t\leq T}\int_X\Phi(m)\,d\mu_t
 \leq
 \sup_{N\geq1}\sup_{0\leq t\leq T}
 \int_X\Phi(m)f_N(t,x)\,dx
 <\infty,
\end{equation}
which proves \eqref{eq:limit-Phi-bound}.

Since \(m\mapsto\Phi(m)/m\) is nondecreasing and
\(\Phi(m)/m\to\infty\), for every \(R>0\) and \(t\in[0,T]\),
\begin{align*}
 \int_{\{m>R\}}m\,d\mu_t
 \leq
 \left(\inf_{m>R}\frac{\Phi(m)}{m}\right)^{-1}
 \int_{\{m>R\}}\Phi(m)\,d\mu_t\leq
 \left(\inf_{m>R}\frac{\Phi(m)}{m}\right)^{-1}
 \sup_{0\leq s\leq T}\int_X\Phi(m)\,d\mu_s.
\end{align*}
Consequently,
\begin{equation}
 \sup_{0\leq t\leq T}\int_{\{m>R\}}m\,d\mu_t
 \leq
 \left(\inf_{m>R}\frac{\Phi(m)}{m}\right)^{-1}
 \sup_{0\leq t\leq T}\int_X\Phi(m)\,d\mu_t
 \longrightarrow0
\end{equation}
as \(R\to\infty\). This proves
\eqref{eq:measure-limit-no-gelation}.

\noindent\underline{\textbf{Step 5: tightness of the collision
rates.}}
We first prove the required one-particle tightness estimates for the
approximate measures. Since
\begin{equation*}
 K_{r,R}^c
 \subset
 \{m<r\}\cup\{m>R\}
 \cup\{m\geq r,\ \abs v>R\},
\end{equation*}
\eqref{eq:truncated-moments} gives
\begin{align}
 \mu_N(t)(K_{r,R}^c)
 &\leq F_N(r,t)
 +\frac1R\int_Xm\,d\mu_N(t)
 +\frac1{rR^2}\int_Xm|v|^2\,d\mu_N(t) \notag\\
 &\leq F_N(r,t)
 +\frac{M_1(f_0)}R
 +\frac{M_2(f_0)}{rR^2}.
 \label{eq:collision-number-tightness-N}
\end{align}
Similarly,
\begin{align}
 \int_{K_{r,R}^c}m\,d\mu_N(t)
 &\leq
 \int_{\{m<r\}}m\,d\mu_N(t)
 +\int_{\{m>R\}}m\,d\mu_N(t)
 +\int_{\{\abs v>R\}}m\,d\mu_N(t) \notag\\
 &\leq
 rM_0(f_0)
 +\int_{\{m>R\}}m\,d\mu_N(t)
 +\frac{M_2(f_0)}{R^2}.
 \label{eq:collision-mass-tightness-N}
\end{align}
Taking the supremum over \(N\geq1\) and \(t\in[0,T]\), then letting
\(R\to\infty\) and \(r\downarrow0\), equations
\eqref{eq:small-mass-uniform}, \eqref{eq:collision-number-tightness-N},
\eqref{eq:collision-mass-tightness-N}, and
\cref{prop:no-gelation} yield
\begin{align}
 &\lim_{r\downarrow0}\lim_{R\uparrow\infty}
 \sup_{N\geq1}\sup_{0\leq t\leq T}
 \mu_N(t)(K_{r,R}^c)=0,
 \label{eq:collision-number-tightness-N-limit}\\
 &\lim_{r\downarrow0}\lim_{R\uparrow\infty}
 \sup_{N\geq1}\sup_{0\leq t\leq T}
 \int_{K_{r,R}^c}m\,d\mu_N(t)=0.
 \label{eq:collision-mass-tightness-N-limit}
\end{align}

We now establish the corresponding estimates for the limit curve \(t\mapsto \mu_t\).
Since \(\{m<r\}\) is open, the Portmanteau theorem gives, for every
\(t\in[0,T]\),
\begin{equation}
 \mu_t(\{m<r\})
 \leq\liminf_{k\to\infty}
 \mu_{N_k}(t)(\{m<r\})
 \leq
 \sup_{N\geq1}\sup_{0\leq s\leq T}F_N(r,s).
\end{equation}
Using \eqref{eq:limit-physical-moments}, we therefore obtain
\begin{align}
 \mu_t(K_{r,R}^c)
 &\leq
 \mu_t(\{m<r\})
 +\frac1R\int_Xm\,d\mu_t
 +\frac1{rR^2}\int_Xm|v|^2\,d\mu_t \notag\\
 &\leq
 \sup_{N\geq1}\sup_{0\leq s\leq T}F_N(r,s)
 +\frac{M_1(f_0)}R
 +\frac{M_2(f_0)}{rR^2}.
 \label{eq:collision-number-tightness-limit}
\end{align}
Likewise,
\begin{align}
 \int_{K_{r,R}^c}m\,d\mu_t
 &\leq
 rM_0(\mu_t)
 +\int_{\{m>R\}}m\,d\mu_t
 +\frac1{R^2}M_2(\mu_t) \notag\\
 &\leq
 rM_0(f_0)
 +\int_{\{m>R\}}m\,d\mu_t
 +\frac{M_2(f_0)}{R^2}.
 \label{eq:collision-mass-tightness-limit}
\end{align}
It follows from \eqref{eq:small-mass-uniform} and
\eqref{eq:measure-limit-no-gelation} that
\begin{align}
 &\lim_{r\downarrow0}\lim_{R\uparrow\infty}
 \sup_{0\leq t\leq T}\mu_t(K_{r,R}^c)=0,
 \label{eq:collision-number-tightness-limit-curve}\\
 &\lim_{r\downarrow0}\lim_{R\uparrow\infty}
 \sup_{0\leq t\leq T}
 \int_{K_{r,R}^c}m\,d\mu_t=0.
 \label{eq:collision-mass-tightness-limit-curve}
\end{align}

Because
\(
 X^2\setminus K_{r,R}^2
 \subset
 (K_{r,R}^c\times X)\cup(X\times K_{r,R}^c),
\)
and the collision-rate measure is symmetric in \(x\) and \(x_1\), we have
\(
 \mathcal R_N(X^2\setminus K_{r,R}^2,t)
 \leq2\mathcal R_N(K_{r,R}^c\times X,t).
\)
Applying \cref{lem:localized-collision-rate} to
\(\mu_N(t)\) with \(A=K_{r,R}^c\), and then using
\eqref{eq:truncated-moments}, gives
\begin{align}
 &\mathcal R_N(X^2\setminus K_{r,R}^2,t) \leq
 2a_+\norm b_{L^1(\Sph)}
 \Biggl[
 \bigl(M_0(f_0)\mu_N(t)(K_{r,R}^c)\bigr)^{1-\gamma}
 \bigl(2M_0(f_0)M_2(f_0)\bigr)^\gamma \notag\\
 &\qquad\qquad+
 \left(
 M_0(f_0)\int_{K_{r,R}^c}m\,d\mu_N(t)
 +M_1(f_0)\mu_N(t)(K_{r,R}^c)
 \right)^{1-\gamma}
 \bigl(2M_1(f_0)M_2(f_0)\bigr)^\gamma
 \Biggr].
 \label{eq:localized-collision-tightness-N}
\end{align}
Since \(1-\gamma>0\), equations
\eqref{eq:collision-number-tightness-N-limit} and
\eqref{eq:collision-mass-tightness-N-limit} imply
\begin{equation}
 \lim_{r\downarrow0}\lim_{R\uparrow\infty}
 \sup_{N\geq1}\sup_{0\leq t\leq T}
 \mathcal R_N(X^2\setminus K_{r,R}^2,t)=0.
 \label{eq:collision-rate-tightness-N}
\end{equation}

The same argument applies to \(\mu_t\). Indeed,
\eqref{eq:limit-physical-moments} and
\cref{lem:localized-collision-rate} give
\begin{align}
 &\mathcal R(X^2\setminus K_{r,R}^2,t) \leq
 2a_+\norm b_{L^1(\Sph)}
 \Biggl[
 \bigl(M_0(f_0)\mu_t(K_{r,R}^c)\bigr)^{1-\gamma}
 \bigl(2M_0(f_0)M_2(f_0)\bigr)^\gamma \notag\\
 &\qquad\qquad+
 \left(
 M_0(f_0)\int_{K_{r,R}^c}m\,d\mu_t
 +M_1(f_0)\mu_t(K_{r,R}^c)
 \right)^{1-\gamma}
 \bigl(2M_1(f_0)M_2(f_0)\bigr)^\gamma
 \Biggr].
 \label{eq:localized-collision-tightness-limit}
\end{align}
Together with
\eqref{eq:collision-number-tightness-limit-curve} and
\eqref{eq:collision-mass-tightness-limit-curve}, this yields
\begin{equation}
 \lim_{r\downarrow0}\lim_{R\uparrow\infty}
 \sup_{0\leq t\leq T}
 \mathcal R(X^2\setminus K_{r,R}^2,t)=0.
 \label{eq:collision-rate-tightness-limit-curve}
\end{equation}
Adding the last two limits proves
\eqref{eq:collision-rate-tightness}.
\end{proof}

\subsection{Identification of the collision operator}
\begin{lemma}[Finite signed measure collision form]
\label{lem:measure-valued-collision-operator}
Let \(\mu\) be a finite nonnegative Radon measure on \(X\) satisfying
the physical-moment bounds in \eqref{eq:limit-physical-moments}.  Then,
for every bounded Borel function \(\psi:X\to\R\), the integral
\begin{align}
 \mathcal Q(\mu,\mu)[\psi]
 :=\frac12\int_{X^2}\int_0^1\int_{\Sph}
 &a(m,m_1,\alpha)E(x,x_1)^\gamma
 b\left(\frac{u}{\abs u}\cdot\omega\right)
 \Delta\psi
 d\omega\,d\alpha\,d\mu(x_1)\,d\mu(x)
 \label{eq:measure-static-collision-form}
\end{align}
is absolutely convergent.  Moreover, there exists a unique finite
signed Radon measure \(\mathbf Q(\mu,\mu)\) on \(X\) such that
\begin{equation}
 \int_X\psi\,d\mathbf Q(\mu,\mu)
 =\mathcal Q(\mu,\mu)[\psi]
 \qquad\text{for every bounded Borel }\psi.
 \label{eq:measure-collision-duality}
\end{equation}
More explicitly, for every Borel set \(A\subset X\),
\begin{align}
 \mathbf Q(\mu,\mu)(A)
 =\frac12\int_{X^2}\int_0^1\int_{\Sph}
 &a(m,m_1,\alpha)E(x,x_1)^\gamma
 b\left(\frac{u}{\abs u}\cdot\omega\right)
 \notag\\
 &\times
 \bigl[
 \one_A(x')+\one_A(x_1')
 -\one_A(x)-\one_A(x_1)
 \bigr]
 \,d\omega\,d\alpha\,d\mu(x_1)\,d\mu(x).
 \label{eq:measure-collision-set-form}
\end{align}
Its total variation satisfies
\begin{align}
 \norm{\mathbf Q(\mu,\mu)}_{\mathrm{TV}}
 \leq
 2a_+\norm{b}_{L^1(\Sph)}
 \int_{X^2}(1+S)E^\gamma
 \,d\mu(x_1)\,d\mu(x)\leq 2C_{\mathrm{rate}}.
 \label{eq:measure-collision-TV}
\end{align}
In particular,
\begin{equation}
 \abs{\mathcal Q(\mu,\mu)[\psi]}
 \leq
 2C_{\mathrm{rate}}\norm{\psi}_\infty
 \qquad\text{for every bounded Borel }\psi.
 \label{eq:measure-collision-bounded-test}
\end{equation}
\end{lemma}

\begin{proof}
Introduce the finite nonnegative measure on
\(X^2\times(0,1)\times\Sph\) given by
\begin{align}
 d\Lambda_\mu(x,x_1,\alpha,\omega)
 &:=
 a(m,m_1,\alpha)E(x,x_1)^\gamma
 b\left(\frac{u}{\abs u}\cdot\omega\right)
 \,d\omega\,d\alpha\,d\mu(x_1)\,d\mu(x).
 \label{eq:collision-parameter-measure}
\end{align}
 By
\eqref{eq:model-linear-upper} and
\eqref{eq:model-angular-cutoff},
\begin{align}
 \Lambda_\mu\bigl(X^2\times(0,1)\times\Sph\bigr)
 &\leq
 a_+\norm{b}_{L^1(\Sph)}
 \int_{X^2}(1+S)E^\gamma
 \,d\mu(x_1)\,d\mu(x).
 \label{eq:collision-parameter-total-mass}
\end{align}
The measure version of \cref{lem:joint-rate}, together with
\eqref{eq:limit-physical-moments}, shows that the right-hand side is
bounded by \(C_{\mathrm{rate}}\) and \(\Lambda_\mu\) is finite.

Let \(\nu,\nu_1,\nu',\nu_1'\) be the push-forwards of
\(\Lambda_\mu\) under the four Borel maps
\begin{equation*}
 (x,x_1,\alpha,\omega)\mapsto x,\qquad
 (x,x_1,\alpha,\omega)\mapsto x_1,\qquad
 (x,x_1,\alpha,\omega)\mapsto x',\qquad
 (x,x_1,\alpha,\omega)\mapsto x_1',
\end{equation*}
respectively.  Each of these is a finite nonnegative Borel measure on
\(X\), and
\begin{equation}
 \nu(X)=\nu_1(X)=\nu'(X)=\nu_1'(X)
 =\Lambda_\mu\bigl(X^2\times(0,1)\times\Sph\bigr).
 \label{eq:four-marginal-masses}
\end{equation}
Since \(X=(0,\infty)\times\R^d\) is a locally compact Polish space,
these finite Borel measures are Radon measures. Define
\begin{equation}
 \mathbf Q(\mu,\mu)
 :=\frac12\bigl(\nu'+\nu_1'-\nu-\nu_1\bigr),
 \label{eq:measure-collision-definition}
\end{equation}
which is a finite signed Radon measure.  By the definition of a
push-forward measure, every bounded Borel function \(\psi\) satisfies
\begin{align*}
 \int_X\psi\,d\nu'
 &=\int_{X^2\times(0,1)\times\Sph}
 \psi(x')\,d\Lambda_\mu,\\
 \int_X\psi\,d\nu_1'
 &=\int_{X^2\times(0,1)\times\Sph}
 \psi(x_1')\,d\Lambda_\mu,\\
 \int_X\psi\,d\nu
 &=\int_{X^2\times(0,1)\times\Sph}
 \psi(x)\,d\Lambda_\mu,\\
 \int_X\psi\,d\nu_1
 &=\int_{X^2\times(0,1)\times\Sph}
 \psi(x_1)\,d\Lambda_\mu.
\end{align*}
Substituting these four identities into
\eqref{eq:measure-collision-definition} gives
\eqref{eq:measure-collision-duality}.  Taking
\(\psi=\one_A\) gives \eqref{eq:measure-collision-set-form}.

It remains to estimate the total variation.  The triangle inequality
for finite signed measures and \eqref{eq:four-marginal-masses} give
\begin{align*}
 \norm{\mathbf Q(\mu,\mu)}_{\mathrm{TV}}
 \leq\frac12
 \bigl(
 \nu'(X)+\nu_1'(X)+\nu(X)+\nu_1(X)
 \bigr)
 =2\Lambda_\mu\bigl(X^2\times(0,1)\times\Sph\bigr).
\end{align*}
Combining this identity with
\eqref{eq:collision-parameter-total-mass} and
\cref{lem:joint-rate} proves
\eqref{eq:measure-collision-TV}.  Finally,
\begin{equation*}
 \abs{\mathcal Q(\mu,\mu)[\psi]}
 =\left|\int_X\psi\,d\mathbf Q(\mu,\mu)\right|
 \leq\norm{\psi}_\infty
 \norm{\mathbf Q(\mu,\mu)}_{\mathrm{TV}},
\end{equation*}
which gives \eqref{eq:measure-collision-bounded-test}.  Uniqueness
follows because a finite signed measure is determined by its
integrals against bounded Borel functions.
\end{proof}
\begin{proposition}[Identification of the collision form]
\label{prop:collision-identification}
Along the subsequence in \cref{prop:global-narrow-limit}, for every
\(\psi\in C_c^1(X)\) and every \(T>0\),
\begin{equation}
 \lim_{k\to\infty}\sup_{0\leq t\leq T}
 \left|
 \mathcal Q_{N_k}(f_{N_k}(t),f_{N_k}(t))[\psi]
 -\mathcal Q(\mu_t,\mu_t)[\psi]
 \right|=0.
 \label{eq:collision-identification}
\end{equation}
\end{proposition}
\begin{proof}
Fix \(\psi\in C_c^1(X)\) and \(T>0\).  Throughout the proof, write
\(\mu_{N_k}(t)=f_{N_k}(t,x)\,dx\).

\noindent\underline{\textbf{Step 1: continuity of the angularly averaged
collision kernel.}}
Define
\begin{align}
 G_{N,\psi}(x,x_1)
 &:=
 \frac12\int_0^1\int_{\Sph}
 a_N(m,m_1,\alpha)E(x,x_1)^\gamma
 b\left(\frac{u}{\abs u}\cdot\omega\right)
 \Delta\psi
 \,d\omega\,d\alpha,
 \label{eq:identification-kernel-N}\\
 G_\psi(x,x_1)
 &:=
 \frac12\int_0^1\int_{\Sph}
 a(m,m_1,\alpha)E(x,x_1)^\gamma
 b\left(\frac{u}{\abs u}\cdot\omega\right)
 \Delta\psi
 \,d\omega\,d\alpha.
 \label{eq:identification-kernel-limit}
\end{align}
Then
\begin{align}
 \mathcal Q_{N_k}(f_{N_k}(t),f_{N_k}(t))[\psi]
 &=
 \int_{X^2}G_{N_k,\psi}(x,x_1)
 \,d\mu_{N_k}(t)(x_1)\,d\mu_{N_k}(t)(x),
 \label{eq:identification-QN-kernel}\\
 \mathcal Q(\mu_t,\mu_t)[\psi]
 &=
 \int_{X^2}G_\psi(x,x_1)
 \,d\mu_t(x_1)\,d\mu_t(x).
 \label{eq:identification-Q-kernel}
\end{align}

We first prove that \(G_\psi\in C(X^2)\).  Fix \(0<r<R\) and set
\begin{equation*}
 A_{r,R}
 :=
 \max_{\substack{r\leq m,m_1\leq R\\0\leq\alpha\leq1}}
 a(m,m_1,\alpha)<\infty.
\end{equation*}
For \((x,x_1)\in K_{r,R}^2\), \eqref{eq:model-reduced-energy} gives
\begin{equation}
 0\leq E(x,x_1)
 \leq m|v|^2+m_1|v_1|^2
 \leq2R^3.
 \label{eq:identification-energy-compact}
\end{equation}

Choose \(b_\ell\in C([-1,1])\), \(b_\ell\geq0\), such that, for any
fixed \(e\in\Sph\),
\begin{equation}
 \norm{b_\ell-b}_{L^1(\Sph)}
 :=
 \int_{\Sph}
 \left|b_\ell(e\cdot\omega)-b(e\cdot\omega)\right|\,d\omega
 \longrightarrow0.
 \label{eq:identification-angular-approximation}
\end{equation}
By rotational invariance, the integral in
\eqref{eq:identification-angular-approximation} is independent of
\(e\).  Define
\begin{align}
 G_\psi^{(\ell)}(x,x_1)
 :=
 \frac12\int_0^1\int_{\Sph}
 &a(m,m_1,\alpha)E(x,x_1)^\gamma
 b_\ell\left(\frac{u}{\abs u}\cdot\omega\right)
 \Delta\psi
 \,d\omega\,d\alpha.
 \label{eq:identification-continuous-angular-kernel}
\end{align}

Let \((x_n,x_{1,n})\to(x,x_1)\) in \(K_{r,R}^2\), and denote the
corresponding quantities by \(u_n,E_n,x_n',x_{1,n}'\).  For fixed
\((\alpha,\omega)\in(0,1)\times\Sph\), if \(u\neq0\), then
\begin{align}
 &a(m_n,m_{1,n},\alpha)E_n^\gamma
 b_\ell\left(\frac{u_n}{\abs{u_n}}\cdot\omega\right)
 \bigl[
 \psi(x_n')+\psi(x_{1,n}')
 -\psi(x_n)-\psi(x_{1,n})
 \bigr]
 \notag\\
 &\qquad\longrightarrow
 a(m,m_1,\alpha)E^\gamma
 b_\ell\left(\frac{u}{\abs u}\cdot\omega\right)
 \bigl[
 \psi(x')+\psi(x_1')
 -\psi(x)-\psi(x_1)
 \bigr].
 \label{eq:identification-pointwise-nonzero-u}
\end{align}
If \(u=0\), then \(E_n\to0\), and
\begin{align}
 &\left|
 a(m_n,m_{1,n},\alpha)E_n^\gamma
 b_\ell\left(\frac{u_n}{\abs{u_n}}\cdot\omega\right)
 \bigl[
 \psi(x_n')+\psi(x_{1,n}')
 -\psi(x_n)-\psi(x_{1,n})
 \bigr]
 \right|
 \notag\\
 &\qquad\leq
 4A_{r,R}\norm{\psi}_\infty
 \norm{b_\ell}_\infty E_n^\gamma
 \longrightarrow0.
 \label{eq:identification-pointwise-zero-u}
\end{align}
Moreover, for every \(n\),
\begin{align}
 &\left|
 a(m_n,m_{1,n},\alpha)E_n^\gamma
 b_\ell\left(\frac{u_n}{\abs{u_n}}\cdot\omega\right)
 \bigl[
 \psi(x_n')+\psi(x_{1,n}')
 -\psi(x_n)-\psi(x_{1,n})
 \bigr]
 \right|\leq
 4A_{r,R}(2R^3)^\gamma
 \norm{b_\ell}_\infty\norm{\psi}_\infty.
 \label{eq:identification-angular-dominating-function}
\end{align}
 Hence
\eqref{eq:identification-pointwise-nonzero-u},
\eqref{eq:identification-pointwise-zero-u}, and dominated convergence
give
\(
 G_\psi^{(\ell)}(x_n,x_{1,n})
 \to G_\psi^{(\ell)}(x,x_1).
\)
Thus \(G_\psi^{(\ell)}\in C(K_{r,R}^2)\).  Furthermore,
\begin{align}
 \sup_{(x,x_1)\in K_{r,R}^2}
 \left|G_\psi^{(\ell)}(x,x_1)-G_\psi(x,x_1)\right|
 &\leq
 2A_{r,R}(2R^3)^\gamma\norm{\psi}_\infty
 \norm{b_\ell-b}_{L^1(\Sph)}
 \longrightarrow0.
 \label{eq:identification-uniform-angular-limit}
\end{align}
Therefore \(G_\psi\in C(K_{r,R}^2)\).  Since \(0<r<R\) were
arbitrary,
\(
 G_\psi\in C(X^2).
 \)

\noindent\underline{\textbf{Step 2: compact localization of the
collision kernel.}}
Fix \(0<r<R\).  Choose
\(\vartheta_{r,R}\in C_c(X;[0,1])\) such that
\begin{equation}
 \vartheta_{r,R}=1\quad\text{on }K_{r,R},
 \qquad
 \operatorname{supp}\vartheta_{r,R}
 \subset K_{r/2,2R},
 \label{eq:identification-one-particle-cutoff}
\end{equation}
and set
\begin{equation}
 \Theta_{r,R}(x,x_1)
 :=\vartheta_{r,R}(x)\vartheta_{r,R}(x_1),
 \qquad
 H_{r,R,\psi}(x,x_1)
 :=\Theta_{r,R}(x,x_1)G_\psi(x,x_1).
 \label{eq:identification-two-particle-cutoff}
\end{equation}
Then
\begin{equation}
 H_{r,R,\psi}\in C_c(X^2),
 \qquad
 0\leq1-\Theta_{r,R}
 \leq\one_{X^2\setminus K_{r,R}^2}.
 \label{eq:identification-cutoff-properties}
\end{equation}

On \(\operatorname{supp}\Theta_{r,R}\), one has
\(m,m_1\leq2R\), and hence \(S=m+m_1\leq4R\).  Hence for
\(N_k\geq4R\),
\begin{equation}
 \chi_{N_k}(S)=1,
 \qquad
 a_{N_k}(m,m_1,\alpha)=a(m,m_1,\alpha),
 \qquad
 \Theta_{r,R}G_{N_k,\psi}=H_{r,R,\psi}.
 \label{eq:identification-cutoff-removal-compact}
\end{equation}

The estimate \(\abs{\Delta\psi}\leq4\norm{\psi}_\infty\) gives
\begin{align}
 \abs{G_{N,\psi}(x,x_1)}
 &\leq
 2\norm{\psi}_\infty
 \int_0^1\int_{\Sph}
 a_N(m,m_1,\alpha)E^\gamma
 b\left(\frac{u}{\abs u}\cdot\omega\right)
 \,d\omega\,d\alpha,
 \label{eq:identification-kernel-rate-bound-N}\\
 \abs{G_\psi(x,x_1)}
 &\leq
 2\norm{\psi}_\infty
 \int_0^1\int_{\Sph}
 a(m,m_1,\alpha)E^\gamma
 b\left(\frac{u}{\abs u}\cdot\omega\right)
 \,d\omega\,d\alpha.
 \label{eq:identification-kernel-rate-bound-limit}
\end{align}
Consequently,
\begin{align}
 &\int_{X^2}
 (1-\Theta_{r,R})\abs{G_{N,\psi}(x,x_1)}
 \,d\mu_N(t)(x_1)\,d\mu_N(t)(x)\leq
 2\norm{\psi}_\infty
 \mathcal R_N\bigl(X^2\setminus K_{r,R}^2,t\bigr),
 \label{eq:identification-tail-bound-N}\\
 &\int_{X^2}
 (1-\Theta_{r,R})\abs{G_\psi(x,x_1)}
 \,d\mu_t(x_1)\,d\mu_t(x)\leq
 2\norm{\psi}_\infty
 \mathcal R\bigl(X^2\setminus K_{r,R}^2,t\bigr).
 \label{eq:identification-tail-bound-limit}
\end{align}

\noindent\underline{\textbf{Step 3: uniform convergence of the
localized product integrals.}}
Let \(\varepsilon>0\).  Since
\(H_{r,R,\psi}\in C_c(X^2)\), by Stone--Weierstrass theorem, there exist
\(\phi_j,\phi_{1,j}\in C_c(X)\), \(1\leq j\leq J\), such that
\begin{equation}
 H_{r,R,\psi}^{(J)}(x,x_1)
 :=\sum_{j=1}^J\phi_j(x)\phi_{1,j}(x_1),
 \qquad
 \norm{H_{r,R,\psi}-H_{r,R,\psi}^{(J)}}_\infty
 <\varepsilon.
 \label{eq:identification-tensor-approximation}
\end{equation}
For every \(t\in[0,T]\),
\begin{align}
 &\left|
 \int_{X^2}H_{r,R,\psi}\,d\mu_{N_k}(t)\,d\mu_{N_k}(t)
 -
 \int_{X^2}H_{r,R,\psi}\,d\mu_t\,d\mu_t
 \right|\leq
 2M_0(f_0)^2\varepsilon
 \notag\\
 &\qquad+
 \sum_{j=1}^J
 \left|
 \left(\int_X\phi_j\,d\mu_{N_k}(t)\right)
 \left(\int_X\phi_{1,j}\,d\mu_{N_k}(t)\right)
 -
 \left(\int_X\phi_j\,d\mu_t\right)
 \left(\int_X\phi_{1,j}\,d\mu_t\right)
 \right|.
 \label{eq:identification-product-first-bound}
\end{align}
For each \(1\leq j\leq J\),
\begin{align}
 &\left|
 \left(\int_X\phi_j\,d\mu_{N_k}(t)\right)
 \left(\int_X\phi_{1,j}\,d\mu_{N_k}(t)\right)
 -
 \left(\int_X\phi_j\,d\mu_t\right)
 \left(\int_X\phi_{1,j}\,d\mu_t\right)
 \right|
 \notag\\
 &\quad\leq
 M_0(f_0)\norm{\phi_{1,j}}_\infty
 \left|\int_X\phi_j\,d\mu_{N_k}(t)
       -\int_X\phi_j\,d\mu_t\right|
 \notag\\
 &\qquad+
 M_0(f_0)\norm{\phi_j}_\infty
 \left|\int_X\phi_{1,j}\,d\mu_{N_k}(t)
       -\int_X\phi_{1,j}\,d\mu_t\right|.
 \label{eq:identification-product-factor-bound}
\end{align}
Taking the supremum over \(t\in[0,T]\), then using
\eqref{eq:narrow-compact-time}, gives
\begin{align}
 &\limsup_{k\to\infty}
 \sup_{0\leq t\leq T}
 \left|
 \int_{X^2}H_{r,R,\psi}\,d\mu_{N_k}(t)\,d\mu_{N_k}(t)
 -
 \int_{X^2}H_{r,R,\psi}\,d\mu_t\,d\mu_t
 \right|
 \leq2M_0(f_0)^2\varepsilon.
 \label{eq:identification-localized-limsup}
\end{align}
Letting \(\varepsilon\downarrow0\), we obtain
\begin{equation}
 \lim_{k\to\infty}
 \sup_{0\leq t\leq T}
 \left|
 \int_{X^2}H_{r,R,\psi}\,d\mu_{N_k}(t)\,d\mu_{N_k}(t)
 -
 \int_{X^2}H_{r,R,\psi}\,d\mu_t\,d\mu_t
 \right|=0.
 \label{eq:identification-localized-product-limit}
\end{equation}

\noindent\underline{\textbf{Step 4: removal of the localization.}}
For \(k\) sufficiently large that \(N_k\geq4R\),
\eqref{eq:identification-QN-kernel},
\eqref{eq:identification-Q-kernel}, and
\eqref{eq:identification-cutoff-removal-compact} give
\begin{align}
 &\left|
 \mathcal Q_{N_k}(f_{N_k}(t),f_{N_k}(t))[\psi]
 -\mathcal Q(\mu_t,\mu_t)[\psi]
 \right|
 \notag\\
 &\quad\leq
 \left|
 \int_{X^2}H_{r,R,\psi}
 \,d\mu_{N_k}(t)\,d\mu_{N_k}(t)
 -
 \int_{X^2}H_{r,R,\psi}
 \,d\mu_t\,d\mu_t
 \right|
 \notag\\
 &\qquad+
 \int_{X^2}(1-\Theta_{r,R})
 \abs{G_{N_k,\psi}}
 \,d\mu_{N_k}(t)\,d\mu_{N_k}(t)
 \notag\\
 &\qquad+
 \int_{X^2}(1-\Theta_{r,R})
 \abs{G_\psi}
 \,d\mu_t\,d\mu_t.
 \label{eq:identification-full-decomposition}
\end{align}
Using \eqref{eq:identification-tail-bound-N} and
\eqref{eq:identification-tail-bound-limit}, we obtain
\begin{align}
 &\left|
 \mathcal Q_{N_k}(f_{N_k}(t),f_{N_k}(t))[\psi]
 -\mathcal Q(\mu_t,\mu_t)[\psi]
 \right|
 \notag\\
 &\quad\leq
 \left|
 \int_{X^2}H_{r,R,\psi}
 \,d\mu_{N_k}(t)\,d\mu_{N_k}(t)
 -
 \int_{X^2}H_{r,R,\psi}
 \,d\mu_t\,d\mu_t
 \right|
 \notag\\
 &\qquad+
 2\norm{\psi}_\infty
 \left[
 \mathcal R_{N_k}\bigl(X^2\setminus K_{r,R}^2,t\bigr)
 +
 \mathcal R\bigl(X^2\setminus K_{r,R}^2,t\bigr)
 \right].
 \label{eq:identification-full-rate-bound}
\end{align}
Taking the supremum over \(t\in[0,T]\), passing to the upper limit
\(k\to\infty\), and applying
\eqref{eq:identification-localized-product-limit}, we find
\begin{align}
 &\limsup_{k\to\infty}\sup_{0\leq t\leq T}
 \left|
 \mathcal Q_{N_k}(f_{N_k}(t),f_{N_k}(t))[\psi]
 -\mathcal Q(\mu_t,\mu_t)[\psi]
 \right|
 \notag\\
 &\quad\leq
 2\norm{\psi}_\infty
 \sup_{0\leq t\leq T}
 \left[
 \sup_{N\geq1}
 \mathcal R_N\bigl(X^2\setminus K_{r,R}^2,t\bigr)
 +
 \mathcal R\bigl(X^2\setminus K_{r,R}^2,t\bigr)
 \right].
 \label{eq:identification-final-tail-bound}
\end{align}
Finally, \eqref{eq:collision-rate-tightness} yields
\begin{align*}
 &\limsup_{k\to\infty}\sup_{0\leq t\leq T}
 \left|
 \mathcal Q_{N_k}(f_{N_k}(t),f_{N_k}(t))[\psi]
 -\mathcal Q(\mu_t,\mu_t)[\psi]
 \right|\\
 &\quad\leq
 2\norm{\psi}_\infty
 \lim_{r\downarrow0}\lim_{R\uparrow\infty}
 \sup_{0\leq t\leq T}
 \left[
 \sup_{N\geq1}
 \mathcal R_N\bigl(X^2\setminus K_{r,R}^2,t\bigr)
 +
 \mathcal R\bigl(X^2\setminus K_{r,R}^2,t\bigr)
 \right]
 =0.
\end{align*}
This proves \eqref{eq:collision-identification}.
\end{proof}
\begin{corollary}[Measure equation and total-variation regularity]
\label{cor:measure-equation}
The family
\(
 (\mathbf Q(\mu_t,\mu_t))_{t\geq0}
\)
is a finite signed-measure kernel on \(X\).  Moreover, for every
\(0\leq s\leq t\) and every bounded Borel \(\psi:X\to\R\),
\begin{equation}
 \int_X\psi\,d\mu_t-\int_X\psi\,d\mu_s
 =\int_s^t\mathcal Q(\mu_\tau,\mu_\tau)[\psi]\,d\tau.
 \label{eq:measure-bounded-test}
\end{equation}
In particular,
\begin{equation}
 \norm{\mu_t-\mu_s}_{\mathrm{TV}}
 \leq2C_{\mathrm{rate}}\abs{t-s}.
 \label{eq:measure-TV-Lipschitz}
\end{equation}
\end{corollary}
\begin{proof}
\noindent\underline{\textbf{Step 1: the identity for smooth tests.}}
Let \(\psi\in C_c^1(X)\).  By
\eqref{eq:truncated-weak-identity},
\begin{equation*}
 \int_X\psi\,d\mu_{N_k}(t)-\int_X\psi\,d\mu_{N_k}(s)
 =
 \int_s^t
 \mathcal Q_{N_k}(f_{N_k}(\tau),f_{N_k}(\tau))[\psi]\,d\tau.
\end{equation*}
Passing to the limit using \eqref{eq:narrow-compact-time} and
\eqref{eq:collision-identification} gives
\begin{equation}
 \int_X\psi\,d\mu_t-\int_X\psi\,d\mu_s
 =
 \int_s^t\mathcal Q(\mu_\tau,\mu_\tau)[\psi]\,d\tau.
 \label{eq:measure-smooth-test}
\end{equation}

\noindent\underline{\textbf{Step 2: total-variation regularity.}}
By \eqref{eq:measure-collision-bounded-test},
\begin{equation*}
 \left|\int_X\psi\,d(\mu_t-\mu_s)\right|
 \leq
 2C_{\mathrm{rate}}\abs{t-s}\norm{\psi}_\infty,
 \qquad \psi\in C_c^1(X).
\end{equation*}
Since \(C_c^1(X)\) is dense in \(C_0(X)\) under the uniform norm, the
same estimate holds for every \(\psi\in C_0(X)\).  The dual
characterization of the total variation of a finite signed Radon
measure therefore yields
\(
 \norm{\mu_t-\mu_s}_{\mathrm{TV}}
 \leq2C_{\mathrm{rate}}\abs{t-s},
\)
which is \eqref{eq:measure-TV-Lipschitz}.

\noindent\underline{\textbf{Step 3: the collision-measure kernel.}}
The total-variation continuity \eqref{eq:measure-TV-Lipschitz} implies that
\(\tau\mapsto\mu_\tau(A)\) is measurable for every Borel set
\(A\subset X\).  Thus \((\mu_\tau)_{\tau\geq0}\) is a finite-measure
kernel. Moreover, since
\[
 (\mu_\tau\otimes\mu_\tau)(A\times B)
 =\mu_\tau(A)\mu_\tau(B)
\]
is measurable in \(\tau\) for every Borel rectangle \(A\times B\),
a monotone-class argument shows that
\((\mu_\tau\otimes\mu_\tau)_{\tau\geq0}\) is a finite-measure kernel
on \(X^2\).
Since the collision weight in
\eqref{eq:collision-parameter-measure} is nonnegative and Borel, and
the four maps
\begin{equation*}
 (x,x_1,\alpha,\omega)\mapsto x,\qquad
 (x,x_1,\alpha,\omega)\mapsto x_1,\qquad
 (x,x_1,\alpha,\omega)\mapsto x',\qquad
 (x,x_1,\alpha,\omega)\mapsto x_1'
\end{equation*}
are Borel, the collision parameter measures
\((\Lambda_{\mu_\tau})_{\tau\geq0}\) and their four push-forwards are
finite-measure kernels.  By
\eqref{eq:measure-collision-definition},
\((\mathbf Q(\mu_\tau,\mu_\tau))_{\tau\geq0}\) is therefore a finite
signed-measure kernel.

\noindent\underline{\textbf{Step 4: extension to bounded Borel tests.}}
For fixed \(0\leq s\leq t\), define
\begin{equation*}
 \mathbf M_{s,t}(A)
 :=
 \int_s^t\mathbf Q(\mu_\tau,\mu_\tau)(A)\,d\tau,
 \qquad A\subset X\ \text{Borel}.
\end{equation*}
It is a finite signed Radon measure, and
\begin{equation*}
 \norm{\mathbf M_{s,t}}_{\mathrm{TV}}
 \leq
 \int_s^t
 \norm{\mathbf Q(\mu_\tau,\mu_\tau)}_{\mathrm{TV}}\,d\tau
 \leq2C_{\mathrm{rate}}\abs{t-s}.
\end{equation*}
By \eqref{eq:measure-collision-duality} and
\eqref{eq:measure-smooth-test}, the measures
\(\mathbf M_{s,t}\) and \(\mu_t-\mu_s\) agree against every
\(\psi\in C_c^1(X)\).  Hence they are equal as finite signed Radon
measures:
\begin{equation}
 \mu_t-\mu_s=\mathbf M_{s,t}.
 \label{eq:measure-integral-identity}
\end{equation}
Testing \eqref{eq:measure-integral-identity} against an arbitrary
bounded Borel function \(\psi\) gives
\eqref{eq:measure-bounded-test}.
\end{proof}
\subsection{Absolute continuity and the strong formulation}

The preceding \cref{cor:measure-equation} provides the two ingredients needed below:
the measure equation \eqref{eq:measure-bounded-test} is valid for bounded Borel tests, and the
collision measures form a signed-measure kernel \((\mathbf Q(\mu_\tau,\mu_\tau))_{\tau\geq0}\) in time.  We use the
first property to exclude singular components of the limiting
measures \(\mu_t\) and the second to recover a jointly measurable
\(L^1(X)\)-valued collision term \(\mathbf Q(f(t),f(t))\).  

\begin{lemma}[Absolute continuity of mixed gain marginals]
\label{lem:cross-gain-absolute-continuity}
Let \(\rho,\nu\) be finite nonnegative Radon measures on \(X\), and
define
\begin{align}
 d\Lambda_{\rho,\nu}(x,x_1,\alpha,\omega)
 &:=
 a(m,m_1,\alpha)E(x,x_1)^\gamma
 b\left(\frac{u}{\abs u}\cdot\omega\right)
 \,d\omega\,d\alpha\,d\nu(x_1)\,d\rho(x).
 \label{eq:cross-collision-parameter-measure}
\end{align}
Assume that \(\Lambda_{\rho,\nu}\) is finite.  Let
\(\Gamma_0(\rho,\nu)\) and \(\Gamma_1(\rho,\nu)\) be its push-forwards
under the two outgoing maps:
\begin{align}
 \Gamma_0(\rho,\nu)
 &:=
 \bigl[(x,x_1,\alpha,\omega)\mapsto x'\bigr]_\#
 \Lambda_{\rho,\nu},
 \label{eq:first-cross-gain}\\
 \Gamma_1(\rho,\nu)
 &:=
 \bigl[(x,x_1,\alpha,\omega)\mapsto x_1'\bigr]_\#
 \Lambda_{\rho,\nu}.
 \label{eq:second-cross-gain}
\end{align}
If either \(\rho\ll dx\) or \(\nu\ll dx\), then
\begin{equation}
 \Gamma_0(\rho,\nu)\ll dx,
 \qquad
 \Gamma_1(\rho,\nu)\ll dx.
 \label{eq:cross-gain-absolute-continuity}
\end{equation}
Here \(dx\) denotes Lebesgue measure on the space
\(X\).
Consequently, if \(\mu=f\,dx+\sigma\) is the Lebesgue decomposition
of a finite nonnegative measure and \(\Lambda_{\mu,\mu}\) is finite,
then in the expansion
\begin{equation}
 \Gamma_i(\mu,\mu)
 =
 \Gamma_i(f\,dx,f\,dx)
 +\Gamma_i(f\,dx,\sigma)
 +\Gamma_i(\sigma,f\,dx)
 +\Gamma_i(\sigma,\sigma),
 \qquad i=0,1,
\end{equation}
the first three terms are absolutely continuous.  Hence only
\(\Gamma_i(\sigma,\sigma)\) can contribute to the singular part of an
outgoing gain marginal.
\end{lemma}
\begin{proof}
Let \(A\subset X\) be Borel with \(dx(A)=0\).

\noindent\underline{\textbf{Step 1: Jacobians.}}
Assume first that \(d\rho=f\,dx\).  Fix
\(x_1=(m_1,v_1)\), \(\alpha\in(0,1)\), and \(\omega\in\Sph\), and set
\(
 F_0(x):=x',
 F_1(x):=x_1'.
\)
Since \(D_vV=\theta I_d\) and \(D_vu=I_d\),
\begin{equation}
 D_vv'=\theta I_d+(1-\alpha)sR_\omega,
 \qquad
 D_vv_1'=\theta I_d-\alpha sR_\omega.
\end{equation}
Moreover, \(m'=\alpha(m+m_1)\) and
\(m_1'=(1-\alpha)(m+m_1)\) are independent of \(v\).  Therefore,
\begin{equation}
 D_xF_0=
 \begin{pmatrix}
  \alpha&0\\
  \partial_m v'&\theta I_d+(1-\alpha)sR_\omega
 \end{pmatrix},
 \qquad
 D_xF_1=
 \begin{pmatrix}
  1-\alpha&0\\
  \partial_m v_1'&\theta I_d-\alpha sR_\omega
 \end{pmatrix}.
\end{equation}
Since \(R_\omega\) has eigenvalue \(-1\) in the \(\omega\)-direction
and eigenvalue \(1\) on \(\omega^\perp\),
\begin{align*}
 \abs{\det D_xF_0}
 &=
 \alpha
 \abs{\theta-(1-\alpha)s}
 \bigl(\theta+(1-\alpha)s\bigr)^{d-1},\\
 \abs{\det D_xF_1}
 &=
 (1-\alpha)
 \bigl(\theta+\alpha s\bigr)
 \abs{\theta-\alpha s}^{d-1},
\end{align*}
For the fixed parameters \((x_1,\alpha,\omega)\), define the critical
sets
\begin{equation}
 Z_i:=\{x\in X:\det D_xF_i(x)=0\},
 \qquad i=0,1.
\end{equation}
The first determinant can vanish only when
\(
 \alpha=1-\theta.
\)
For \(d\geq2\), the second determinant can vanish only when
\(
 \alpha=\theta.
\)
When \(d=1\),
\(\abs{\det D_xF_1}=(1-\alpha)(\theta+\alpha s)>0\).
It follows that
\begin{equation}
 Z_0\subset\{x\in X:\alpha=1-\theta\}=\left\{\dfrac{(1-\alpha)m_1}{\alpha}\right\}\times\R^d,
 \qquad
 Z_1\subset\{x\in X:\alpha=\theta\}=\left\{\dfrac{\alpha m_1}{1-\alpha}\right\}\times\R^d.
 \label{eq:cross-gain-two-resonances}
\end{equation}
In particular, \(Z_1=\varnothing\) when \(d=1\). Hence $dx(Z_i)=0$ for \(i=0,1\).

\noindent\underline{\textbf{Step 2: null preimages.}}
Fix \(i\in\{0,1\}\).  At every point of \(X\setminus Z_i\), the
inverse-function theorem gives a neighborhood on which \(F_i\) is a
\(C^1\)-diffeomorphism.  By second countability, there are countably
many such neighborhoods \(U_{i,n}\) covering \(X\setminus Z_i\), with
locally Lipschitz inverse maps
\begin{equation*}
 G_{i,n}:F_i(U_{i,n})\longrightarrow U_{i,n}.
\end{equation*}
It follows that
\begin{equation}\label{eq:cross-gain-preimage-decomposition}
 F_i^{-1}(A)
 \subset
 Z_i\cup
 \bigcup_{n=1}^{\infty}
 G_{i,n}\bigl(A\cap F_i(U_{i,n})\bigr).
\end{equation}
For every \(n\), the set
\(A\cap F_i(U_{i,n})\) is Lebesgue-null because it is contained in
\(A\).  Since \(G_{i,n}\) is Lipschitz,
\(G_{i,n}(A\cap F_i(U_{i,n}))\) is also Lebesgue-null.  Since \(dx(Z_i)=0\), \eqref{eq:cross-gain-preimage-decomposition} gives
\(dx(F_i^{-1}(A))=0.\)
Tonelli's theorem now gives
\begin{align*}
 \Gamma_i(\rho,\nu)(A)
 &=
 \int_X\int_0^1\int_{\Sph}
 \left[
 \int_{F_i^{-1}(A)}
 a(m,m_1,\alpha)E(x,x_1)^\gamma
 b\left(\frac{u}{\abs u}\cdot\omega\right)
 f(x)\,dx
 \right]
 d\omega\,d\alpha\,d\nu(x_1)
 =0.
\end{align*}
Hence \(\Gamma_0(\rho,\nu),\Gamma_1(\rho,\nu)\ll dx.\)
If instead \(\nu\ll dx\), apply the preceding argument after the
Borel change of variables
\(
 (x,x_1,\alpha,\omega)
 \mapsto
 (x_1,x,1-\alpha,-\omega),
\)
which exchanges the two incoming variables and the two
outgoing labels, while preserving the collision weight by the
symmetries in
\eqref{eq:model-fixed-share}--\eqref{eq:model-rate-disintegration}.  Thus the same
conclusion holds when \(\nu\ll dx\).

\noindent\underline{\textbf{Step 3: the Lebesgue decomposition.}}
Let \(\mu=f\,dx+\sigma\).  By bilinearity,
\begin{equation*}
 \Gamma_i(\mu,\mu)
 =
 \Gamma_i(f\,dx,f\,dx)
 +\Gamma_i(f\,dx,\sigma)
 +\Gamma_i(\sigma,f\,dx)
 +\Gamma_i(\sigma,\sigma),
 \qquad i=0,1.
\end{equation*}
The first three terms are absolutely continuous by Step~2.  Hence only
\(\Gamma_i(\sigma,\sigma)\) can contribute to the singular part of
\(\Gamma_i(\mu,\mu)\), proving the final assertion.
\end{proof}
\begin{proposition}[Absolute continuity and completion of the global
construction]
\label{prop:absolute-continuity-strong}
For every \(t\geq0\), the measure \(\mu_t\) is absolutely continuous:
\(\mu_t=f(t,x)\,dx\) for a nonnegative density \(f(t)\in L^1(X)\).
This density is the global integral weak solution asserted in
\cref{thm:model-global-no-gelation}.
\end{proposition}
\begin{proof}
  Fix \(T>0\).  We first prove that \(\mu_t\) is absolutely continuous
for every \(t\in[0,T]\).  Step~1 uses the total-variation continuity
of \(t\mapsto\mu_t\) to find a single Lebesgue-null Borel set \(Z_T\)
such that $\sigma_t(X\setminus Z_T)=0$ for every $t\in[0,T]$.  This allows us to use \(\one_{Z_T}\) as one fixed
bounded Borel test throughout the interval.  In Step~2, we only need to consider the singular--singular outgoing gains, which are bounded by the corresponding
singular--singular incoming losses.  It follows that
\(
 \mathcal Q(\mu_t,\mu_t)[\one_{Z_T}]\leq0,
\)
so \(t\mapsto\mu_t(Z_T)=\sigma_t(X)\) is nonincreasing.  Since
\(\sigma_0=0\), the singular part vanishes on \([0,T]\).

Once \(\mu_t\ll dx\) has been established, Steps~3--5 show that the
collision measure is also absolutely continuous, choose its density
jointly measurably in time, and rewrite the measure equation as an
\(L^1(X)\)-valued Bochner equation.

\noindent\underline{\textbf{Step 1: concentration of all singular parts
on a fixed null set.}}
For each \(t\in[0,T]\), write the Lebesgue decomposition
\begin{equation}
 \mu_t=f_t\,dx+\sigma_t,
 \qquad
 f_t\in L^1(X),\quad \sigma_t\perp dx.
\end{equation}
Since the singular projection in the Lebesgue decomposition is linear and
contractive on finite signed Radon measures,
\begin{equation}
 \norm{\sigma_t-\sigma_s}_{\mathrm{TV}}
 \leq\norm{\mu_t-\mu_s}_{\mathrm{TV}}
 \leq2C_{\mathrm{rate}}\abs{t-s},
 \qquad 0\leq s,t\leq T,
\end{equation}
by \eqref{eq:measure-TV-Lipschitz}. Choose a countable dense set
\(\{t_j\}_{j\geq1}\subset[0,T]\).  For every \(j\), choose a Borel set
\(Z_j\subset X\) such that
\begin{equation}
 dx(Z_j)=0,\quad
 \sigma_{t_j}(X\setminus Z_j)=0,\quad  Z_T:=\bigcup_{j\geq1}Z_j
\end{equation}
Then \(dx(Z_T)=0\).  Given
\(t\in[0,T]\), choose \(t_{j_k}\to t\).  Since \(Z_{j_k}\subset Z_T\),
\begin{equation}
 \sigma_t(X\setminus Z_T)
 \leq
 \|\sigma_t-\sigma_{t_{j_k}}\|_{\mathrm{TV}}
 +\sigma_{t_{j_k}}(X\setminus Z_T)\leq 2C_{\mathrm{rate}}\abs{t-t_{j_k}}
 \to 0.
\end{equation}
Thus
\begin{equation}
   \sigma_t(X\setminus Z_T)=0
 \qquad\text{for every }t\in[0,T].
\end{equation}
Hence all the singular parts are concentrated on the same
Lebesgue-null Borel set \(Z_T\).  Since \(f_t\,dx(Z_T)=0\), one has
\begin{equation}
 \mu_t\!\restriction_{Z_T}=\sigma_t,
 \qquad
 \mu_t(Z_T)=\sigma_t(X),
 \qquad 0\leq t\leq T.
 \label{eq:common-singular-support}
\end{equation}

\noindent\underline{\textbf{Step 2: ruling out the singular part.}}
Fix \(\tau\in[0,T]\). By \cref{lem:cross-gain-absolute-continuity},
\begin{equation*}
 \Gamma_i(f_\tau\,dx,f_\tau\,dx),\ 
 \Gamma_i(f_\tau\,dx,\sigma_\tau),\ 
 \Gamma_i(\sigma_\tau,f_\tau\,dx)\ll dx,
 \qquad i=0,1.
\end{equation*}  Since \(dx(Z_T)=0\),
\begin{equation}
 \Gamma_i(\mu_\tau,\mu_\tau)(Z_T)
 =
 \Gamma_i(\sigma_\tau,\sigma_\tau)(Z_T)
 \leq
 \Lambda_{\sigma_\tau,\sigma_\tau}(X^2\times (0,1)\times \Sph).
\end{equation}
Since \(\mu_\tau|_{Z_T}=\sigma_\tau\)  and \eqref{eq:common-singular-support},
\begin{align*}
 &\int_{X^2\times (0,1)\times \Sph}
 \bigl[\one_{Z_T}(x)+\one_{Z_T}(x_1)\bigr]
 \,d\Lambda_{\mu_\tau,\mu_\tau}\\
 &\qquad=
 \Lambda_{\sigma_\tau,\mu_\tau}(X^2\times (0,1)\times \Sph)
 +\Lambda_{\mu_\tau,\sigma_\tau}(X^2\times (0,1)\times \Sph)\\
 &\qquad=
 2\Lambda_{\sigma_\tau,\sigma_\tau}(X^2\times (0,1)\times \Sph)
 +\Lambda_{\sigma_\tau,f_\tau\, dx}(X^2\times (0,1)\times \Sph)
 +\Lambda_{f_\tau\, dx,\sigma_\tau}(X^2\times (0,1)\times \Sph).
\end{align*}
Using \eqref{eq:measure-collision-set-form}, we obtain
\begin{align}
 \mathcal Q(\mu_\tau,\mu_\tau)[\one_{Z_T}]
 &=
 \frac12\left[
 \Gamma_0(\mu_\tau,\mu_\tau)(Z_T)
 +\Gamma_1(\mu_\tau,\mu_\tau)(Z_T)-
 \int_{X^2\times(0,1)\times\Sph}
 (\one_{Z_T}(x)+\one_{Z_T}(x_1))
 \,d\Lambda_{\mu_\tau,\mu_\tau}
 \right]
 \notag\\
 &\leq
 -\frac12\left[
 \Lambda_{\sigma_\tau,f_\tau\,dx}
 \bigl(X^2\times(0,1)\times\Sph\bigr)
 +
 \Lambda_{f_\tau\,dx,\sigma_\tau}
 \bigl(X^2\times(0,1)\times\Sph\bigr)
 \right]
 \leq0.
 \label{eq:singular-number-dissipation}
\end{align}
Applying \eqref{eq:measure-bounded-test} with
\(\psi=\one_{Z_T}\),
\eqref{eq:common-singular-support} and
\eqref{eq:singular-number-dissipation} give
\begin{equation}
 \sigma_t(X)-\sigma_s(X)
 =
 \int_s^t
 \mathcal Q(\mu_\tau,\mu_\tau)[\one_{Z_T}]\,d\tau
 \leq0,
 \qquad 0\leq s\leq t\leq T.
\end{equation}
Since \(\mu_0=f_0\,dx\), one has \(\sigma_0=0\).  Positivity of
\(\sigma_t\) therefore implies \(\sigma_t=0\) for every \(t\in[0,T]\).
As \(T>0\) was arbitrary,
\begin{equation}
 \mu_t=f_t\,dx
 \qquad\text{for every }t\geq0.
 \label{eq:limit-measure-absolute-continuity}
\end{equation}

\noindent\underline{\textbf{Step 3: absolute continuity of the
collision measure.}}
Fix \(t\geq0\).  By
\eqref{eq:limit-measure-absolute-continuity}, \(\mu_t\ll dx\).
Since \(\Lambda_{\mu_t,\mu_t}=\Lambda_{\mu_t}\), applying
\cref{lem:cross-gain-absolute-continuity} with
\(\rho=\nu=\mu_t\) gives
\(
 \Gamma_0(\mu_t,\mu_t)\ll dx,
 \Gamma_1(\mu_t,\mu_t)\ll dx.
\)
Thus the two outgoing marginals are absolutely continuous.

We next consider the two incoming marginals.  Let \(A\subset X\) be
Borel with \(dx(A)=0\).  Since \(\mu_t\ll dx\), one has
\(\mu_t(A)=0\).  By
\eqref{eq:collision-parameter-measure},
\(\Lambda_{\mu_t}\) has a nonnegative density with respect to the
product measure
\(
 \mu_t\otimes\mu_t\otimes d\alpha\otimes d\omega.
\)
Consequently, Tonelli's theorem gives
\begin{align*}
 \int_{X^2\times(0,1)\times\Sph}
 \one_A(x)\,d\Lambda_{\mu_t}
 &=
 \Lambda_{\mu_t}
 \bigl(A\times X\times(0,1)\times\Sph\bigr)
 =0,\\
 \int_{X^2\times(0,1)\times\Sph}
 \one_A(x_1)\,d\Lambda_{\mu_t}
 &=
 \Lambda_{\mu_t}
 \bigl(X\times A\times(0,1)\times\Sph\bigr)
 =0,
\end{align*}
which are precisely the values on \(A\) of the push-forwards of
\(\Lambda_{\mu_t}\) under the two incoming maps
\(
 (x,x_1,\alpha,\omega)\mapsto x,
 (x,x_1,\alpha,\omega)\mapsto x_1.
\)
Hence both incoming marginals are also absolutely continuous with
respect to \(dx\).

For the chosen null set \(A\), the four terms in
\eqref{eq:measure-collision-definition} therefore satisfy
\begin{align*}
 \mathbf Q(\mu_t,\mu_t)(A)
 &=
 \frac12\Biggl[
 \Gamma_0(\mu_t,\mu_t)(A)
 +\Gamma_1(\mu_t,\mu_t)(A)\\
 &
 -\int_{X^2\times(0,1)\times\Sph}
 \one_A(x)\,d\Lambda_{\mu_t}
 -\int_{X^2\times(0,1)\times\Sph}
 \one_A(x_1)\,d\Lambda_{\mu_t}
 \Biggr]
 =0.
\end{align*}
Since \(A\) was an arbitrary \(dx\)-null Borel set, we conclude that
\begin{equation}
 \mathbf Q(\mu_t,\mu_t)\ll dx
 \qquad\text{for every }t\geq0.
 \label{eq:limit-collision-measure-ac}
\end{equation}

\noindent\underline{\textbf{Step 4: a jointly measurable collision
density.}}
Fix \(T>0\). 

\noindent\emph{(a) Construction of a space--time density.} By \cref{cor:measure-equation},
\(t\mapsto\mathbf Q(\mu_t,\mu_t)\) is a finite signed-measure kernel
on \(X\).  Since its total variation is uniformly bounded by
\eqref{eq:measure-collision-TV}, kernel integration defines a
finite signed measure on \((0,T)\times X\) by
\begin{equation}
 \mathsf Q_T(B)
 :=
 \int_0^T
 \mathbf Q(\mu_\tau,\mu_\tau)(B_\tau)\,d\tau,
 \qquad
 B_\tau:=\{x\in X:(\tau,x)\in B\}.
 \label{eq:space-time-collision-measure}
\end{equation}
Moreover,
\begin{equation*}
 \norm{\mathsf Q_T}_{\mathrm{TV}}
 \leq
 \int_0^T
 \norm{\mathbf Q(\mu_\tau,\mu_\tau)}_{\mathrm{TV}}\,d\tau
 \leq2TC_{\mathrm{rate}}.
\end{equation*}
If \((d\tau\otimes dx)(B)=0\), then \(dx(B_\tau)=0\) for almost every
\(\tau\) by Fubini's theorem.
\eqref{eq:limit-collision-measure-ac} then gives
\(
 \mathbf Q(\mu_\tau,\mu_\tau)(B_\tau)=0
\)
for almost every \(\tau\).  Hence
\(
 \mathsf Q_T\ll d\tau\otimes dx.
\)
The Radon--Nikodym theorem therefore gives a jointly measurable
function \(\mathfrak Q^T\in L^1((0,T)\times X)\) such that
\begin{equation*}
 d\mathsf Q_T(\tau,x)
 =
 \mathfrak Q^T(\tau,x)\,d\tau\,dx.
\end{equation*}

\noindent\emph{(b) Identification of almost every time slice.}
Let \(\mathcal A\) be a countable algebra generating the Borel
\(\sigma\)-algebra of \(X\).  For every Borel set
\(I\subset(0,T)\) and every \(A\in\mathcal A\),
\begin{equation*}
 \int_I\int_A\mathfrak Q^T(\tau,x)\,dx\,d\tau
 =
 \mathsf Q_T(I\times A)
 =
 \int_I\mathbf Q(\mu_\tau,\mu_\tau)(A)\,d\tau.
\end{equation*}
Hence there is a null set \(N_A\subset(0,T)\) such that
\begin{equation*}
 \int_A \mathfrak Q^T(\tau,x)\,dx
 =
 \mathbf Q(\mu_\tau,\mu_\tau)(A),
 \qquad
 \tau\in(0,T)\setminus N_A.
\end{equation*}
By Fubini's theorem, there is also a null set \(N_0\subset(0,T)\) such
that \(\mathfrak Q^T(\tau,\cdot)\in L^1(X)\) for every
\(\tau\in(0,T)\setminus N_0\).  Since \(\mathcal A\) is countable,
\begin{equation*}
 N_T:=N_0\cup\bigcup_{A\in\mathcal A}N_A
\end{equation*}
is null.  Hence, for every \(\tau\in(0,T)\setminus N_T\),
\begin{equation*}
 \int_A \mathfrak Q^T(\tau,x)\,dx
 =
 \mathbf Q(\mu_\tau,\mu_\tau)(A)
 \qquad\text{for every }A\in\mathcal A.
\end{equation*}
Fix such a \(\tau\).  The set functions
\begin{equation*}
 B\longmapsto\int_B \mathfrak Q^T(\tau,x)\,dx,
 \qquad
 B\longmapsto\mathbf Q(\mu_\tau,\mu_\tau)(B)
\end{equation*}
are finite signed measures on \(X\) and agree on the algebra
\(\mathcal A\).  Since \(\mathcal A\) generates the Borel
\(\sigma\)-algebra of \(X\), uniqueness of finite signed measures
gives
\begin{equation}
 \mathfrak Q^T(\tau,\cdot)\,dx
 =
 \mathbf Q(\mu_\tau,\mu_\tau),
 \qquad
 \tau\in(0,T)\setminus N_T.
 \label{eq:space-time-density-slices}
\end{equation}
It follows from \eqref{eq:space-time-density-slices} and
\eqref{eq:measure-collision-TV} that
\begin{equation*}
 \|\mathfrak Q^T(\tau)\|_{L^1(X)}
 =
 \norm{\mathbf Q(\mu_\tau,\mu_\tau)}_{\mathrm{TV}}
 \leq2C_{\mathrm{rate}}
 \qquad\text{for almost every }\tau\in(0,T).
\end{equation*}
Redefine \(\mathfrak Q^T(\tau,\cdot)=0\) for \(\tau\in N_T\).
Joint measurability and the separability of \(L^1(X)\) show that
\(\tau\mapsto\mathfrak Q^T(\tau)\) is strongly measurable as an
\(L^1(X)\)-valued map.

\noindent\emph{(c) Globalization in time.}
Apply this construction with \(T=n\) for each \(n\in\mathbb N\), and
denote the resulting density by \(\mathfrak Q^n\).  Define
\(\mathfrak Q(0)=0\) and
\begin{equation*}
 \mathfrak Q(t):=\mathfrak Q^n(t),
 \qquad
 t\in[n-1,n),\quad t>0,\quad n\geq1.
\end{equation*}
This gives a strongly measurable function
\(
 \mathfrak Q\in L^\infty(0,\infty;L^1(X)).
\)
Writing \(\mathfrak Q_t=\mathfrak Q(t,\cdot)\), we have
\begin{equation}
 \mathfrak Q_t(x)\,dx
 =
 \mathbf Q(\mu_t,\mu_t),
 \qquad
 \norm{\mathfrak Q_t}_{L^1(X)}
 \leq2C_{\mathrm{rate}}
 \quad\text{for almost every }t>0.
 \label{eq:limit-collision-density-bound}
\end{equation}

\noindent\underline{\textbf{Step 5: the \(L^1\)-valued collision
equation.}}
Set \(f(t)=f_t\).  Since
\(\mu_t-\mu_s=(f(t)-f(s))\,dx\),
\eqref{eq:measure-TV-Lipschitz} gives
\begin{equation}
 \norm{f(t)-f(s)}_{L^1(X)}
 =
 \norm{\mu_t-\mu_s}_{\mathrm{TV}}
 \leq2C_{\mathrm{rate}}\abs{t-s},
 \qquad s,t\geq0.
 \label{eq:limit-density-Lipschitz}
\end{equation}
By \eqref{eq:limit-collision-density-bound},
\(t\mapsto\mathfrak Q_t\) is Bochner integrable on every bounded time
interval.  For every \(0\leq s\leq t\) and every Borel set
\(A\subset X\),
\eqref{eq:measure-integral-identity} and
\eqref{eq:limit-collision-density-bound} give
\begin{align*}
 &\int_A[f(t,x)-f(s,x)]\,dx
=
 \int_s^t\mathbf Q(\mu_\tau,\mu_\tau)(A)\,d\tau\\
 &\qquad=
 \int_s^t\int_A\mathfrak Q_\tau(x)\,dx\,d\tau=
 \int_A
 \left(\int_s^t\mathfrak Q_\tau\,d\tau\right)(x)\,dx.
\end{align*}
Since this holds for every Borel set \(A\),
\begin{equation}
 f(t)-f(s)
 =
 \int_s^t\mathfrak Q_\tau\,d\tau
 \qquad\text{in }L^1(X).
 \label{eq:limit-L1-increment}
\end{equation}
For almost every \(\tau>0\), define
\(
 \mathbf Q(f(\tau),f(\tau)):=\mathfrak Q_\tau.
\)
Then
\begin{equation*}
 \mathbf Q(f(\tau),f(\tau))\,dx
 =
 \mathbf Q(\mu_\tau,\mu_\tau)
 \qquad\text{for almost every }\tau>0.
\end{equation*}
Taking \(s=0\) in \eqref{eq:limit-L1-increment}, we obtain
\begin{equation}
 f(t)
 =
 f_0+\int_0^t\mathbf Q(f(\tau),f(\tau))\,d\tau
 \qquad\text{in }L^1(X),
 \label{eq:limit-Bochner-identity}
\end{equation}
where the integral is a Bochner integral.  Consequently,
\begin{equation}
 \mathbf Q(f,f)\in L^\infty(0,\infty;L^1(X)),
 \qquad
 \partial_t f=\mathbf Q(f,f)
 \quad\text{for almost every }t>0.
\end{equation}
Equation \eqref{eq:limit-Bochner-identity}, together with
\eqref{eq:limit-collision-density-bound}, gives
\(
 f\in W^{1,\infty}(0,\infty;L^1(X)).
\)
Moreover, since \(\mu_t=f(t,x)\,dx\),
\eqref{eq:measure-bounded-test} becomes
\eqref{eq:model-bounded-test-identity}.  Finally,
\eqref{eq:limit-physical-moments} gives
\begin{equation}
 \sup_{t\geq0}
 \int_X(1+m+m|v|^2)f(t,x)\,dx 
 \leq
 M_0(f_0)+M_1(f_0)+M_2(f_0).
\end{equation}
Together with
\eqref{eq:limit-collision-density-bound} and
\eqref{eq:limit-Bochner-identity}, this proves that \(f\) is the
global integral weak solution in
\cref{def:model-integral-solution}.
\end{proof}
\begin{proof}[Proof of \cref{thm:model-global-no-gelation}]
If \(M_0(f_0)=0\), then \(f_0=0\) almost everywhere, and the zero
solution has all the asserted properties.  Assume \(M_0(f_0)>0\).
By \cref{lem:Phi-properties}, the function \(\Phi\) defined in
\eqref{eq:Phi-definition} satisfies
\eqref{eq:model-superlinear-weight}.  Let \(f_N\) be the approximate
solutions supplied by \cref{prop:truncated-solutions}.  Then
\cref{thm:Phi-propagation,prop:no-gelation} give
\eqref{eq:model-uniform-Phi-bound} and
\eqref{eq:model-approximate-no-gelation}.

By \cref{prop:global-narrow-limit}, a subsequence of the approximate
solutions converges to a narrowly continuous measure-valued curve
\((\mu_t)_{t\geq0}\) satisfying
\eqref{eq:limit-physical-moments},
\eqref{eq:limit-Phi-bound}, and
\eqref{eq:measure-limit-no-gelation}.  The collision-form
identification in
\cref{prop:collision-identification,cor:measure-equation}, followed by
\cref{prop:absolute-continuity-strong}, yields
\(\mu_t=f(t,x)\,dx\), where \(f\) is a global integral weak solution
satisfying \eqref{eq:model-global-solution-class} and
\eqref{eq:model-global-collision-density}.  Since
\(\mu_t=f(t,x)\,dx\), \eqref{eq:limit-physical-moments} gives
\eqref{eq:model-moment-laws}, while
\eqref{eq:measure-limit-no-gelation} gives
\eqref{eq:model-limit-no-gelation}.  This proves all the assertions.
\end{proof}
\section*{Acknowledgments}
S. Luo gratefully acknowledges the hospitality of the Department of
Mathematics at Duke University during his visit in summer 2026, where
part of this work was carried out. 

\noindent\emph{Declaration of AI use.}
During the preparation of this manuscript, the authors used OpenAI's ChatGPT to assist with literature searches and English-language
editing. All AI-assisted
language was reviewed and revised by the authors, who take full
responsibility for the content and arguments of the
manuscript.
\bibliographystyle{amsalpha}
\bibliography{references}

@article{BenNaimKrapivsky2003,
  author  = {Ben-Naim, Eli and Krapivsky, P. L.},
  title   = {Exchange-Driven Growth},
  journal = {Physical Review E},
  volume  = {68},
  number  = {3},
  pages   = {031104},
  year    = {2003},
  doi     = {10.1103/PhysRevE.68.031104}
}

@article{SiGiri2025,
  author  = {Si, Saroj and Giri, Ankik Kumar},
  title   = {Existence and Non-Existence for Exchange-Driven Growth
             Model},
  journal = {Nonlinearity},
  volume  = {38},
  number  = {12},
  pages   = {125014},
  year    = {2025},
  doi     = {10.1088/1361-6544/ae277b}
}

@article{DegondLiu2025,
  author  = {Degond, Pierre and Liu, Jian-Guo},
  title   = {Binary Particle Collisions with Mass Exchange},
  journal = {Journal of Statistical Physics},
  volume  = {192},
  number  = {2},
  pages   = {27},
  year    = {2025},
  doi     = {10.1007/s10955-025-03406-z}
}

@misc{LuoLiu2026,
  author        = {Luo, Siwei and Liu, Jian-Guo},
  title         = {On the Spatially Homogeneous {Boltzmann} Equation
                   with Mass Exchange},
  year          = {2026},
  eprint        = {2607.20684},
  archivePrefix = {arXiv},
  primaryClass  = {math.AP},
  note          = {Preprint, arXiv:2607.20684},
  doi           = {10.48550/arXiv.2607.20684}
}

@article{Flory1941,
  author  = {Flory, Paul J.},
  title   = {Molecular Size Distribution in Three Dimensional
             Polymers. {I}. Gelation},
  journal = {Journal of the American Chemical Society},
  volume  = {63},
  number  = {11},
  pages   = {3083--3090},
  year    = {1941},
  doi     = {10.1021/ja01856a061}
}

@book{Friedlander2000,
  author    = {Friedlander, Sheldon K.},
  title     = {Smoke, Dust, and Haze: Fundamentals of Aerosol Dynamics},
  edition   = {2},
  publisher = {Oxford University Press},
  address   = {New York},
  year      = {2000},
  isbn      = {9780195129991}
}

@article{Aldous1999,
  author  = {Aldous, David J.},
  title   = {Deterministic and Stochastic Models for Coalescence
             (Aggregation and Coagulation): A Review of the Mean-Field
             Theory for Probabilists},
  journal = {Bernoulli},
  volume  = {5},
  number  = {1},
  pages   = {3--48},
  year    = {1999},
  doi     = {10.2307/3318611}
}

@article{Leyvraz2003,
  author  = {Leyvraz, Fran{\c{c}}ois},
  title   = {Scaling Theory and Exactly Solved Models in the Kinetics
             of Irreversible Aggregation},
  journal = {Physics Reports},
  volume  = {383},
  number  = {2--3},
  pages   = {95--212},
  year    = {2003},
  doi     = {10.1016/S0370-1573(03)00241-2}
}

@book{BanasiakLambLaurencot2019,
  author    = {Banasiak, Jacek and Lamb, Wilson and
               Lauren{\c{c}}ot, Philippe},
  title     = {Analytic Methods for Coagulation--Fragmentation Models,
               Volume I},
  series    = {Monographs and Research Notes in Mathematics},
  publisher = {Chapman and Hall/CRC},
  address   = {Boca Raton},
  year      = {2019},
  isbn      = {9781498772655},
  doi       = {10.1201/9781315154428}
}

@article{EscobedoMischlerPerthame2002,
  author  = {Escobedo, Miguel and Mischler, St{\'e}phane and
             Perthame, Beno{\^\i}t},
  title   = {Gelation in Coagulation and Fragmentation Models},
  journal = {Communications in Mathematical Physics},
  volume  = {231},
  number  = {1},
  pages   = {157--188},
  year    = {2002},
  doi     = {10.1007/s00220-002-0680-9}
}

@article{EscobedoLaurencotMischlerPerthame2003,
  author  = {Escobedo, Miguel and Lauren{\c{c}}ot, Philippe and
             Mischler, St{\'e}phane and Perthame, Beno{\^\i}t},
  title   = {Gelation and Mass Conservation in
             Coagulation--Fragmentation Models},
  journal = {Journal of Differential Equations},
  volume  = {195},
  number  = {1},
  pages   = {143--174},
  year    = {2003},
  doi     = {10.1016/S0022-0396(03)00134-7}
}

@article{LaurencotVanRoessel2015,
  author  = {Lauren{\c{c}}ot, Philippe and van Roessel, Henry J.},
  title   = {Absence of Gelation and Self-Similar Behavior for a
             Coagulation--Fragmentation Equation},
  journal = {SIAM Journal on Mathematical Analysis},
  volume  = {47},
  number  = {3},
  pages   = {2355--2374},
  year    = {2015},
  doi     = {10.1137/140976236}
}

@book{BinghamGoldieTeugels1987,
  author    = {Bingham, Nicholas H. and Goldie, Charles M. and
               Teugels, Jozef L.},
  title     = {Regular Variation},
  series    = {Encyclopedia of Mathematics and its Applications},
  volume    = {27},
  publisher = {Cambridge University Press},
  address   = {Cambridge},
  year      = {1987},
  doi       = {10.1017/CBO9780511721434}
}

@article{AshgrizPoo1990,
  author  = {Ashgriz, Nasser and Poo, J. Y.},
  title   = {Coalescence and Separation in Binary Collisions of Liquid
             Drops},
  journal = {Journal of Fluid Mechanics},
  volume  = {221},
  pages   = {183--204},
  year    = {1990},
  doi     = {10.1017/S0022112090003536}
}

@article{AnidjarTambourGreenberg1995,
  author  = {Anidjar, F. and Tambour, Y. and Greenberg, J. B.},
  title   = {Mass Exchange between Droplets during Head-On Collisions
             of Multisize Sprays},
  journal = {International Journal of Heat and Mass Transfer},
  volume  = {38},
  number  = {18},
  pages   = {3369--3383},
  year    = {1995},
  doi     = {10.1016/0017-9310(95)00091-M}
}

@article{FilippovZuritaRosner2000,
  author  = {Filippov, Andrey V. and Zurita, Mauricio and
             Rosner, Daniel E.},
  title   = {Fractal-Like Aggregates: Relation between Morphology and
             Physical Properties},
  journal = {Journal of Colloid and Interface Science},
  volume  = {229},
  number  = {1},
  pages   = {261--273},
  year    = {2000},
  doi     = {10.1006/jcis.2000.7027}
}
\end{document}